\documentclass[a4paper,11pt,oneside,reqno]{amsart}
\usepackage[colorlinks,citecolor=blue,urlcolor=blue]{hyperref}
\usepackage[ansinew]{inputenc}
\usepackage{xcolor}
\usepackage{amsfonts}
\usepackage{amssymb, bbm, dsfont, mathpazo}
\usepackage{mathrsfs}
\usepackage{amsmath,amsthm,amsfonts,amssymb}
\usepackage[left=2.25cm,right=2.00cm, top=3cm, bottom=3cm]{geometry}
\usepackage[english]{babel}
\usepackage{enumerate}
\usepackage{soul,xcolor}

\newcommand{\embed}{\hookrightarrow }
\usepackage{fouridx}

\newcommand{\lb}{\langle}
\newcommand{\rb}{\rangle}

\numberwithin{equation}{section}

\newcommand{\me}{\mathbb{E}}
\newcommand\del[1]{}
\newcommand\dela[1]{}

\newcommand{\rA}{\mathrm{A}}
\newcommand{\rrA}{\mathrm{A}_1}
\newcommand{\hrrA}{\hat{\mathrm{A}}_1}
\newcommand{\trrA}{\tilde{\mathrm{A}}_1}

\newcommand{\rK}{\mathrm{K}}

\newcommand{\divv}{\mathrm{div }\;}

\newcommand{\err}{\mathbb{R}}
\newcommand{\ku}{\mathrm{K}}

\newcommand\red[1]{{\color{red}#1}}

\newcommand\toup{\nearrow}

\newcommand{\ps}{\Psi^\prime}
\newcommand{\pp}{\Psi^{\prime\prime}}
\newcommand{\fp}{f^\prime}
\newcommand{\fpp}{f^{\prime\prime}}
\renewcommand{\v}{\bv}
\newcommand{\W}{W_1}
\renewcommand{\a}{\Delta}

\newcommand{\db}{\bar{\d}}
\newcommand{\A}{\mathbf{A}}
\newcommand{\f}{\mathbf{F}}

\newcommand{\y}{\mathbf{y}}
\renewcommand{\d}{\mathbf{n}}
\newcommand{\MO}{\mathcal{O}}
\DeclareMathOperator{\tr}{tr}
\newcommand{\w}{\mathbf{w}}
\newcommand{\vm}{\mathbf{v}_m}
\newcommand{\vmt}{\vm(t)}
\newcommand{\bv}{\mathbf{v}}
\newcommand{\dm}{\mathbf{n}_m}
\newcommand{\dmt}{\dm(t)}
\newcommand{\bd}{\mathbf{n}}
\newcommand{\prm}{\pi_m}
\newcommand{\qm}{\hat{\pi}_m}
\newcommand{\el}{\mathbb{L}}
\newcommand{\ve}{\mathbb{V}}
\newcommand{\h}{\mathbb{H}}
\newcommand{\mo}{\mathcal{O}}
\newcommand{\bvt}{\bv(t)}
\newcommand{\bvs}{\bv(s)}
\newcommand{\bdt}{\bd(t)}
\newcommand{\bds}{\bd(s)}
\newcommand{\bn}{\boldsymbol{\nu}}
\newcommand{\bu}{\mathbf{u}}
\newcommand{\ma}{\mathrm{A}}
\newcommand{\eps}{\varepsilon}
\newcommand{\rve}{\rVert}
\newcommand{\lve}{\lVert}
\newcommand{\bvm}{\bar{\bv}_m}
\newcommand{\bdm}{\bar{\bd}_m}
\newcommand{\bvmk}{\bar{\bv}_{m_k}}
\newcommand{\bdmk}{\bar{\bd}_{m_k}}
\newcommand{\bwi}{\bar{W}_{1}^{\mk}}
\newcommand{\bwd}{\bar{W}_{2}^{\mk}}
\newcommand{\wi}{\bar{W}_{1}}
\newcommand{\wt}{\bar{W}_{2}}
\newcommand{\bx}{\mathbb{X}}

\newcommand{\EE}{\mathbb{E}^\prime}
\newcommand{\E}{\mathbb{E}}
\newcommand{\mk}{m_k}
\DeclareMathOperator{\lsp}{linspan} 
\theoremstyle{plain}
\newtheorem{condition}{Condition}[section]
\newtheorem{assum}[condition]{Assumption}

\newtheorem{lem}{Lemma}[section]
\newtheorem{thm}[lem]{Theorem}
\newtheorem{prop}[lem]{Proposition}
\newtheorem{cor}[lem]{Corollary}
\theoremstyle{definition}
\newtheorem{Def}[lem]{Definition}
\newtheorem{Rem}[lem]{Remark}
\title[Nematic liquid crystals driven by Wiener noise]{Some results on the penalised nematic liquid crystals driven by multiplicative noise}
\author[Z.~Brze{\'z}niak]{Zdzislaw Brze{\'z}niak}
\address{Department of Mathematics\\
The University of York\\
Heslington, York YO10 5DD, UK} \email{zdzislaw.brzezniak@york.ac.uk}
\author[E. Hausenblas]{Erika Hausenblas}

\address{Department of Mathematics and Information Technology, Montanuniversit\"at Leoben,
Fr. Josefstr. 18, 8700 Leoben, Austria} \email{erika.hausenblas@unileoben.ac.at}
\author[P. Razafimandimby]{Paul Razafimandimby}
\address{Department of Mathematics and Applied Mathematics, University of Pretoria, Corner of Lynnwood Road and Roper Street
Hatfield, Pretoria
0083, South Africa}
\email{paul.razafimandimby@up.ac.za}
\date{\today}

\begin{document}
\begin{abstract}
In this paper we prove several results related to the existence and uniqueness of solution to coupled highly  nonlinear stochastic partial differential equations (PDEs). These equations are motivated by the dynamics of  nematic liquid crystals under the influence of stochastic external forces.  Firstly, we prove the existence  of global weak solution (in sense of both stochastic analysis and PDEs). We show the pathwise uniqueness of the solution in 2D domain. Secondly, we establish the existence and uniqueness of local maximal solution which is strong in sense of both PDEs and stochastic analysis. In the 2D case, we show that this solution is global.
In contrast to several works in the deterministic setting we replace the Ginzburg-Landau function
	 $\mathds{1}_{\lvert \d\rvert \le 1}(\lvert \d\rvert^2-1)\d$ by a general polynomial $f(\d)$ and we give sufficient conditions on the polynomial $f$ for the aforementioned results to hold. \\	As a by-product of our investigation we present a general method based on fixed point argument to establish the existence and uniqueness of a local maximal solution of an abstract stochastic evolution equations with coefficients satisfying local Lipschitz condition involving the norms of two different Banach spaces. This general method can be used to treat several  stochastic hydrodynamical models such as Navier-Stokes, Magnetohydrodynamic (MHD) equations, and the $\alpha$-models of Navier-Stokes equations and their MHD counterparts.
\end{abstract}
\maketitle
\section{Introduction}
Nematic liquid crystal is a state of matter  that has
properties between amorphous liquid and crystalline solid.
Molecules of nematic liquid crystals are long and thin, and they
tend to align along a common axis. This preferred axis indicates
the orientations of the crystalline molecules, hence it is useful
to characterize its orientation with a vector field
$\mathbf{n}$ which is called the \textbf{director}. Since its
magnitude has no significance, we shall take $\mathbf{n}$ as a
unit vector. We refer to \cite{Chandrasekhar} and \cite{Gennes}
for a comprehensive treatment of the physics of liquid
crystals. To
model the dynamics of nematic liquid crystals most scientists use
the continuum theory developed by Ericksen \cite{Ericksen} and
Leslie \cite{Leslie}. From this theory F. Lin and C. Liu
\cite{Lin-Liu} derived the most basic and simplest form of
dynamical system describing the motion of nematic liquid crystals
filling a bounded region $\MO\subset \mathbb{R}^d, d=2,3$. This
system is given by
\begin{align}
 \bv_t+(\bv\cdot\nabla)\bv-\mu \Delta \bv+\nabla p=&-\lambda \nabla\cdot(\nabla \bd \odot  \nabla \bd), \text{ in } (0,T]\times \MO \label{erick-leslie1}\\
\divv \bv=&0,\text{ in } (0,T]\times \MO\\
\bd_t+(\bv\cdot\nabla)\bd=&\gamma(\Delta \bd+|\nabla \bd|^2\bd), \text{ in } (0,T]\times \MO\label{erick-leslie2}\\
\bd(0)=&\bd_0, \text{ and } \bv(0)= \bv_0 \text{ in } \MO
\\
|\bd|^2=&1,\text{ on } (0,T]\times \MO. \label{erick-leslie4}
\end{align}
  Here $p$ represents the pressure of the fluid and $\bv$ its velocity. By the symbol $\nabla \bd \odot
\nabla \bd$ we mean  a  $d\times d$-matrix with entries
defined by
\begin{equation*}
[\nabla \bd \odot \nabla \bd]_{i,j}=\sum_{k=1}^n \frac{\partial
\bd^k}{\partial x_i}\frac{\partial \bd^k}{\partial x_j},\;\;
\mbox{ } i,j=1,\dots, d.
\end{equation*}
 In the present work we assume that the
boundary of $\MO$ is smooth and that
\begin{equation}\label{BC}
\bv=0 \text{ and } \frac{\partial \bd}{\partial \bn}=0 \text{ on }
\partial \MO,
\end{equation}
where the vector field $\bn$  is the unit outward unit  normal to $\partial \MO$, i.e.  at each
point $x$ of $\MO$, $\bn(x)$ is perpendicular to the tangent space $T_x\partial \MO$, of length $1$ and facing outside of $\MO$.

Although the system \eqref{erick-leslie1}-\eqref{BC} is the most
basic and simplest form of equations from the Ericksen-Leslie
continuum theory, it retains the most physical significance of the
nematic liquid crystals. Moreover it offers many interesting
mathematical problems. In fact, on one hand, two of the  main mathematical difficulties related to the system \eqref{erick-leslie1}-\eqref{BC} are non-parabolicity of equation \eqref{erick-leslie2}  and high nonlinearity of the term $\divv\sigma^E= -\divv( \nabla \bd \odot \nabla \bd) $. The non-parabolicity  follows from the fact that 
 	\begin{equation} 
 \Delta \bd+|\nabla \bd|^2\bd=\bd \times (\Delta \bd\times \bd),
 	\end{equation}
 	so that the linear term $\Delta \bd$ in \eqref{erick-leslie2} is only a tangential part of the full Laplacian.
 The term $\divv( \nabla \bd \odot \nabla \bd) $ makes the problem \eqref{erick-leslie1}-\eqref{BC} a fully nonlinear and constrained system of PDEs coupled via a quadratic gradient nonlinearity. In addition to the above difficulties, we will certainly have to deal with new challenging and difficult problems due to the introduction of stochastic perturbations. On the other hand, a number of
 	 challenging questions
 	 about the solutions to Navier-Stokes equations (NSEs) and Geometric Heat equation (GHE) are still open.

 In 1995, F. Lin and C. Liu
\cite{Lin-Liu} proposed an approximation of the system
\eqref{erick-leslie1}-\eqref{BC} to relax the constraint
$|\bd|^2=1$ and  the gradient nonlinearity $|\nabla \bd |^2\bd$.
More precisely, they studied the following system of equations
\begin{align}
 \bv_t+(\bv\cdot\nabla)\bv-\mu\Delta \bv+\nabla p=&-\lambda \nabla\cdot(\nabla \bd \odot \nabla \bd), \text{ in } (0,T]\times \MO \label{ginz-leslie1}\\
\divv \bv=&0, \text{ in } [0,T]\times \MO \\
\bd(0)=&\bd_0 \text{ and }\bv(0)=\bv_0 \text{ in } \MO,\\
\bd_t+(\bv\cdot\nabla)\bd=&\gamma\left(\Delta \bd-\frac1{\eps^2}
(|\bd|^2-1)\bd\right)\text{ in } (0,T]\times \MO,\label{ginz-leslie3}
\end{align}
where $\eps>0$ is an arbitrary constant.

Problem \eqref{ginz-leslie1}-\eqref{ginz-leslie3} with boundary conditions \eqref{BC}
is much simpler than \eqref{erick-leslie1}-\eqref{erick-leslie4}
with \eqref{BC}, but it is still a difficult and interesting
problem. Since the pioneering work \cite{Lin-Liu} the systems
\eqref{ginz-leslie1}-\eqref{ginz-leslie3} and
\eqref{erick-leslie1}-\eqref{erick-leslie4} have been the subject
of intensive mathematical studies. We refer, among others, to
\cite{Rojas-Medar, Hong,Lin-Liu,
Lin-Liu2,Lin-Wang, Lin-Lin-Wang, Shkoller,Dai+Schonbeck_2014} and
references therein for the relevant results.

In this paper we are interested in the mathematical analysis of a
stochastic version of problem \eqref{ginz-leslie1}-\eqref{ginz-leslie3}.
Basically,
 we will investigate a system of stochastic evolution equations which is obtained by introducing
 appropriate noise term in \eqref{erick-leslie1}-\eqref{erick-leslie4}. In contrast to the unpublished manuscript \cite{BHP13} we replace the bounded Ginzburg-Landau function
 $\mathds{1}_{\lvert \d\rvert \le 1}(\lvert \d\rvert^2-1)\d$ in the coupled system by a general polynomial function $f(\d)$. More precisely, we set $\mu=\lambda=\gamma=1$ and  we consider cylindrical Wiener processes $W_1$ on a separable  Hilbert space $\rK_1$  and
 a standard real-valued Brownian motion $W_2$. We assume that $W_1$ and $W_2$ are independent.
 We consider the problem
 \begin{align}
 d\bvt+\biggl[(\bvt\cdot\nabla)\bvt-\Delta \bvt+\nabla p\biggr]dt= & -\nabla\cdot(\nabla \bdt \odot \nabla \bdt)dt+ S(\bvt) dW_1(t),\label{eqn-SLQE-v} \\
\divv \bvt= & 0,\label{eqn-SLQE-div} \\
 d\bdt+(\bvt\cdot\nabla)\bdt dt = & \biggl[\Delta
\bdt-f(\d)\biggr]+(\bdt\times h)\circ
dW_2(t), \label{eqn-SLQE-d}
\end{align}
where $(\bdt\times h)\circ dW_2(t)$ is understood in
the Stratonovich sense and $f$ is a polynomial function and the above system holds in $\MO_T:=(0,T]\times \MO$. We equip the system with the boundary condition  \eqref{BC}
and the initial condition
\begin{equation*}
 \bv(0)=\bv_0 \text{ and } \bd(0)=\bd_0 ,
\end{equation*}
where $\bv_0$ and $\bd_0$ are given mappings defined on $\MO$ . We will give more details about the polynomial $f$ later on.

Our work is motivated by the importance of external perturbation
on the dynamics of the director field $\bd$. Indeed, an essential
property of nematic liquid crystals is that its director field
$\mathbf{n}$ can be easily distorted. However, it can also be
 aligned to form a specific pattern under some external perturbations.
 This pattern formation occurs when a threshold value of the external perturbations is attained; this is the so called Fr\'eedericksz transition.
 Random external perturbations change a little bit the threshold value for the Fr\'eedericksz transition.
 For example, it has been found that with the fluctuation of the magnetic field the relaxation time of an unstable state diminishes, \textit{i.e.},  the time for a noisy system to leave an unstable state is much shorter than the unperturbed system. For these results we refer, among others, to \cite{Horsthemke+Lefever-1984, San
Miguel-1985,FS+MSanM} and references therein. In all of these works the
effect of the hydrodynamic flow has been neglected. However, it is
pointed out in \cite[Chapter 5]{Gennes}
 that the fluid flow disturbs the alignment and conversely a change in the alignment will induce a
 flow in the nematic liquid crystal.
Hence for a full understanding of the effect of
 fluctuating magnetic field on
 the behavior of the liquid crystals one needs to take into account the dynamics of $\d$ and $\v$.
 To initiate this kind of investigation we propose a mathematical study of \eqref{eqn-SLQE-v}-\eqref{eqn-SLQE-d} which basically
 describes an approximation of the
 system governing the nematic liquid crystals under the influence of fluctuating external forces.

 In the present paper we prove several results that are basically the stochastic counterparts of the ones obtained by Lin and Liu in \cite{Lin-Liu}. Our results can be described as follows.
 \begin{enumerate}[(i)]
\item \label{i} In section \ref{SLC-sect3} we establish  the existence of global martingale solutions (weak in the PDEs sense).
To prove this result we first find a suitable  finite dimensional Galerkin approximation   of system \eqref{eqn-SLQE-v}-\eqref{eqn-SLQE-d}
 which can be solved locally in time. Our choice of the approximation yields the global existence of the  approximating solutions $(\vm,\dm)$.  For this purpose we derive several important global a priori estimates in higher order Sobolev spaces involving the following two energy functionals
 \begin{equation*}
  \mathcal{E}_1(\d,t):= \lve \bd(t)\rve^q+q \int_0^t \lve \bds \rve^{q-2} \lve \nabla \bds\rve^2 ds
  + q \int_0^t \lve \bds \rve^{q-2} \lve \bds\rve^{2N+2}_{\mathbb{L}^{2N+2}} ds
 \end{equation*}
 and
 \begin{equation*}
\begin{split}
\mathcal{E}_2(\v,\d,t):= \lve\bv(t)\rve^{2}+\ell \lve \d(t) \rve^2+\lve \nabla \bd(t)\rve^{2}+\int_\MO F(\d(t,x) dx \\+ \left(\int_0^{t}
\lve \nabla
\bv(s)\rve^2+ \lve \Delta \bd(s)-f(\bd(s))\rve^2\right)ds.
\end{split}
\end{equation*}
 Here $F(\cdot)$ is the antiderivative of $f$ such that $F(0)=0$.
These global a priori estimates, the proofs of which are non-trivial and require long and tedious calculation, are very crucial for the
 proof of the tightness of the distributions of the Galerkin solution $(\vm,\dm)$ in appropriate topological spaces such as
 $L^2(0,T;\h)\times L^2(0,T;\h^1(\MO))$.
 This
  compactness result will enable us to construct a new probability space on which  we also find a new sequence  $(\bvm,\bdm, \bar{W}_1^m,\bar{W}_2^m)$
   of solutions of the Galerkin
  equations. This new sequence is proved to converge to a system
  $(\v,\d,\bar{W}_1, \bar{W}_2)$ which along with the new
  probability space will form our martingale solution. To close the first part of our results we show that the weak solution is pathwise unique in the 2-D case.

 \item We establish several results about the strong solution to the problem \eqref{eqn-SLQE-v}-\eqref{eqn-SLQE-d} in Section \ref{SLC-Sect4}. We mainly prove the existence of a local and maximal strong solution which is
 understood in the sense of stochastic calculus and PDEs.  In the case $d=2$ we prove the non-explosion of the maximal solution by a method based on a choice of an appropriate energy functional
  (or Lyapunov function) and an idea introduced by Schmalfu\ss in \cite{Bjorn} to prove the pathwise uniqueness of solution of stochastic Navier-Stokes equations. The estimates for the energy functional are of the form
  \begin{eqnarray*}
\mathcal{E}_3&=& \E \biggl[\Phi(t\wedge \tau_k)\left(\lve \nabla
\v(t\wedge \tau_k)\rve^2+\Psi(\d(t\wedge \tau_k))\right)\biggr]\\
&+& \E \int_0^{t\wedge \tau_k}
\Phi(s)\Big(2\lve \Delta\v(s)
\rve^2+ \lve \nabla (\Delta \d(s) -f(\d(s)) )\rve^2 \Big)ds,
\end{eqnarray*}
where $\Phi(\cdot)$ is a certain positive function to be defined later and $\{\tau_k; k=1,2, \ldots\}$ is a sequence of stopping times, see Section \ref{SLC-Sect4} for more details.
Schmalfu\ss\, trick has so far been used to prove pathwise uniqueness of solution of stochastic PDEs, see for instance \cite{Bjorn}, \cite{ZB+EM}, \cite{ZB+EH+JZ} and references therein. Our novelty is that we extend the use of this idea to prove the global existence of strong solution of the Ginzburg-Landau approximation of nematic liquid crystals. In particular, our proof gives another proof of the global existence of 2D stochastic Navier-Stokes equations with multiplicative for initial data with finite enstrophy. Thus, our paper can be seen also as a generalization of the results for the existence and uniqueness of local, maximal and global solutions of strong solutions of stochastic Navier-Stokes proved in \cite{Nathan1}, \cite{Nathan2} and \cite{Mikulevicius}.

 \item In Section \ref{ABST-STRONG} we use a fixed point argument and prove several general results on an abstract nonlinear stochastic evolution
      equations (SEE) with coefficients satisfying local Lipschitz condition involving the norms of two different Banach spaces. We mainly establish the existence  and uniqueness of a local and a
          maximal solution to this SEE and derive a lower estimate for the life-span of the solution. These general results are of independent interest and  that is why they are the subject of a separate section. The existence and uniqueness result in Section \ref{SLC-Sect4} are corollaries of the general results in the last section. In fact, in Section \ref{SLC-Sect4} we consider some functional
      spaces with more regularity than in the case of weak and
      martingale solution and show that on these spaces the problem \eqref{eqn-SLQE-v}-\eqref{eqn-SLQE-d} with \eqref{BC} falls within the general framework in Section \ref{ABST-STRONG}.  We are also strongly convinced that with these general results it is possible, although has not been done in detail, to prove the existence of strong solution of several stochastic hydrodynamical models such as the NSEs, Magnetohydrodynamic (MHD) equations, $\alpha$-models for Navier-Stokes (Lagrangian Navier-Stokes $\alpha$-models, Leray-$\alpha$ models) and the MHD equations associated to them.

      \item We prove a maximum principle type theorem in Section \ref{MAX-PRIN-SEC}. More precisely, if we  consider $f(\d)=\mathds{1}_{\lvert d\rvert \le 1}(\lvert \d\rvert^2-1)\d$ instead and if the initial condition $\d_0$ satisfies $\lvert \d_0\rvert\le 1$ almost everywhere, then the solution $\d$ also remains in the unit ball almost everywhere. In contrast to the deterministic case, this result does not follow in straightforward way from well-known results. Here the method of proofs are based on the blending of ideas from \cite{Rojas-Medar2} and \cite{Matoussi}.
  \end{enumerate}
  To the best of our knowledge our work is the first mathematical work
 which studies the existence and uniqueness of weak martingale of the \eqref{eqn-SLQE-v}-\eqref{eqn-SLQE-d}.
 Under the assumption that
 $f(\cdot)$ is a bounded function, the authors proved in the unpublished manuscript \cite{BHP13} that
 the system \eqref{eqn-SLQE-v}-\eqref{eqn-SLQE-d}  has a maximal strong solution which is global for the 2D case. Therefore, the present article is a generalization of
  \cite{BHP13} in the sense that we allow $f(\cdot)$ to be a general polynomial function.  Note that some of the arguments from the last part of the present paper have already been used in the paper \cite{BHR-2014} which studied the strong solution of some stochastic hydrodynamic equations driven by L\'evy noise.

 The organization of the present article is as follows.  In Section \ref{sec-spaces} we introduce the notations that are frequently used throughout this
 paper. In the same section we also state and prove some useful lemmata. By using the scheme we outlined in item \eqref{i} above we show in  Section \ref{SLC-sect3} that
 \eqref{eqn-SLQE-v}-\eqref{eqn-SLQE-d} admits a weak martingale solution which is pathwise unique in the two dimensional case.
 In Section \ref{SLC-Sect4} we prove the existence of a unique maximal strong
 solution to \eqref{eqn-SLQE-v}-\eqref{eqn-SLQE-d} which turns out to be a global solution when the space dimension is two. The existence and uniqueness
 of a maximal solution to the stochastic system \eqref{eqn-SLQE-v}-\eqref{eqn-SLQE-d} is a corollary of a general theorem given in  Section \ref{ABST-STRONG}.
 This result is obtained by making use of some fixed point argument and truncation method. In Section \ref{MAX-PRIN-SEC} a maximum principle type
    theorem is proved when $f(\d)=\mathds{1}_{\lvert \bd\rvert \le 1}(\lvert \d\rvert^2-1)\d$. In Appendix we prove several crucial
estimates about the solution
 of the stochastic equation for nematic liquid crystals. 

 \section{Functional spaces and Preparatory lemma}\label{sec-spaces}

 Let $d\in \{2,3\}$ and assume that $\MO \subset \mathbb{R}^d$ is a bounded domain with boundary $\partial \MO$ of class $\mathcal{C}^\infty$.
 For any $p\in [1,\infty)$ and $k\in \mathbb{N}$,  $\el^p(\MO)$ and
$\mathbb{W}^{k,p}(\MO)$ are the well-known Lebesgue and Sobolev
spaces, respectively, of $\mathbb{R}^d$-valued functions. The
corresponding spaces of scalar functions we will denote by
standard letter, e.g. ${W}^{k,p}(\MO)$.
 For $p=2$ we denote $\mathbb{W}^{k,2}(\MO)=\h^k$ and its norm are denoted by $\lve \bu
 \rve_k$. By $\h^1_0$ we mean the space of functions in $\h^1$
 that vanish on the boundary on $\MO$; $\h^1_0$ is a Hilbert space when endowed with the scalar product induced by that of $\h^1$.
The usual scalar product on $\el^2$ is denoted by $\langle
u,v\rangle$ for $u,v\in \el^2$ and its associated norm is denoted by $\lVert
u\rVert$, $u\in \el^2$. We also introduce the following spaces
 \begin{align*}
\mathcal{V}& =\left\{ \bu\in \mathcal{C}_{c}^{\infty }(\MO,\mathbb{R}^d)\,\,\text{such that}%
\,\,\nabla \cdot \bu=0\right\} \\
\mathbb{V}& =\,\,\text{closure of $\mathcal{V}$ in }\,\,\mathbb{H}_0^{1}(\MO) \\
\mathbb{H}& =\,\,\text{closure of $\mathcal{V}$ in
}\,\,\mathbb{L}^{2}(\MO).
\end{align*}
We endow $\h$ with the scalar product and norm of $\el^2$. As
usual we equip the space $\ve$ with the the scalar product
$\langle \nabla \bu, \nabla \bv\rangle$ which, owing to the Poincare\'e inequality, is equivalent to the
$\h^1(\MO)$-scalar product.

Let $\Pi: \el^2 \rightarrow \h$ be the Helmholtz-Leray projection
from $\el^2$ onto $\h$. We denote by $\rA=-\Pi\Delta$ the
Stokes operator with domain $D(\rA)$.  It is well-known (see for e.g. \cite[Chapter I, Section 2.6]{Temam}) that there
exists an orthonormal basis $\{\varphi_i; i=1, 2, 3, \ldots\}$  of $\h$ consisting of the
eigenfunctions of the Stokes operator $\rA$.

From \cite[Proposition 1.24]{Ouhabaz} we can define a self-adjoint
operator $A_1: \h^1\rightarrow \left(\h^1\right)^\ast$ by
\begin{align}\label{ST0}
\langle \rA_1 \bu, \w \rangle =&a(\bu, \w)=\int_\MO \nabla \bu \cdot \nabla \w \,dx, \;\;\; \bu, \,\, \w\in \h^1.
\end{align}
 The Neumann Laplacian acting on $\mathbb{R}^n$-valued
function will be  denoted by $\rrA$, that is,
\begin{equation}
\begin{split}
D(\rrA)&:=\biggl\{\bu\in \h^2: \frac{\partial \bu}{\partial \bn}=0 \text{ on } \partial \MO\biggr\},\\
\rrA\bu&:=-\sum_{i=1}^n \frac{\partial^2 \bu}{\partial
x_i^2},\;\;\; \bu \in D(\rrA).
\end{split}
\end{equation}
It can be shown that $D(\rrA)=\{ \bu \in \h^1 :
\rA_1 \bu \in \el^2\}$ and $\rrA \bu =A_1\bu$ for  $\bu\in
D(\rrA)$.
It can also be shown, see e.g. \cite[Theorem 5.31]{Haroske}, that $\hrrA=I+\rrA$ is a
definite positive and self-adjoint operator in the Hilbert space
$\mathbb{L}^2= \mathbb{L}^2(\MO)$ with compact resolvent. In particular, there exists an ONB $(\phi_k)_{k=1}^\infty$ of $\mathbb{L}^2$ and  an
a  increasing sequence  $\big(\lambda_k\big)_{k=1}^\infty $ with $\lambda_1=0$ and
$\lambda_k\toup \infty$ as $k\toup \infty$ (the
eigenvalues of the Neumann Laplacian $\rrA$)  such that $\rrA\phi_j=\lambda_j \phi_j$ for any $i\in\mathbb{N}$.

 For any  $\alpha \in [-\frac12,\infty)$ we
denote by  $\bx_\alpha=D(\hrrA^{\frac12+\alpha})$,  the domain
of the fractional power operator $\hrrA^{\frac12+\alpha}$. We have the
following characterization of the spaces $\bx_\alpha$,
\begin{equation}\label{frac-space}
\bx_{\alpha}=\{\bu= \sum_{k\in \mathbb{N}} u_k \phi_k: \sum_k
(1+\lambda_k)^{1+2\alpha} \lvert u_k\rvert^2<\infty\}.
\end{equation}
It can be shown that  $\bx_\alpha\subset \h^{1+2\alpha}$, for all $\alpha \geq 0$  and
$\bx:=\bx_{0}=\h^1$.

Similarly, for $\beta \in [0,\infty)$,  we denote by $\ve_\beta$ the Hilbert space $D(\rA^\beta)$ endowed with the graph inner product.
The Hilbert space
$\ve_\beta=D(\rA^\beta)$ for $\beta\in (-\infty, 0)$ can be defined by standard extrapolation methods.  In particular, the space $D(\rA^{-\beta})$ is the dual
of $\ve_\beta$ for
$\beta\ge0$. Moreover,  for every $\beta,\delta \in \mathbb{R}$ the map $\rA^\delta$ is a linear isomorphism between $\ve_\beta$ and $\ve_{\beta-\delta}$. It is also well-known that $\ve_\frac12 = \ve.$

Throughout this paper $\mathbf{B}^\ast$ denotes the dual space of
a Banach space $\mathbf{B}$. We denote by $\langle \Psi,
\mathbf{b}\rangle$ the value of $\Psi\in \mathbf{B}^\ast$ on
$\mathbf{b}\in \mathbf{B}$.

  Given two Hilbert spaces $K$ and $H$,  we denote by $\mathcal{L}(K,H)$ and $\mathcal{T}_2(K,H)$ the space of bounded linear operators and the Hilbert space
 of all Hilbert-Schmidt operators from $K$ to $H$, respectively. For $K=H$ we just write $\mathcal{L}(K)$ instead of $\mathcal{L}(K,K)$.


Now, we introduce a trilinear form
\begin{equation*}
b(\bu,\bv,\w)=\sum_{i,j=1}^d \int_\mathcal{O}\bu^{(i)}\frac{\partial
\bv^{(j)}}{\partial x_i}\w^{(j)} dx,\,\, \bu\in \el^p ,\bv\in
\mathbb{W}^{1,q}, \text{ and } \w\in \el^r,
\end{equation*}
with numbers $p,q,r \in [1,\infty]$ satisfying $$\frac 1p +\frac 1q +\frac 1r\le 1.$$

The map $b$ is the trilinear form used in the mathematical
analysis of the Navier-Stokes equations, see for instance
\cite[Chapter II, Section 1.2]{Temam}. It is well known (see for e.g. \cite[Chapter II, Section 1.2]{Temam}) that
 one can define a bilinear map $\tilde{B}$
defined on $\h^1 \times \h^1$ with values in $(\h^1)^\ast$
 \del{(is the dual of $\h^1$)}
 such that
 \begin{equation}\label{DEF-B2}
 \langle
\tilde{B}(\bu,\bv),\w\rangle=b(\bu,\bv,\w)\,\, \text{ for any }
\bu, \,\, \bv, \,\,\w\in \h^1.
 \end{equation}
We can also a define bilinear  map $B$ from $\ve \times \ve$ with
values in $\ve^\ast$ such that
\begin{equation}\label{DEF-B1}
\langle
B(\bu,\bv),\w\rangle=b(\bu,\bv,\w)\text{ for }  \w\in \ve,\text{ and } \bu, \bv\in \h^1.
\end{equation}
Well-known properties of $B$ and $\tilde{B}$ will be given in Appendix.

 Let
$\mathfrak{m}$ be the trilinear form defined by
\begin{equation}\label{INT-md}
 \mathfrak{m}(\d_1, \d_2,\bu)= -\sum_{i,j,k=1}^d \int_\MO
\partial_{x_i}\bd_1^{(k)} \partial_{x_j}\bd_2^{(k)} \partial_{x_j}\bu^{(i)} \,dx
\end{equation}
for any $\d_1\in \W^{1,p}$, $\d_2\in \W^{1,q}$ and $\bu\in
\W^{1,r}$ with $r,\,p,\, q\in(1,\infty)$ satisfying  $$ \frac 1p
+\frac 1q+\frac 1r\le 1.$$ \del{where the RHS in the above
identity should be understood as
\begin{equation}
\langle \Pi(\nabla \cdot[\nabla \bd_1 \odot \nabla \bd_2]), \bu
\rangle,\del{ M_{\bd_1, \d_2}(\bu)}=.
\end{equation}}
Here $\partial_{x_i}=\frac{\partial}{\partial x_i} $ and $\phi^i$
is the $i$-th entry of any vector-valued $\phi$. Since $n\le 4$,
the integral in \eqref{INT-md} is well defined for  When $\bd_1,
\d_2 \in \h^2$ and $\bu\in \ve$.\del{  When $\d_1=\d_2=\d$ we just
write $$M_{\d,\d}=M_\d.$$ } We have the following lemma.
\begin{lem}\label{DET-EST-M}
There exist a constant $C>0$ such that
\begin{equation}\label{IM3-0}
\lvert \mathfrak{m}(\bd_1, \d_2, \bu)\rvert \le C \lve \nabla \d_1
\rve^{1-\frac d4} \lve \Delta \d_1\rve^{\frac d4} \lve \nabla
\d_2\rve^{1-\frac d4} \lve \Delta \d_2\lve^{\frac d4} \lve \nabla
\bu\rve,
\end{equation}
for any $\bd_1, \d_2\in \h^2$ and $\bu\in \ve$.
\end{lem}

\begin{proof}[Proof of Lemma \ref{DET-EST-M}]
\del{ Let $\phi \in \ve$. Since $\phi$ is divergence-free and zero
at the boundary, we get by integration-by-parts that
 \begin{equation*}
  \begin{split}
   \langle M(\bd), \phi\rangle=\langle \Pi\circ \nabla\cdot (\nabla \bd \odot \nabla \bd), \phi\rangle,\\
   =\langle  (\nabla \bd \odot \nabla \bd), \nabla \phi\rangle.
  \end{split}
 \end{equation*}}
From \eqref{INT-md} and H\"older's inequality we derive that
\begin{equation*}
 \begin{split}
  \lvert \mathfrak{m}(\bd_1,\d_2,\bu)\rvert \le \int_Q |\nabla \bd_1||\nabla \bd_2| |\nabla \bu| dx.
 \end{split}
\end{equation*}
The above integral is well-defined since \del{for $\d_i,
\,i=1,\,2$} $\nabla \d_i\in \el^{\frac{2d}{d-2}}, \,i=1,\, 2$,
$\nabla \bu \in \el^2 $ and  $ \frac{d-2}{d}+\frac 12 \le 1$ for $
d\le 4.$ When $d=2$ we replace $2d/(d-2)$ by any $q\in
[4,\infty)$. Note that for $d\le 4$ we have $|\nabla \bd_i| \in
\el^{4}$, $i=1,2$. Hence
\begin{equation*}
 \lvert \mathfrak{m}(\bd_1, \d_2, \bu)\rvert \le C \lve \nabla \bd_1\rve_{\el^4} \lve \nabla \bd_2\rve_{\el^4} \lve \nabla \bu\rve.
\end{equation*}
This last estimate and Gagliardo-Nirenberg's inequality
\eqref{GAG-l4} \del{in the form
$$\lve \psi \rve_{\el^4}\le c \lve \psi \rve^{1-a} \lve \nabla \psi\rve^a, \text{   } a=\frac d4 ,$$
} lead us to
\begin{equation}\label{Gagl-Nir}
 \lvert \mathfrak{m}(\d_1,\d_2, \bu)\rvert\le C \lve \nabla \d_1
\rve^{1-\frac d 4} \lve \Delta \d_1\rve^{\frac d4} \lve \nabla
\d_2\rve^{1-\frac d4} \lve \Delta \d_2\lve^{\frac d4} \lve \nabla
\bu\rve.
\end{equation}
This concludes the proof of our claim.
\end{proof}
The above result tells us that, for any $\bu\in \mathbb{V}$,
$\mathfrak{m}(\bd_1, \d_2,\bu)$ is an element of $\mathcal{L}(\ve,
\mathbb{R})$ whenever $\bd_1, \d_2\in \h^2$. Now, we state and prove the following proposition.
 \begin{prop}\label{LEM-M}
There exists a bilinear operator $M$ defined on $\h^2\times \h^2$
taking values on $\ve^\ast$ \del{
\[M: \h^2 \times \h^2 \ni (\bd_1,\bd_2) \mapsto M(\d_1,\d_2)\in \ve^\ast, \del{
:= \Pi(\nabla \cdot[\nabla \bd_1 \odot \nabla \bd_2]) \in  \ve
^\ast}\]} such that for any $\d_1, \, \d_2 \in \h^2$
\begin{align}\label{def-Md}
\langle M(\bd_1,\d_2), \bu\rangle= \mathfrak{m}(\d_1,\d_2,\bu)
\del{-\sum_{i,j,k}\int_\MO
\partial_{x_i}\bd_1^k \partial_{x_j}\bd_2^k \partial_{x_j}\bu^i dx,} \;\; \bu \in \ve.
\end{align}
Furthermore, there exists a constant $C>0$ such that
\begin{equation}\label{IM3-1}
  \lve M(\bd_1,\d_2)\rve_{\ve^\ast} \le C \lve \nabla \d_1
\rve^{1-\frac d4} \lve \Delta \d_1\rve^{\frac d4} \lve \nabla
\d_2\rve^{1-\frac d4} \lve \Delta \d_2\lve^{\frac d4},
 \end{equation}
for any $\bd_1, \d_2\in \h^2$.
We also have the following identity
\begin{equation}
\langle \tilde{B}(\bv,\bd), \Delta \bd\rangle=\langle
M(\bd,\d), \bv\rangle, \text{ for any } \bv\in \ve, \bd \in \h^2.\label{G1-eq-Md}
\end{equation}

\end{prop}
\begin{proof}
The first part and \eqref{IM3-1} follow from Lemma \ref{DET-EST-M}.

To prove \eqref{G1-eq-Md} we first note that $\langle \tilde{B}(\bv,\bd_2), \Delta
\bd_1\rangle=b(\bv,\bd_2,\Delta\bd_1)$ is well-defined for any
$\bv\in \ve, \bd_1, \d_2 \in \h^2.$  Thus, taking into account
that $\bv$ is divergence free and vanishes on the boundary we can perform an
integration-by-parts and deduce that
\begin{align*}
 \langle B(\v,\d),\Delta \d)=& \int_\MO \v^{(i)} \frac{\partial \d^{(k)}}{\partial x_i}\frac{\partial^2 \d^{(k)}}{\partial x_l \partial x_l} dx\\
 =& -\int_\MO \frac{\partial \v^{(i)}}{\partial x_l} \frac{\partial \d^{(k)}}{\partial x_i} \frac{\partial \d^{(k)}}{\partial x_l} dx- \int_\MO \v^{(i)}
 \frac{\partial^2 \d^{(k)}}{\partial x_i\partial x_l} \frac{\partial \d^{(k)}}{\partial x_l} dx\\
 =&- \int_\MO \frac{\partial \v^{(i)}}{\partial x_l} \frac{\partial \d^{(k)}}{\partial x_i} \frac{\partial \d^{(k)}}{\partial x_l} dx -
 \frac12 \int_\MO \v^{(i)} \frac{\partial \lvert \nabla \d \rvert^2}{\partial x_i}dx\\
 =& -\int_\MO \frac{\partial \v^{(i)}}{\partial x_l} \frac{\partial \d^{(k)}}{\partial x_i} \frac{\partial \d^{(k)}}{\partial x_l} dx\\
 =& \mathfrak{m}(\d,\d,\v)=\langle M(\d,\d),\v\rangle.
\end{align*}
In the above chain of equalities summation over repeated indexes are enforced.
\end{proof}
\begin{Rem}
\begin{enumerate}
\item We have
$\nabla \cdot (\nabla \mathbf{f} \odot \nabla \mathbf{g})\in
\el^2$ for any $\mathbf{f}, \mathbf{g}\in \bx_{\frac 12}\cap\bx_{1}$. Thus, we infer that for any $\mathbf{f}, \mathbf{g}\in \bx_{\frac 12}\cap\bx_{1}$
$$M(\mathbf{f}, \mathbf{g})= \Pi[\nabla \cdot (\nabla \mathbf{f}
\odot \nabla \mathbf{g})].$$
\item In some places in this manuscript we use the following shorthand notation:
$$ B(\bu):= B(\bu,\bu) \text{ and } M(\bd):= M(\bd,\bd),$$ for any $\bu$ and $\bd$ such that the above quantities are meaningful.
\end{enumerate}

\end{Rem}

We assume that  $h\in\el^\infty$  is fixed. We can
define a linear bounded operator $G$ from $\el^2$ into itself by
$$ G: \el^2 \ni \bd \mapsto \bd\times h\in \el^2.$$ It is straightforward to check that  satisfies
$$ \lve G(\bd)\rve \le \lve h\rve_{\el^\infty} \lve \bd \rve.$$
 Let $(\Omega,
\mathcal{F}, \mathbb{P})$ be a complete probability space equipped
with a filtration $\mathbb{F}=\{\mathcal{F}_t: t\geq 0\}$
satisfying the usual condition.  Let
$W_2=(W_2(t))_{t\geq 0}$ be a standard
$\mathbb{R}$-valued Wiener process on $(\Omega,
\mathcal{F},\mathbb{F}, \mathbb{P})$. Let us also assume that $\rK_1$ is a separable Hilbert space and $W_1=(W_1(t))_{t\geq 0}$ is a $\rK_1$-cylindrical
Wiener process on $(\Omega,
\mathcal{F},\mathbb{F}, \mathbb{P})$.
  Throughout we assume
that $W_2$ and $W_1$ are  independent. Thus we can assume that $W=(W_1(t),W_2(t))$ is $\rK$-cylindrical
Wiener process on $(\Omega,
\mathcal{F},\mathbb{F}, \mathbb{P})$, where
\[ \rK=\rK_1\times \mathbb{R}.\]
\begin{Rem}\label{Wiener}
If $\ku_2$ is a Hilbert space such that the embedding $\ku_1\subset \ku_2$ is Hilbert-Schmidt, then $W_1$ can be viewed as an $\ku_2$-valued Wiener process. Moreover, there exists a trace class symmetric nonnegative operator $Q\in \mathcal{L}(\ku_2)$ such that $W_1$ has  covariance $Q$. This $\ku_2$-valued $\ku_1$-cylindrical  Wiener process is characterised by, for all $t\geq 0$,
\[
\mathbb{E} e^{i \;\fourIdx{}{\ku_2^\ast}{}{\ku_2} {\lb x^\ast, W(t)\rb}  }= e^{-\frac{t}2 \vert x^\ast\vert_{\ku_1}^2}, \;\; x^\ast \in \ku_2^\ast,
\]
where $\ku_2^\ast$ is the dual space to $\ku_2$ such that identifying $\ku_1^\ast$ with $\ku_1$ we have
\[
\ku_2^\ast \embed \ku_1^\ast=\ku_1 \embed \ku_2.\]
\end{Rem}

We have the following relation between
Stratonovich and It\^o's integrals
\begin{equation*}
G(\bd)\circ dW_2= \frac 12 G^2(\bd) \,dt +
G(\bd)\,dW_2,
\end{equation*}
where $G^2=G\circ G$ and defined by
\begin{align*}
G^2(\bd)=&G\circ G(\d)= (\bd \times h)\times h, \text{ for any } \bd \in \el^2.
\end{align*}


 Now, we  rewrite the
problem \eqref{eqn-SLQE-v}-\eqref{eqn-SLQE-d} in the following abstract
form \del{\begin{equation} ,\label{ABS-D1}
\end{equation}}
\begin{align}
d\bv(t)+\biggl(\rA\bv(t)+B(\bv(t),
\bv(t))+M(\bd(t))\biggr)dt=S(\bv(t))dW_1,\label{ABS-v1}
\end{align}
\begin{equation}
d\bd(t)+\biggl(\rA_1\bd(t)+ \tilde{B}(\bv(t),\bd(t))+
f(\bd(t))-\frac 12 G^2(\bd(t))\biggr)dt=G(\bd(t))dW_2.
\label{ABt-d1}
\end{equation}

Throughout this paper we impose that the function $f(\cdot)$ satisfies the following set of conditions.
\begin{assum}\label{eqn-f}
Let $I_d$ be the set defined by
 \begin{equation}\label{DEG-POL}
  I_d=\begin{cases}
       \mathbb{N}:=\{1, 2, 3, \ldots\} \text{ if } d=2,\\
       \{1\}, \text{ if } d=3.
      \end{cases}
 \end{equation}
Throughout this paper we fix $N\in I_d
$  and a family of numbers $a_k$, $k=0,\ldots, N$,
with $a_N>0$.  We define a function $\tilde{f}:[0,\infty) \rightarrow \mathbb{R}$ by
$$ \tilde{f}(r)=\sum_{k=0}^N a_k r^k, \text{ for any } r\in \err_+.$$
We define a map $f:\mathbb{R}^n\rightarrow \mathbb{R}^n$ by $f(\d)=\tilde{f}(\vert \d\vert^2)\d$ where $\tilde{f}$ is as above.

Let $F: \mathbb{R}^d \rightarrow \mathbb{R}$ be a Fr\'echet differentiable map such that for any $\d\in \mathbb{R}^d$ and $\mathbf{g}\in \mathbb{R}^d$ $$ F^\prime(\d)[\mathbf{g}]= f(\d)\cdot \mathbf{g}
.$$ Let also $\tilde{F}$ be an antiderivative of $\tilde{f}$ such that $\tilde{F}(0)=0$. We have
$$ \tilde{F}(r)=a_{N+1}r^{N+1}+U(r),$$
where $U$ is a polynomial function of at most degree $N$ and $a_{N+1}>0$.
\end{assum}
\begin{Rem}
For any $r\in [0,\infty) $ let $\tilde{f}(r):=r-1$. Since $1\in I_d$ then the maps $f$ and $F$ defined on $\mathbb{R}^d$ by $f(\d):=\tilde{f}(\vert \d \vert^2)\d$ and $F(\d):=\frac14 [\tilde{F}(\vert \d \vert^2) ]^2$
for any $\d\in \mathbb{R}^d$ satisfy the above set of assumptions.
\end{Rem}

\begin{Rem}
 We have the following
facts
\begin{align}
\lvert \tilde{f}(r)\rvert\le \ell_1 (1+r^N), \,\, r>0,\label{ST6-B-0}\\
\lvert \tilde{f}^\prime (r) \rvert\le \ell_2 (1+r^{N_1}), \,\, r>0.\label{ST6-B-1}
\end{align}
\end{Rem}

 \begin{Rem}\label{REM-H2}
 Let $f$ be defined in Assumption \ref{eqn-f}. Then, there exist two positive
constants $c>0$ and $\tilde{c}>0$ such that
 $$ \lvert f(\bd)\rvert \le c (1+\vert \d\vert^{2N+1}) \text{ and } \lvert f^{\prime}(\d)\rvert\le
 \tilde{c}(1+\vert \bd \vert^{2N}) \text{ for any } \bd\in \err^n.$$
 Also, by performing elementary calculation we can check that \mbox{   for any } $\d\in \h^2$
  \begin{align}
\lve \Delta \d\rve^2=&\lve \Delta \d -f(\d)+f(\d)\rve^2
\le  2 \lve \Delta \d -f(\d)\rve^2+2 \lve f(\d)\rve^2,\nonumber \\
\le & 2 \lve \Delta \d -f(\d)\rve^2+C \lve \d\rve^{\tilde{q}}_{\mathbb{L}^{\tilde{q}}}+C
,\label{bigdandel}
  \end{align}
  where $\tilde{q}=4N+2$.

 Hence there exists a constant $C>0$
  such that
  \begin{equation}\label{bigdanh2}
\lve \d\rve^2_2\le C (\lve \Delta \d -f(\d)\rve^2
+\lve \d\rve^{\tilde{q}}_{\mathbb{L}^{\tilde{q}}}+1), \mbox{   for any }\d\in \h^2(\MO).
  \end{equation}

  \end{Rem}
\section{Existence and uniquness of weak martingale solution}\label{SLC-sect3}
In this section we are going to establish the existence of weak martingale solution of the system
\eqref{ABS-v1}-\eqref{ABt-d1}. For this purpose we use the Galerkin approximation to reduce the original
system to a system of finite dimensional ordinary stochastic differential equations (SDEs for short).
We establish several crucial uniform a priori estimates which will be used to prove the tightness of the family of laws of the sequence of solutions of the system of SDEs on appropriate
topological spaces. But before we proceed further we define what we mean by weak martingale solution.
\begin{Def}\label{WEAK-MART}
Let $\mathrm{K}_1$ be as in Remark \ref{Wiener}\dela{and $Q\in \mathcal{L}(\mathrm{K})$ be a nonnegative and symmetric with $\tr
Q<\infty$}. By a weak martingale solution to
\eqref{ABS-v1}-\eqref{ABt-d1} we mean a system consisting of a
complete and filtered probability space \[(\Omega^\prime,
\mathcal{F}^\prime, \mathbb{P}^\prime,
\mathbb{F}^\prime),\] with the filtration $\mathbb{F}^\prime=(\mathcal{F}_t^\prime)_{t\in [0,T]}$
satisfying the usual condition,  and progressively measurable
stochastic processes \[(\bv(t), \bd(t), \wi(t),
\wt(t))_{t\in[0,T]}\] such that:
\begin{enumerate}
\item $(\wi(t))_{t\in [0,T]}$ (resp. $(\wt(t))_{t\in[0,T]}$) is {$\mathbb{F}$-progressively measurable} and {$\mathrm{K}_1$-cylindrical} (resp. real-valued) Wiener process,
\item $(\bv, \bd): [0,T]\times
\Omega^\prime\rightarrow \ve\times\h^2$ and
$\mathbb{P}^\prime$-a.e. $\omega^\prime\in \Omega^\prime$
\begin{align}
(\bv, \bd)(\cdot, \omega^\prime) \in C([0,T]; \ve^{-\beta})\times
C([0,T];\bx_{\beta-\frac12}), \text{ for
any } \beta \in (0,\frac 12),\\
\mathbb{E}^\prime \sup_{0\le s\le
T}\left[\lve\bv(s,\omega^\prime)\rve+\lve\nabla
\bd(s,\omega^\prime)\rve \right]+ \mathbb{E}^\prime\int_0^T
\left(\lve \nabla \bv(s,\omega^\prime)\rve^2+ \lve \Delta
\bd(s,\omega^\prime)\rve^2\right)ds<\infty,
\end{align}
\item for each $(\Phi, \Psi)\in \ve\times \el^2$ we have, for all
$t\in [0,T]$, \,\, $\mathbb{P}^\prime$-a.s..

\begin{align}
\nonumber \langle \bv(t)-\bv_0, \Phi\rangle&+\int_0^t
\biggl\langle \rA \bv(s)+B(\bv(s), \bv(s))+M(\bd(s)),
\Phi\biggr\rangle ds\\ &= \int_0^t \langle S(\bv(s)),\Phi\rangle
dW_1(s),
\end{align}
and
\begin{align}
\nonumber
 \langle
\bd(t)-\bd_0,\Psi\rangle &+\int_0^t \biggl\langle \rA_1\bd(s)+
\tilde{B}(\bv(s),\bd(s))+f(\bd(s))-\frac 12 G^2(\bd(s)), \Psi\biggr\rangle ds\\
&=\int_0^t \langle G(\bd(s)), \Psi\rangle dW_2(s).
\label{ABS-D1}
\end{align}
\end{enumerate}
\end{Def}
The function $S$ is defined in the next set of hypotheses.
\begin{assum}\label{HYPO-ST-weak}

We assume that $S: \h\to \mathcal{T}_2(\rK_1,\h)$ is a globally Lipschitz map. In particular,
there exists $\ell_3\geq 0$
such that
\begin{equation}\label{LIN-GROW}
\lve  S(\bu)\rve^2_{\mathcal{T}_2}:=\lve  S(\bu)\rve^2_{\mathcal{T}_2(\rK_1,\h)}\leq \ell_3 (1+\lve \bu \rve^2),\;\; \mbox{ for any } \bu \in \h.
\end{equation}
\end{assum}
Now we can state our first result in the following theorem.
\begin{thm}\label{thm-main-weak-mart}
 If Assumption \ref{HYPO-ST-weak}  and \ref{eqn-f} are satisfied,  $h\in \mathbb{W}^{1,3}\cap \el^\infty$, $\bv_0\in \h$,
 $\d_0\in \h^1$, and $d=2,3$. Then, the system \eqref{ABS-v1}-\eqref{ABt-d1} has a weak
martingale solution in the sense of Definition \ref{WEAK-MART}.
\end{thm}
\begin{proof}
	The proof relies on Galerkin approximation and compactness method and is the subject of the subsections \ref{GALERKIN}-\ref{Proof-of-EXIST}.
\end{proof}
Before we state the uniqueness of the weak martingale solution we should make the following remark.
\begin{Rem}
 We should note that the existence of weak martingale solution stated in Theorem \ref{thm-main-weak-mart} still holds if we assume that the map $S(\cdot)$ is only continuous and satisfies a linear growth condition of the form
 \eqref{LIN-GROW}.
\end{Rem}
To close this subsection we assume that $d=2$ and we  state the following uniqueness result. 
 \begin{thm}\label{UNIQUENESS}
 	Let $d=2$ and assume that $(\v_i,\d_i)$,\,\, $i=1,2$ are two
 	solution of \eqref{ABS-v1} and \eqref{ABt-d1} defined on the same stochastic system $(\Omega, \mathcal{F}, \mathcal{F}_t,\mathbb{P},W_1,
 	W_2)$ and with the same
 	initial condition $(\bv_0,\d_0)\in \h\times \h^1$. Then, for any $t\in (0,T]$, $\mathbb{P}$-a.s. $$(\v_1(t),\d_1(t))=(\v_2(t),\d_2(t) ).$$
 \end{thm}
 \begin{proof}
 	The proof of this result will be carried out in Subsection \ref{Proof-of-Uniq}.
 \end{proof}
\subsection{Galerkin approximation and a priori uniform estimates}\label{GALERKIN}
As we mentioned earlier the proof of the existence of weak martingale solution relies on the Galerkin and compactness methods. This subsection will be devoted to the construction of the approximating solutions and the proofs of crucial estimates satisfied by these solutions.

Recall that there exists an orthonormal basis $\{\varphi_i; i=1, 2, 3, \ldots\}$  of $\h$ consisting of the
eigenvectors of the Stokes operator $\rA$. Recall also that there exists an orthonormal basis $\{\phi_i; i=1,2, 3,\ldots\}$
of $\el^2$ consisting of the eigenvectors of the Neumann Laplacian $\rrA$. For any $m\in \mathbb{N}$ let us define the following finite-dimensional spaces
 \begin{align*}
  \h_m:=\lsp\{\varphi_1, \ldots, \varphi_m\},\\
  \el_m:=\lsp\{\phi_1, \ldots, \phi_m\}.
 \end{align*}
 In this subsection we introduce the
finite-dimensional approximation of the system
\eqref{ABS-v1}-\eqref{ABt-d1}, justify the existence of solution of such approximation and We are mainly interested in
derive uniform estimates for the sequence of approximating solutions.
 For this we
purpose, we denote by $\prm$ (resp. $\qm$) the projection from $\h$
(resp. $\el^2$) onto $\h_m$ (resp. $\el_m$). These operators are
self-adjoint. We also consider several mappings:
\begin{align*}
&{B}_{m}: \h_m \ni \bu \mapsto \prm B(\bu,\bu) \in \h_m,\\
&S_m: \h_m \ni \bu \mapsto \prm \circ S(\bu)\in \mathcal{T}_2(\rK_1, \h_m),\\
&M_m: \el_m \ni \bd \mapsto  \prm M(\bd) \in \h_m,\\
&f_m: \el_m \ni \bd \mapsto  \qm f(\bd) \in \el_m,\\
&G_m: \el_m \ni \bd \mapsto \qm G(\bd)\in \el_m,\\
&G^2_m:\el_m \ni \bd \mapsto \qm G^2(\bd) \in \el_m,\\
&\tilde{B}_{m}:\h_m\times \el_m \ni (\bu,\bd)\mapsto \qm
\tilde{B}(\bv,\bd) \in \el_m.
\end{align*}
The proof of the following lemma is easy and we omit it.
\begin{lem}\label{LEM-LOC}
The above mappings are locally Lipschitz.
\end{lem}
Let $\d_{0m}=\qm \d_0$ and $\v_{0m}=\prm \v_0$. The Galerkin
approximation to \eqref{ABS-v1}-\eqref{ABt-d1} is
\begin{align}
&d\vmt+\big[\ma\vmt+{B}_{m}(\vmt)+M_m(\dmt)\big]dt=S_m(\vmt)dW_1,\label{GAL-v1}\\
&d\dmt+\big[\rA_1\dmt+\tilde{B}_{m}(\vmt,\dmt)+f_m(\dmt)\big]dt=\frac 12
G^2_m(\dmt)+G_m(\dmt)dW_2.\label{GAL-d1}
\end{align}
The equations \eqref{GAL-v1}-\eqref{GAL-d1} with initial condition
$\vm(0)=\v_{0m} \text{ and } \dm(0)=\d_{0m}$ form a system of
stochastic ordinary differential equations with locally Lipschitz
coefficients hence owing to \cite[Theorem 38, p. 303]{Protter} it
has a unique local maximal solution$(\vm;\dm; T_m)$ where $T_m$ is a stopping time.
\begin{Rem}
 In case we assume that $S(\cdot)$ is only continuous and satisfies \eqref{LIN-GROW}, Lemma \ref{LEM-LOC} is no longer true. In fact the mappings defined just above Lemma \ref{LEM-LOC}
  are only continuous and locally bounded. However, with this fact we can still justify the existence, possibly non-unique, of a weak local martingale solution to \eqref{GAL-v1}-\eqref{GAL-d1} by using results in \cite[Chapter IV, Section 2, pp
167-177]{IW-89}.
\end{Rem}

\begin{prop}\label{zero-est}
If all the assumptions of Theorem \ref{thm-main-weak-mart} are satisfied, then for any $p\ge2$ there exists a positive constant $C_p$ such that
 \begin{equation}\label{EST-Nm}
 \begin{split}
 \sup_{m\in \mathbb{N}}\biggl(\mathbb{E} \sup_{t\in [0,T_m]}\lve \dm(t)\rve^p+p \int_0^{t\wedge T_m} \lve \dm(s) \rve^{p-2} \lve \nabla \dm(s)\rve^2 ds
 \\ + p \int_0^{t\wedge T_m} \lve \dm(s) \rve^{p-2} \lve \dm(s)\rve^{2N+2}_{\mathbb{L}^{2N+2}} ds\biggr)\le \mathbb{E}\mathfrak{G}_0(T,p).
 \end{split}
 \end{equation}
 where
 \begin{equation}\label{G_0}
 \mathfrak{G}_0(T,p):=\lve \d_0\rve^p (C_p+C_p e^{C_p T}).
 \end{equation}

\end{prop}
\begin{proof}
Note that by the self-adjointness of $\tilde{\pi}_m$
\begin{align*}
\langle G_m(\dm),\dm \rangle= &\langle \qm(\dm\times h),
\dm\rangle,\\
=& \langle \dm\times h, \dm\rangle,\\
=&0.
\end{align*}
Since $\vmt$ is a divergence free function it follows from  \eqref{tild-b-0} that
$$\langle \qm \tilde{B}(\vmt,\dmt),\dmt\rangle=\langle \tilde{B}(\vmt,\dmt),\dmt\rangle=0. $$
Now we can prove the proposition by using the same ideas as in the proof of the estimate \eqref{EST-N}.
\end{proof}
We also have the following estimates.
\begin{prop}\label{first-est}
If all the assumptions of Theorem \ref{thm-main-weak-mart} are satisfied, then there exists $\ell>0$ such that  for any $p \in [1,\infty)$  the pair
$(\vm, \dm)$ satisfies
\begin{equation}
\begin{split}
\mathbb{E} \biggl[\sup_{0\le s\le
T_m}\left(\lve\vm(s)\rve^{2}+\lve \nabla \dm(s)\rve^{2}+\ell\Vert \dm(s)\Vert^2 +\int_\MO F(\dm(s,x)) dx \right)^p\\+
\left(\int_0^{T_m}
\left(\lve\nabla
\vm(s)\rve^2+ \lve \Delta \dm(s)-f(\dm(s))\rve^2\right)\right)^p\biggr]
\le \mathfrak{G}_1(T, p),\label{EST-VD-A-Gal}
\end{split}
\end{equation}
and
\begin{equation}\label{EST-D-A-Gal}
  \me \biggl[\int_0^{T_m} \lve \Delta \dm(s)\rve^2 ds\biggr]^p \le \mathfrak{G}_1(T, p\cdot(2N+1)),
 \end{equation}
 where
 \begin{equation}\label{G_1}
  \begin{split}
   \mathfrak{G}_1(T, p):= \biggl[\left(\Vert \v_0\Vert^2 +\Vert \d_0\Vert^2+\Vert \nabla \d_0\Vert^2+\int_\MO F(\d_0(x)) dx\right)^p+ \kappa T+\kappa\mathfrak{G}_0(T,p)\biggr]\\ \times
 \left[1+ \kappa T(T+1) e^{\kappa(T+1)T }\right],
  \end{split}
 \end{equation}
$\mathfrak{G}_0$ is defined in \eqref{G_0} and $\kappa$ is a positive constant depending only on $p$.
%
\end{prop}
\begin{proof}
The estimates \eqref{EST-VD-A-Gal} follows from Proposition \ref{EST1} by taking $\tau_R=T_m$.
\end{proof}
We state the following corollary.
\begin{cor}
 The estimates \eqref{EST-VD-A-Gal} and \eqref{EST-D-A-Gal} remains true with $T_m$ replaced by $T$.
\end{cor}
\begin{proof}
 Since the constants on the right hand sides of \eqref{EST-VD-A-Gal} and \eqref{EST-D-A-Gal} does not depend on $m$, we can use the same argument as in
\cite{Albeverio+al} to show that for any $m\in \mathbb{N}$
\begin{equation*}
T_m =T \quad \mathbb{P}-\text{a.s.}.
\end{equation*}
This shows that the solution of our Galerkin approximation is
globally defined, that is $T_m=T$ and the estimates
\eqref{EST-VD-A-Gal} and \eqref{EST-D-A-Gal} remain valid with $T_m$ replaced by $T$.
\end{proof}

In the next proposition we prove two uniform estimates for $\vm$
and $\dm$ which are very crucial for our purpose.
\begin{prop}\label{far-est}
In addition to the assumptions of Theorem  \ref{thm-main-weak-mart}, let $\alpha \in (0,\frac 12)$ and $p\in [2,\infty)$ such
that $1-\frac{d}{4}\ge \alpha-\frac 1p$. Then, there exist positive constants $\bar{\kappa}_5$ and $\bar{\kappa}_6$ such that
 \begin{equation}\label{FSOB-est-vm}
 \sup_{m\in \mathbb{N}} \mathbb{E}\lve \vm\rve^2_{W^{\alpha,p}(0,T; \ve^\ast)}\le \bar{\kappa}_5,
\end{equation}
 and
 \begin{equation}\label{FSOB-est-dm}
 \sup_{m\in \mathbb{N}}\mathbb{E} \lve \dm
\rve^2_{W^{\alpha,p}(0,T;\el^{r})}\le \bar{\kappa}_6,
\end{equation}
where $r=\frac{2N+2}{2N+1}\in (1,2)$.
\end{prop}
\begin{proof}
We rewrite the  equation for $\vm$ as
\begin{equation*}
 \begin{split}
  \vmt=&\v_{0m}-\int_0^t \rA \vm(s)ds -\int_0^t {B}_{m}(\vm(s),\vm(s))ds -\int_0^t M_m(\dm(s)) ds\\
  & \quad \quad \quad+\int_0^t S_m(\vm(s))dW_1,\\
  =& \v_{0m}+\sum_{i=1}^4 I_m^i(t).
 \end{split}
\end{equation*}
Since $\rA \in \mathcal{L}(\ve,\ve^\ast)$ we easily check that there exists a
certain constant $C>0$ such that
\begin{equation}\label{IM1}
\sup_{m\in \mathbb{N}}\mathbb{E} \lve I^1_m\rve_{W^{1,2}(0,T;
\ve^\ast)}^2\le C.
\end{equation}
We recall that for any $\alpha \in
(0,\frac 12)$ and $p\in [2,\infty)$
\begin{equation}\label{est-INTSTO}
 \mathbb{E}\biggl\lve \int_0^{.} \xi(s) dW(s)\biggr\rve_{W^{\alpha, p}(0,T;\h)}^p\le c \mathbb{E}\int_0^T \lve \xi(s)\rve^p_\h ds;
\end{equation}
see, for instance, \cite{Flandoli+Gatarek}.
From the above fact we infer that
 \begin{equation*}
\begin{split}
 \mathbb{E}\lve I^4_m\rve^p_{W^{\alpha, p}(0,T;\h)}\le & \mathbb{E}\int_0^T \lve S_m(\vmt)\rve_{\mathcal{T}_2}^p dt,\\
 \le & c \mathbb{E}\int_0^T(\ell_1+\ell_2 \lve \vmt\rve^p)ds.
\end{split}
\end{equation*}
Thanks to \eqref{EST-VD-A} we derive that
\begin{equation}\label{IM4}
\sup_{m\in \mathbb{N}} \mathbb{E}\lve I^4_m\rve^p_{W^{\alpha,
p}(0,T;\h)}\le C.
\end{equation}
Now we treat the term $I^3_m(t)$. As a consequence of Lemma
\ref{DET-EST-M}, we have
\begin{equation*}
 \mathbb{E} \lve M_m(\dm)\rve^2_{\el^{\frac d4}(0,T; \ve^\ast) }\le
 \biggl[\mathbb{E}\biggl(\int_0^T \lve \Delta \dmt \rve^2 dt \biggr)^d \biggr]^\frac 12,
\end{equation*}
from which altogether with \eqref{EST-VD-A} we infer that
\begin{equation}\label{IM3-2}
 \mathbb{E} \lve I^3_m\rve^2_{W^{1,\frac d4}(0,T; \ve^\ast)}\le C.
\end{equation}

Using \eqref{B4} and an argument similar to the proof of the estimate for  $I_m^3$ we
conclude that
\begin{equation}\label{IM2}
 \sup_{m\in \mathbb{N}}\mathbb{E} \lve I^2_m(t)\rve^2_{W^{1,\frac d4}(0,T; \ve^\ast) }\le C.
\end{equation}
By \cite[Section 11, Corollary 19]{Simon-90}  we have the
continuous imbedding
\begin{equation}\label{Imdedd}
W^{1,\frac d4}(0,T;\ve^\ast) \subset W^{\alpha, p}(0,T; \ve^\ast),
\end{equation}
 for $\alpha \in (0,\frac 12)$ and $p\in [2,\infty)$ such that
$1-\frac{d}{4}\ge \alpha-\frac 1p$.  Owing to Eqs. \eqref{IM1},
\eqref{IM3-2}, \eqref{IM4} and \eqref{IM2} and this continuous
embedding we infer that \eqref{FSOB-est-vm} holds.

The second equations for the Galerkin approximation is written as
\begin{equation*}
 \begin{split}
  \dmt=&\d_{0m}-\int_0^t \rA_1\dm(s) ds -\int_0^t\qm[ \tilde{B}_{m}(\vm(s),\dm(s))] ds -\int_0^t f_m(\dm(s)) ds\\
  & \quad \quad +\frac 12 \int_0^t G^2_m(\dm(s)) ds +\int_0^t G_m(\dm(s)) dW_2,\\
  = &\d_{0m}+\sum_{j=1}^5 J^j_m(t).
 \end{split}
\end{equation*}
From \eqref{EST-D-A-Gal} it is clear that \begin{equation}\label{JM1}
                  \sup_{m\in \mathbb{N}}\mathbb{E} \lve J^1_m\rve^q_{W^{1,2}(0,T;\el^2)}\le C.
                 \end{equation}
For $J^2_m$ we have
\begin{equation*}
 \lve\qm[ \tilde{B}_{m}(\vm(s),\dm(s))] \rve\le c \lve \vmt\rve_{\el^4}\lve \nabla \dmt
 \rve_{\el^4}, \text{ for } t\in [0,T].
\end{equation*}
Thanks to Gagliardo-Nirenberg's inequality we have
\begin{equation*}
 \lve\qm[ \tilde{B}_{m}(\vm(s),\dm(s))] \rve\le c \left(\lve \vmt \rve\lve \nabla \dmt \rve\right)^{\frac{4-d}{4}}\left(\lve \nabla \vmt\rve \lve \Delta \dmt\rve\right)^{\frac d4}.
\end{equation*}
Thus
\begin{equation*}
\begin{split}
 \lve\qm[ \tilde{B}_{m}(\vm(s),\dm(s))] \rve_{\el^\frac d4(0,T;\el^2)}^2\le & c \sup_{0\le t\le T}\left(\lve \vmt \rve\lve \nabla \dmt \rve\right)^{\frac{4-d}{2}}
 \biggl[\int_0^T \lve \nabla \vmt \rve^2 dt \biggr]^{\frac d4}\\
 & \times \biggl[\int_0^T\lve \Delta \dmt \rve^2 dt \biggr]^{\frac{d}{4}}.
 \end{split}
\end{equation*}
Taking the mathematical expectation and using H\"older's
inequality lead to
\begin{equation*}\begin{split}
 \sup_{m\in \mathbb{N}}\mathbb{E}\lve\qm[ \tilde{B}_{m}(\vm(s),\dm(s))] \rve_{\el^\frac d4(0,T;\el^2)}^2\le & c \biggl[\mathbb{E}\sup_{0\le t\le T}\lve \vmt\rve^{2(4-d)} \biggr]^\frac 14
  \biggl[\mathbb{E}\sup_{0\le t\le T}\lve \nabla \dmt\rve^{2(4-d)} \biggr]^\frac 14 \\ &
 \times \biggl[\mathbb{E}\left(\int_0^T \lve \nabla \vmt \rve^2  dt\right)^d
                 \end{split}
\end{equation*}

Combining this with \eqref{EST-VD-A-Gal} and \eqref{EST-VD-A-Gal} yield
\begin{equation}\label{JM2}
 \sup_{m\in \mathbb{N}}\mathbb{E}\lve J^2_m\rve^2_{W^{1,\frac d4}(0,T;\el^2)}\le C.
\end{equation}
\del{Thanks to the boundedness of $f(\cdot)$ we easily check that
\begin{equation}\label{JM3}
 \sup_{m\in \mathbb{N}}\mathbb{E}\lve J^3_m(t)\rve^2_{W^{1,\infty}(0,T; \el^2)}\le C.
\end{equation}}

\del{\red{\textsc{Remark:} If $f$ is polynomial then we should
argue as follows. Owing to \eqref{est-dm-POL} there exists a
positive constant $C(\d_0, \MO, T)$ such that
 \begin{equation*}
 \begin{split}
 \sup_{m\in \mathbb{N}}\lve f_m(\dm)\rve_{\el^{q^\ast}([0,T]\times\MO)}\le
 \biggl[\int_{[0,T]\times \MO}\Big\lvert(\lvert \dm(t,x) \rvert^{k-1}-1) \dm(t,x)\Big\rvert^{\frac{2k}{2k-1}}dxdt
 \biggr]^{\frac{2k-1}{2k}},\\
\le \lve \dm\rve^{\frac{2k-1}{2k}}_{\el^{2k}([0,T]\times\MO)},\\
\le C(\d_0, \MO, T) \,\,\, \mathbb{P}-\text{a.s.}.
 \end{split}
\end{equation*}
}}

For any $m$ and $t\in [0,T]$
\begin{equation*}
\begin{split}
 \lve G^2_m(\dmt)\rve\le & \lve h\rve_{\el^\infty }\lve\dmt\rve_{\el^\infty}\lve \dmt \rve,\\
 \le & \lve h\rve_{\el^\infty } \left(\lve \dmt\rve^2 + \lve \dmt \rve \lve \nabla \dmt \rve+ \lve \dmt \rve \lve \Delta \dmt\rve \right),
\end{split}
\end{equation*}
which along with  \eqref{EST-VD-A-Gal} and \eqref{EST-D-A-Gal} yields
\begin{equation}\label{JM4}
 \sup_{m\in \mathbb{N}}\mathbb{E}\lve J^4_m\rve^2_{W^{1,2}(0,T;\el^2) }\le C.
\end{equation}
For any $h\in \el^\infty(\MO)$ we have
\begin{equation}\label{est-G}
\lve h\times \dmt\rve^p\le \lve h\rve_{\el^\infty}^p \lve
\dmt\rve^p,
\end{equation}
from which with \eqref{est-INTSTO} and \eqref{est-G}, we  derive that
for any $\alpha\in (0,\frac 12)$ and $p\in [2,\infty)$
\begin{equation}\label{JM5}
 \sup_{m\in \mathbb{N}}\mathbb{E} \lve J^5_m\rve^p_{W^{\alpha,p}(0,T; \el^2)}\le C.
\end{equation}
For the polynomial nonlinearity $f$ we have
 \begin{align}
 \sup_{m\in \mathbb{N}}\mathbb{E} \lve J^3_m\rve^p_{W^{\alpha,p}(0,T; \el^r)} \le & \me \int_0^T \int_\MO \vert f(\dm(x,s))\vert^r dx ds
 \nonumber \\ \le & C + C \me \int_0^T \int_\MO\vert \dm(x,s)\vert^{2N+2} dxds\nonumber \\
  \le& C + C_2,\label{Jm3}
 \end{align}
 where we have used \eqref{EST-Nm}.
 Combining all these estimates complete the proof of our proposition.
\end{proof}
\subsection{Tightness and Compactness results}
This subsection is devoted to the study of the tightness of the Galerkin solutions and derive several weak convergence results. The estimates from the previous subsection play an important role in this part of the paper.

Let $p\in [2,\infty)$, $\alpha\in (0,\frac12)$ and $r$ be as in Proposition \ref{far-est}. Let us consider the spaces
\begin{align*}
\mathfrak{X}_1=& L^2(0,T;
\ve)\cap W^{\alpha, p}(0,T;\ve^\ast),\\
\mathfrak{Y}_1=& L^2(0,T;\h^2)\cap W^{\alpha,
	p}(0,T;\el^{r})
\end{align*}
 Recall that
$\ve_\beta$, $\beta \in \mathbb{R}$,  is the domain of the of the
fractional power operator $\rA^\beta$. Similarly, $\mathbb{X}_\beta$ is the domain of $(I+\rrA)^{\frac12+\beta}$. The embedding
$\ve_\gamma\subset \ve_\beta$ (resp. $\mathbb{X}_\gamma \subset
\mathbb{X}_\beta$) is compact if that $\gamma > \beta$.   We
set
\begin{align*}
 \mathfrak{X}_2=L^\infty(0,T;\h)\cap W^{\alpha, p}(0,T; \ve^\ast),\\
 \mathfrak{Y}_2=L^\infty(0,T;\h^1)\cap W^{\alpha, p}(0,T; \el^2),
\end{align*}
and
\begin{align*}
 \mathfrak{S}_1=& L^2(0,T;\h)\cap C([0,T];\ve_{-\beta}),\\
 \mathfrak{S}_2=& L^2(0,T;\h^1)\cap C([0,T];\bx_{\beta-\frac12}).\\
\end{align*}
We shall prove the following important result.
 \begin{thm}\label{THM-COM}
Let $p\in [2,\infty)$ and $\alpha\in (0,\frac12)$ be as in Proposition \ref{far-est} and  $\beta\in (0,\frac 12)$ such that
$p\beta>1$. The family of laws $\{\mathcal{L}(\vm, \dm), m\in
\mathbb{N}\}$ is tight on the Polish space $\mathfrak{S}_1\times
\mathfrak{S}_2$.
 \end{thm}
 \begin{proof}
We firstly prove that $\{\mathcal{L}(\vm), m\in \mathbb{N}\}$ is
tight on $L^2(0,T;\h)$. For this aim, we fix an arbitrary number $R>0$  and prove that
\begin{equation*}
\begin{split}
\mathbb{P}\left(\lve \vm \rve_{\mathfrak{X}_1}> R\right)\le
\mathbb{P}\left(\lve \vm \rve_{L^2(0,T;\ve)}>\frac
R2\right)+\mathbb{P}\left(\lve \vm
\rve_{W^{\alpha,p}(0,T;\ve^\ast)}>\frac
R2\right),\\
\le \frac{4}{R^2} \mathbb{E}\left(\lve \vm
\rve^2_{L^2(0,T;\ve)}+\lve \vm
\rve_{W^{\alpha,p}(0,T;\ve^\ast)}\right).
\end{split}
\end{equation*}
Thanks to \eqref{EST-VD-A}, \eqref{FSOB-est-vm}, and
\eqref{FSOB-est-dm} we obtain from the last estimate that
\begin{equation}\label{tight-1}
\sup_{m\in \mathbb{N}}\mathbb{P}\left( \lve \vm
\rve_{\mathfrak{X}_1}> R\right)\le \frac{4C}{R^2}.
\end{equation}
Since $\mathfrak{X}_1$ is compactly embedded into $L^2(0,T;\h)$,
we conclude that the laws of $\vm$ form a family of probability
measures which is tight on $L^2(0,T;\h)$. Secondly, the same
argument is used to prove that the laws of $\dm$ are tight on
$L^2(0,T;\el^r)$. Next, let $\beta\in (0,\frac 12)$ and $p\in
[2,\infty)$ be satisfying $p\beta>1$. By {\cite[Corollary 5 of
Section 8]{Simon-87}} the spaces $\mathfrak{X}_2$ and
$\mathfrak{Y}_2$ are compactly imbedded in $C([0,T];\ve_{-\beta})$
and $C([0,T];\bx_{\beta-\frac12})$, respectively. Hence the same argument as
above provides us with the tightness of $\mathcal{L}(\vm)$ and
$\mathcal{L}(\dm)$ on $C([0,T];\ve_{-\beta})$ and
$C([0,T];\bx_{\beta-\frac12})$.  Now we can easily conclude the proof of the
theorem.
 \end{proof}
Throughout the remaining part of this paper we assume that $\alpha$, $p$ and $\beta$ are as in Theorem \ref{THM-COM}.
We also use notation from Remark \ref{Wiener}.

\begin{prop}\label{Proh-lem}
  Let $\mathfrak{S}=\mathfrak{S}_1\times\mathfrak{S}_2\times C([0,T];\ku_2)\times
 C([0,T];\mathbb{R})$. There exist a Borel probability measure $\mu$ on $\mathfrak{S}$ and a subsequence of $(\vm,\dm, W_1,
 W_2)$ such that its laws weakly converge to $\mu$.
 \end{prop}
 \begin{proof}
Thanks to the above lemma the laws of $(\vm,\dm, W_1,
 W_2)$ form a tight family on $\mathfrak{S}$.  Since $\mathfrak{S}$ is a Polish space, we get the result from the application of Prohorov's theorem.
 \end{proof}
 The following result relates the above convergence in law to
 almost sure convergence.
 \begin{prop}\label{SKO-Lem}
There exist a complete probability space $(\Omega^\prime,
\mathcal{F}^\prime, \mathbb{P}^\prime)$ and a sequence
 of $\mathfrak{S}$-valued random
variables, denoted by $(\bvm,\bdm, W_1^m, W_2^m)$, defined on
$(\Omega^\prime, \mathcal{F}^\prime, \mathbb{P}^\prime)$ such that
its law is equal to the law of $(\vm,\dm, W_1, W_2)$ on
$\mathfrak{S}$. Also, there exists an $\mathfrak{S}$-random
variable $(\bv, \bd, \bar{W}_1, \bar{W}_2)$ defined on
$(\Omega^\prime, \mathcal{F}^\prime, \mathbb{P}^\prime)$ such that
\begin{align}
&\mathcal{L}(\bv, \bd, \wi, \wt)=\mu \text{ on }
\mathfrak{S},\label{id-law}\\
&\bvm \rightarrow \bv \text{ in } L^2(0,T;\h) \,\,
\mathbb{P}^\prime-a.s.,\label{SKO-1}\\
&\bvm \rightarrow \bv \text{ in } C([0,T];\ve_{-\beta}) \,\,
\mathbb{P}^\prime-a.s.,\label{SKO-2}\\
&\bdm \rightarrow \bd \text{ in } L^2(0,T;\h^1) \,\,
\mathbb{P}^\prime-a.s.,\label{SKO-3}\\
&\bdm \rightarrow \bd \text{ in } C([0,T];\bx_{\beta-\frac12}) \,\,
\mathbb{P}^\prime-a.s.,\label{SKO-4}\\
&\bar{W }_1^m \rightarrow \wi \text{ in } C([0,T];\ku_2) \,\,
\mathbb{P}^\prime-a.s.,\label{SKO-5}\\
&\bar{W}_2^m \rightarrow \wt \text{ in } C([0,T];\mathbb{R}) \,\,
\mathbb{P}^\prime-a.s.\label{SKO-6}
\end{align}
 \end{prop}
 \begin{proof}
This is just a consequence of Proposition \ref{Proh-lem} and
Skorokhod's Theorem.
 \end{proof}
 Let $\mathfrak{X}_3=L^\infty(0,T;\h)\cap L^2(0,T; \ve)$ and
$\mathfrak{Y}_3= L^\infty(0,T;\h^1)\cap L^2(0,T; \h^2)$
 \begin{prop}\label{BOUND-OP}
If all the assumptions of Theorem \ref{thm-main-weak-mart} are verified, then  for any $p\ge 2$ $(\bvm, \bdm)$ satisfies the following estimates on the
new probability space $(\Omega^\prime, \mathcal{F}^\prime,
\mathbb{P}^\prime)$:
\begin{align}
& \sup_{m\in \mathbb{N}}\biggl(\mathbb{E}\biggl[ \sup_{t\in [0,T]}\lve \bdm(t)\rve^p+p \int_0^T \lve \bdm(s) \rve^{p-2} \lve \nabla \bdm(s)\rve^2 ds
 \nonumber \\ & \qquad \qquad \qquad + p \int_0^t \lve \bdm(s) \rve^{p-2} \lve \bdm(s)\rve^{2N+2}_{\mathbb{L}^{2N+2}} ds\biggr]\biggr)\le \mathfrak{G}_0(T,p), \label{EST-VD-2-B-A}\\
& \mathbb{E}\biggl[ \sup_{0\le s\le
 	T}\left(\lve\bvm(s)\rve^{2}+\ell \lve \bdm(s) \rve^2+\lve \nabla \bdm(s)\rve^{2}+\int_\MO F(\bdm(s,x)) dx \right)^p\nonumber \\
&\qquad \qquad  + \int_0^T
 \left(\lve \nabla
 \bvm(s)\rve^2+ \lve \Delta \bdm(s)-f(\dm(s))\rve^2\right)\biggr]^p
 \le \mathfrak{G}_1(T,p),\label{EST-VD-2-B}\\
 & \me \biggl[\int_0^T \lve \Delta \dm(s)\rve^2 ds\biggr]^p \le \mathfrak{G}_1(T,p\cdot (2N+1)), \label{EST-VD-2-B-C}
\end{align}
 where $\mathfrak{G}_0(T,p)$, $\ell$ and $\mathfrak{G}_1(T,p)$ are defined as in Proposition \ref{zero-est} and Proposition \ref{first-est}, respectively.
Furthermore, there exists a constant $C>0$ such that
\begin{align}
& \sup_{m \in \mathbb{N}}\mathbb{E}^\prime \biggl[\int_0^T \lve {B}_{m}(\bvm(t),\bvm(t) \rve_{\ve^\ast}^{\frac4d} dt\biggr]^\frac d2\le C,\label{EST-VD-3}\\
& \sup_{m\in \mathbb{N}}\mathbb{E}^\prime \biggl[\int_0^T \lve M_m(\bdm(t))\rve_{\ve^\ast}^{\frac d4} dt\biggr]^\frac d2 \le C, \label{EST-VD-4}\\
& \sup_{m\in \mathbb{N}}\mathbb{E}^\prime \biggl[\int_0^T \lve \tilde{B}_{m}(\bvm(t), \bdm(t)) \rve_{\el^2}^\frac d4 dt\biggr]^\frac d2\le C, \label{EST-VD-5}\\
& \sup_{m\in \mathbb{N}} \me^\prime \int_0^T \lve f_m(\bdm(t))\rve^r_{\el^{r}} dt\le C,\label{EST-VD-6}
\end{align}
where $r$ is defined in Proposition \ref{far-est}.
 \end{prop}

 \del{\red{Because of these estimates we limit ourself to the case
 $n<4$!!!!!!!!!}}
 \begin{proof}
Consider the function $\Phi(\bu, \mathbf{e})$ on
$\mathfrak{X}_3\times \mathfrak{Y}_3\subset \mathfrak{S}_1\times
\mathfrak{S}_2$
 defined by
\begin{equation*}
\Phi(\bu,\mathbf{e})=\sup_{0\le s\le
T}\left[\lve\bu(s)\rve^{2p}+\lve\nabla
\mathbf{e}(s)\rve^{2p}\right]+ \tilde{\kappa}_0 \biggl[\int_0^T
\left(\lve \nabla \bu(s)\rve^2+ \lve \Delta
\mathbf{e}(s)\rve^2\right)ds\biggr]^p
\end{equation*}
It is a continuous function, thus Borel measurable, on $\mathfrak{S}_1\times
\mathfrak{S}_2$. Thanks to  \eqref{id-law}  the processes
$(\vm,\dm)$ and $(\bvm, \bdm)$ are identical in law. \del{ Since
Borel sets of $\mathfrak{X}_3\times \mathfrak{Y}_3$  are Borel
sets of $\mathfrak{S}_1\times \mathfrak{S}_2$} Therefore, we derive
that
\begin{equation*}
\mathbb{E} \Phi(\vm,\dm)= \mathbb{E}^\prime \Phi(\bvm, \bdm).
\end{equation*}
This identity and the estimate \eqref{EST-VD-A} implies
\eqref{EST-VD-2-B}. The estimates \eqref{EST-VD-3}, \eqref{EST-VD-4}
and \eqref{EST-VD-5} can be prove using similar idea to the proof of
\eqref{IM2}, \eqref{IM3-2}, \eqref{JM2} and \eqref{Jm3}, respectively.
 \end{proof}
We prove several convergences which are of the essence for the proof of our existence result. 
\begin{prop}\label{CONV-lem}
 We can extract a subsequence $(\bvmk,\bdmk)$ from $(\bvm, \bdm)$ such that
 \begin{align}
&  \bvmk \rightarrow \bv \,\,\text{strongly in}\,\,L^{2}(\Omega^\prime
\times [0,T];\h\, ),\label{conv-v-h}\\
& \bvmk \rightarrow \bv \,\,\text{strongly in}\,\,L^4(\Omega^\prime; C([0,T];\ve_{-\beta})\, ), \label{conv-v-beta}\\
& \bdmk \rightarrow \bd \,\,\text{strongly in}\,\,L^{2}(\Omega^\prime
\times [0,T];\h^1\, ), \label{conv-d-h1}\\
& \bdmk \rightarrow \bd \,\,\text{strongly
in}\,\,L^4(\Omega^\prime;C([0,T];\bx_{\beta-\frac12})\, )
\label{conv-d-beta}\\
& \bdmk \rightarrow \bd \text{ strongly in } \mathfrak{S}_2 \,\, \mathbb{P}^\prime-a.s.,\label{CONV-AE-0}\\
& \bdmk \rightarrow \bd \text{ for almost everywhere } (x,t) \text{ and }  \mathbb{P}^\prime-a.s..\label{CONV-AE}
\end{align}
\end{prop}
\begin{proof}
From \eqref{EST-VD-2-B} and Banach-Alaoglu's theorem we infer that
there  exist a subsequence of $\bvmk$ satisfying
\begin{equation}\label{weak-conv-H}
\bvmk \rightarrow \v \text{ weakly in }
L^{2p}(\Omega^\prime;L^2(0,T;\h)),
\end{equation}
for any $p\in [2,\infty)$.  Now let us consider the positive
nondecreasing function $\varphi(x)=x^{2p}$, $ p\in [2,\infty)$,
defined on $\mathbb{R}_+$. The function $\varphi$ obviously
satisfies
\begin{equation}
\lim_{x\rightarrow \infty}\frac{\varphi(x)}{x}=\infty.
\end{equation}
Thanks to the estimate $\mathbb{E}^\prime \sup_{t\in [0,T]}\lve
\bvmk\rve^{2p}\le C$ (see  \eqref{EST-VD-2-B}), we have
\begin{equation}
\sup_{m\ge 1}\mathbb{E}^\prime (\varphi(\lve\bvmk
\rve_{L^2(0,T;\h)}))<\infty,
\end{equation}
which along with the uniform integrability criteria in \cite[Chapter 3, Exercice 6]{Kallenberg} implies that the family $%
\{\lve \bvmk \rve_{L^2(0,T;\h)}:m\in\mathbb{N}\}$ is uniform
integrable with respect to the probability measure.
Thus, we can deduce from Vitali's Convergence Theorem (see, for
instance, \cite[Chapter 3, Proposition 3.2]{Kallenberg}) and
\eqref{SKO-1} that
\begin{equation*}
\me^\prime \lve \bvmk \rve^2_{L^2(0,T;\h)} \rightarrow \me^\prime
\lve \v \rve^2_{L^2(0,T;\h)}.
\end{equation*}
Form this and \eqref{weak-conv-H} we derive that
\begin{equation}
\bvmk \rightarrow \bv \,\,\text{strongly in}\,\,L^{2}(\Omega^\prime
\times [0,T];\h\, ).
\end{equation}%

Thanks to \eqref{SKO-3}-\eqref{SKO-6} in Proposition \ref{SKO-Lem}
and \eqref{EST-VD-2-B} we can use the same argument to show the
convergence \eqref{conv-v-beta}, \eqref{conv-d-h1} and
\eqref{conv-d-beta}.
Hence by the tightness of the laws of $\bdm$ on $\mathfrak{S}_2$ we can extract a subsequence still denoted by $\bdmk$ such that
 \eqref{CONV-AE-0} and \eqref{CONV-AE} hold.
\end{proof}
  The stochastic processes $\mathbf{v}$ and $\bd$ satisfies
the following properties.
\begin{prop}
We have
\begin{equation}\label{weakstinh}
\mathbb{E}^\prime\sup_{t\in \lbrack 0, T]}\lve
\mathbf{v}(t)\rve^p<\infty,
\end{equation}
\begin{equation}\label{weakstinh-2}
\mathbb{E}^\prime\sup_{t\in \lbrack 0, T]}\lve
\mathbf{d}(t)\rve^p_{\h^1}<\infty,
\end{equation}
for any $p\in [2, \infty)$.
\end{prop}
\begin{proof}
One can
argue exactly as in \cite[Proof of (4.12), page 20]{ZB+al-2012}, so we omit the details.
\end{proof}
\begin{prop}\label{WEAK-OP}
Let $d\in \{2,3\}$ and $T\ge 0$. There exist four
processes $\mathfrak{B}_1, \mathfrak{M}\in L^2(\Omega^\prime;L^\frac
4d(0,T;\ve^\ast) )$, $\mathfrak{B}_2\in
L^2(\Omega^\prime;L^\frac d4(0,T;\el^2) )$ and $\mathfrak{f} \in L^{\frac{2N+2}{2N+1}}(\Omega^\prime\times [0,T]\times \MO)$ such that
\begin{align}
& B_{m_k}(\bvmk, \bvmk) \rightarrow \mathfrak{B}_1, \text{ weakly in }L^2(\Omega^\prime;L^\frac d4(0,T;\ve^\ast) ),\label{w-conv-B}\\
& M_{m_k}(\bdmk) \rightarrow \mathfrak{M}, \text{ weakly in }L^2(\Omega^\prime;L^\frac d4(0,T;\ve^\ast) ),\label{w-conv-M}\\
& \tilde{B}_{\mk}(\bvmk,\bdmk) \rightarrow \mathfrak{B}_2, \text{ weakly
in }L^2(\Omega^\prime;L^\frac d4(0,T;\el^2) )\label{w-conv-G1},\\
& f_{\mk}(\bdmk) \rightarrow \mathfrak{f},\text{ weakly in } L^{\frac{2N+2}{2N+1}}(\Omega^\prime\times [0,T]\times \MO).\label{w-conv-f}
\end{align}
\end{prop}
\begin{proof}
Note that Proposition \ref{BOUND-OP} remains valid with $\bdm$ replaced by $\bdmk$. Thus, Proposition \ref{WEAK-OP} follows from  Eqs.
\eqref{EST-VD-3}-\eqref{EST-VD-6} and application of
Banach-Alaoglu's theorem.
\end{proof}
\subsection{Passage to the limit and the end of proof of
Theorem \ref{thm-main-weak-mart} } \label{Proof-of-EXIST}

In
this subsection we prove several convergences which will enable us to
conclude that the limiting objects that we found in Proposition
\ref{SKO-Lem} are in fact a weak martingale solution
 to our problem.

Proposition \ref{CONV-lem} will be used to prove the following
result.
\begin{prop}
 For any process $\Psi\in L^2(\Omega^\prime;L^\frac{4}{4-d} (0,T;\ve))$, the following identity holds
 \begin{equation}\label{ident-B}
  \begin{split}
   \lim_{m\rightarrow \infty}\mathbb{E}^\prime \int_0^T {_{\ve^\ast}}\langle {B}_{\mk}(\bvmk(t), \bvmk(t)), \Psi(t) \rangle_{\ve} dt
   =&\mathbb{E}^\prime\int_0^T {_{\ve^\ast}}\langle \mathfrak{B}_1(t), \Psi(t) \rangle_{\ve} dt,\\
   =&\mathbb{E}^\prime\int_0^T {_{\ve^\ast}}\langle B(\bv(t), \bv(t)), \Psi(t)\rangle_{\ve} dt.
  \end{split}
 \end{equation}
\end{prop}
\begin{proof}
 Let
$$\mathbb{D}=\{\Phi=\sum_{i=1}^k \mathds{1}_{D_i}\mathds{1}_{J_i} \psi_i: D_i\subset \Omega, J_i\subset [0,T] \text{ is measurable}, \psi_i \in \mathcal{V}\}.$$
Owing to \cite[Proposition 21.23]{Zeidler-2A} and the density of
$\mathbb{D}$ in $L^2(\Omega, \mathbb{P};L^\frac{4}{4-d}(0,T;\ve))$
(see, for instance, \cite[Theorem 3.2.6]{Ruston}), in order to
show that the identity \eqref{ident-B} holds it is enough to check
that
\begin{equation*}
\lim_{m\rightarrow \infty} \mathbb{E}^\prime \int_0^T
\mathds{1}_J(t)\mathds{1}_D {_{\ve^\ast}}\langle {B}_{\mk}(\bvmk(t),
\bvmk(t))-B(\bv(t), \bv(t) ), \psi\rangle_{\ve} dt= 0,
\end{equation*}
for any $\Phi=\mathds{1}_D 1_J \psi\in \mathbb{D}$. For this
purpose we first note that
\begin{equation*}
\begin{split}
 {_{\ve^\ast}}\langle {B}_{\mk}(\bvmk(t),\bvmk(t))-B(\bv(t), \bv(t)), \psi\rangle_{\ve}= & {_{\ve^\ast}}\langle \tilde{B}_{\mk}(\bvmk(t)-\bv(t), \bvmk(t)),\psi\rangle_{\ve}\\
 &+ {_{\ve^\ast}}\langle \tilde{B}_{\mk}(\bv(t), \bvmk(t)-\bv(t)),\psi\rangle_{\ve},\\
 =& I_1(t)+I_2(t).
\end{split}
\end{equation*}
The mapping $\langle {B}_{\mk}(\bu, \cdot ),\psi\rangle_{\ve^\ast,\ve}$
from $L^2(\Omega^\prime; L^2(0,T;\ve)$ into
$L^2(\Omega^\prime;L^\frac4d(0,T;\mathbb{R}))$ is linear and
continuous. Therefore if $\bvmk $ converges to $\bv$ weakly in
$L^2(\Omega^\prime; L^2(0,T;\ve)$ then $I_2$ converges to 0 weakly
in $L^2(\Omega^\prime;L^\frac4d(0,T;\mathbb{R}))$. To deal
with $I_1$ we recall that
\begin{equation*}
\begin{split}
 \biggl\lvert\mathbb{E}^\prime \int_0^T  \mathds{1}_J\mathds{1}_D (\omega^\prime,t){_{\ve^\ast}}\langle {B}_{\mk}(\bvmk(t)-\bv(t), \bvmk(t)),\psi\rangle_{\ve} dt \biggr\rvert\le \lve \psi\rve_{\el^\infty}
 \biggl[\mathbb{E}^\prime \int_0^T\lve \nabla \bvmk(t)\rve^2 dt \biggr]^\frac 12  \\ \times \biggl[\mathbb{E}^\prime \int_0^T\lve \bvmk(t)-\bv(t)\rve^2 dt \biggr]^\frac 12.
\end{split}
\end{equation*}
Thanks to \eqref{EST-VD-2-B} and the convergence \eqref{conv-v-h} we
see that the right-hand side of last inequality converges to 0 as
$\mk$ goes to infinity. Hence $I_1$ converges to 0 weakly in
$L^2(\Omega^\prime;L^\frac4d(0,T;\mathbb{R}))$. This ends the
proof of our proposition.
\end{proof}
In the  next proposition we will prove that $\mathfrak{M}$
coincide with $M(\d)$.
\begin{prop}\label{IDENT-M}
Assume that $n<4$. For any process $\Psi\in
L^2(\Omega^\prime;L^\frac{4}{4-d} (0,T;\ve))$, the following
identity holds
 \begin{equation}\label{ident-M}
  \begin{split}
  \mathbb{E}^\prime\int_0^T {_{\ve^\ast}}\langle \mathfrak{M}(t), \Psi(t) \rangle_{\ve} dt
   =\mathbb{E}^\prime\int_0^T {_{\ve^\ast}}\langle M(\d(t)), \Psi(t)\rangle_{\ve} dt.
  \end{split}
 \end{equation}
\end{prop}

\begin{proof}
Since $\prm$ strongly converges to the identity operetor $Id$ in
$L^2(\Omega^\prime;L^\frac d4(0,T;\ve^\ast) )$, it is enough to
show that \eqref{ident-M} is true for $M(\bdmk(t)):=M(\d)$ in place
of $M_{\mk}(\bdmk(t))$.
  By the relation \eqref{INT-md} we have
 \begin{equation}\label{BIL-MD}
 \begin{split}
\langle M(\bdmk(t))-M(\bd(t)), \psi\rangle=\sum_{i,j,k}\int_\MO
\partial_{x_j} \psi^i \partial_{x_i} \bdmk^k(t)
\left(\partial_{x_j}\bdmk^k(t)-\partial_{x_i}\bd^k(t)  \right)dx\\
+\sum_{i,j,k}\int_\MO \partial_{x_j}\psi^i \partial_{x_j}\bd^k(t)
\left(\partial_{x_i}\bdmk^k(t)-\partial_{x_i}\bd^k(t) \right)dx,
 \end{split}
 \end{equation}
for any $\psi\in \mathcal{V}$. From this inequality we infer that
\begin{equation}
 \begin{split}
  \biggl\lvert\mathbb{E}^\prime \int_0^T \mathds{1}_J(\omega^\prime,t){_{\ve^\ast}}\langle M(\bdmk(t))-M(\bd(t)), \psi\rangle_{\ve}dt\le
  C \lve \nabla \psi\rve \biggl[\mathbb{E}^\prime\int_0^T \lve \nabla (\bdmk(t)-\bd(t))\rve^2 dt \biggr]^\frac 12 \\
  \times \biggl(\biggl[\mathbb{E}^\prime \sup_{0\le \le T}\lve \nabla \bdmk(t)\rve^2\biggr]^\frac 12 + \biggl[\mathbb{E}^\prime \sup_{0\le \le T}\lve \nabla \bd(t)\rve^2\biggr]^\frac 12
  \biggr)
 \end{split}
\end{equation}
Owing to the estimate \eqref{EST-VD-2-B} and the convergence
\eqref{conv-d-h1} we infer that the left hand side of the last
inequality converges to 0 as $\mk$ goes to infinity. Now we easily
conclude the proof.
\end{proof}

\begin{prop}\label{conv-G1}
 Let $d\in\{2,3\}$. Then, $$\mathfrak{B}_2= \tilde{B}(\bv, \bd) \text{ in } L^2(\Omega^\prime;L^\frac d4(0,T;\el^2)
 ).$$
\end{prop}
\begin{proof}
 The statement in the proposition is equivalent to say that $\tilde{B}_{\mk}(\bvmk(t), \bdmk(t))$ converges to $\tilde{B}(\bv(t),\bd(t))$ weakly in  $L^2(\Omega^\prime;L^\frac d4(0,T;\el^2) )$.
 To prove this we argue as above, but we consider the set
 $$\mathbb{D}=\{\Phi=\mathds{1}_J\mathds{1}_D \mathds{1}_K: J\subset \Omega^\prime, D\subset[0,T], K\subset \MO \text{ is measurable}\}.$$
For any $\Phi\in \mathbb{D}$ we have
\begin{equation}
\begin{split}
 \biggl\lvert\mathbb{E}^\prime \int_{[0,T]\times Q} \tilde{B}_{\mk}(\bvmk(t),\bdmk(t))-\tilde{B}(\bv(t), \bd(t))\Phi(\omega^\prime,t,x)dxdt\biggr\rvert\\ \le
 \biggl[\mathbb{E}^\prime \int_0^T \lve \bvmk(t)-\bv(t)\rve^2dt \biggr]^\frac 12 \biggl[\mathbb{E}^\prime\int_0^T \lve \nabla \bdmk(t)\rve^2 dt\biggr]^\frac12\\
 +\biggl[\mathbb{E}^\prime \int_0^T \lve \bv(t)\rve^2 dt\biggr]^\frac12 \biggl[\mathbb{E}^\prime \int_0^T \lve \nabla\left(\bvmk(t)-\bv(t)\right)\rve^2dt \biggr]^\frac 12
\end{split}
\end{equation}
Thanks to \eqref{EST-VD-2-B} and \eqref{conv-d-h1} we deduce that
the left hand side of the last inequality converges to 0 as $\mk$ goes
to infinity. This proves our claim.
\end{proof}
The following convergences are also important.

\begin{prop}
Let $r$ be as in Proposition \ref{far-est}, that is $r=\frac{2N+2}{2N+1}\in (1,2)$. Then,
 \begin{equation}
 \mathfrak{f} = f(\d) \text{ in } L^r(\Omega^\prime\times \MO\times [0,T])\label{conv-fm}.
 \end{equation}

\end{prop}
\begin{proof}
To prove \eqref{conv-fm}, first remark that by definition the embedding $\mathbb{X}_{\beta-\frac12} \subset \mathbb{L}^r$ is continuous for any $\beta \in (0,\frac12)$.
The convergence \eqref{CONV-AE} implies that for any $k=0,\ldots,N$
\begin{equation}
 \vert \bdmk \vert^{2k} \bdmk \rightarrow \vert \bd \vert^{2k} \bd \text{ for almost everywhere } (x,t) \text{ and }
 \mathbb{P}^\prime-a.s..\label{CONV-AE-2}
\end{equation}
Since $f(\bdmk)$ is bounded in $L^r(\Omega^\prime\times \MO\times [0,T])$ we can infer from \cite[Lemma 1.3, pp. 12]{LIONS} and the convergence
\eqref{CONV-AE-2} that
$$ f(\bdmk)\rightarrow f(\d) \text{ weakly in } L^r(\Omega^\prime\times \MO\times [0,T]) ,$$
which with the uniqueness of weak limit implies the sought result.
\end{proof}

To simplify notation let us define the processes $\mathcal{M}^1_{k}(t)$
and $\mathcal{M}^2_{k}(t)$ by
\begin{align*}
 \mathcal{M}_{m_k}^1(t)=\bvmk(t)-\bvmk(0)+\int_0^t\biggl(A_1\bvmk(s) +\tilde{B}_{\mk}(\bvmk(s),\bvmk(s))-M_{\mk}(\bdmk(s)) \biggr)ds,\\
\end{align*}
and
\begin{equation*}
 \begin{split}
  \mathcal{M}^2_{\mk}(t)=\bdmk(t)-\bdmk(0)+\int_0^t \biggl(A\bdmk(s)+\tilde{B}_{\mk}(\bvmk(s),\bdmk(s))-f_{\mk}(\bdmk(s))\biggr)ds\\-\int_0^tG^2_{\mk}(\bdmk(s))ds.
 \end{split}
\end{equation*}
\begin{prop}\label{CONV-MART}
Let  $\mathcal{M}_1(t)$ and $\mathcal{M}_{2}(t)$ be defined by
 \begin{align}
& \mathcal{M}^1(t)=\bv(t)-\bv_0+\int_0^t\biggl(A_1\bv(s)+B(\bv(s),\bv(s))-M(\bd(s))\biggr)ds,\label{conv-mart-1}\\
 &  \mathcal{M}^2(t)=\bd(t)-\d_0+\int_0^t\biggl(A\bd(s)+\tilde{B}(\bv(s),\bd(s))-f(\bd(s))\biggr)-\int_0^t G^2(\bd(s)) ds,
   \label{conv-mart-2}
 \end{align}
 for any $t\in (0,T]$. Then, for any  $t\in (0,T]$
 \begin{align*}
  \mathcal{M}_{\mk}^1(t) \text{ converges weakly in } L^2(\Omega^\prime;\ve^\ast)
  \text{ to } \mathcal{M}^1(t),\\
   \mathcal{M}_{\mk}^2(t) \text{ converges weakly in } L^2(\Omega^\prime;\el^2)
   \text{ to } \mathcal{M}^2(t),\\
 \end{align*}
%
%
as $m_k \to \infty$.
\end{prop}
\begin{proof}
 Let $t\in (0,T]$, we first prove that $\mathcal{M}^1_{\mk}(t)\rightarrow \mathcal{M}^1(t)$ weakly in $L^2(\Omega^\prime;\ve^\ast)$ as $m$ goes to infinity. To this end we take an arbitrary $\xi\in L^2(\Omega^\prime;\ve)$. We  have
 \begin{equation*}
  \begin{split}
   \mathbb{E}^\prime\biggl[\langle \mathcal{M}^1_{\mk}(t), \xi\rangle \biggr]= \mathbb{E}^\prime \biggl[\langle \bvmk(t)-\bvmk(0), \xi\rangle-
   \int_0^t \langle \nabla\bvmk(s),\nabla\xi\rangle ds-\int_0^t \langle M_{\mk}(\bdmk(s)),\xi\rangle ds\biggr]\\
   +\mathbb{E}^\prime\biggl[\int_0^t \langle B_{\mk}(\bvmk(s), \bvmk(s)), \rangle ds \biggr]
  \end{split}
 \end{equation*}
 Thanks to the pointwise convergence in $C([0,T];\ve^{-\beta})$, thus in $C([0,T];\ve^\ast)$, and the convergences \eqref{w-conv-B}, \eqref{ident-B}, \eqref{w-conv-M} and
 \eqref{ident-M} we obtain
 \begin{equation*}
  \begin{split}
   \lim_{m\rightarrow \infty} \mathbb{E}^\prime\biggl[\langle \mathcal{M}^1_{\mk}(t), \xi\biggr]= \mathbb{E}^\prime \biggl[\langle \bv(t)-\bv_0, \xi\rangle-
   \int_0^t \langle \nabla\bv(s),\nabla\xi\rangle ds-\int_0^t \langle M(\bd(s)),\xi\rangle ds\biggr]\\
   +\mathbb{E}^\prime\biggl[\int_0^t \langle B(\bv(s), \bv(s)), \xi\rangle ds \biggr],
  \end{split}
 \end{equation*}
which proves the sought convergence.

 Secondly, we prove that for any
$t\in (0,T]$ $\mathcal{M}^2_{\mk}(t)\rightarrow \mathcal{M}^2(t)$
weakly in $L^2(\Omega^\prime;\el^2)$ as $m$ tends to infinity. For this purpose, observe that
 $G^2_{\mk}(\cdot))$ is a linear mapping from
$L^2(\Omega^\prime; C([0,T];\el^2))$ into itself and it satisfies
\begin{equation}
 \mathbb{E}^\prime \lve G^2(\bd(t))\rve^p_{C([0,T];\el^2)}\le c \lve h\rve^2_{\el^\infty} \mathbb{E}^\prime \lve\bd(t)\rve^p_{C([0,T];\el^2)},
\end{equation}
for any $p\in [2,\infty)$.  So it is not difficult to show that
\begin{equation}\label{conv-stG}
 G_{\mk}(\bdmk)\rightarrow G(\bd) \text{ strongly in } L^2(\Omega^\prime;C([0,T];\el^2)).
\end{equation}
 Thanks to this observation, the convergences \eqref{conv-d-beta}, \eqref{conv-fm}, \eqref{w-conv-G1} and Proposition \eqref{conv-G1} we can use the same argument as above
 to show that
 \begin{equation}
  \lim_{m\rightarrow \infty} \langle \mathcal{M}^2_{\mk}(t), \xi\rangle=\langle\mathcal{M}^2(t),\xi\rangle,
 \end{equation}
for any $t\in (0,T]$ and $\xi \in L^2(\Omega^\prime;\el^{2})$.
This completes the proof of Proposition \ref{CONV-MART}.
\end{proof}

Let $\mathcal{N}$ be the set of null sets $\mathcal{F}$ and for
any  $t\ge 0$ and $k\in \mathbb{N}$, let
\begin{align}
& \hat{\mathcal{F}}^{m_k}_t:=\sigma\biggl(\sigma\biggl((\bvmk(s), \bdmk(s), \bar{W}_1^{m_k}(s), \bar{W}_2^{m_k}(s)); s\le t\biggr)\cup \mathcal{N}\biggr),\nonumber \\
& {\mathcal{F}}^\prime_t:=\sigma\biggl(\sigma\Big((\mathbf{v}(s),\d(s), \bar{W}_1(s),\bar{W}_2(s)); s\le t\Big)\cup \mathcal{N}\biggr).\label{filtration}
\end{align}
Let us also define the stochastic processes $\mathfrak{M}^1_{\mk}$ and
$\mathfrak{M}^2_{\mk}$ by
\begin{align*}
& \mathfrak{M}_{\mk}^1(t)= \int_0^t S_{\mk}(\bvmk(s)) d\bwi(s)\\ 
& \mathfrak{M}^2_{\mk}(t)=\int_0^t G_{\mk}(\bdmk(s)) d\bwd(s), 
\end{align*}
for any $t\in [0,T]$.

\noindent
From Proposition \ref{CONV-MART} we see that $(\bv, \bd)$ is a
solution to our problem if we can show that the processes $\wi$
and $\wt$ defined in Proposition \ref{SKO-Lem} are Wiener
processes and $\mathcal{M}^1(t)$, $\mathcal{M}^2(t)$ are
stochastic integrals with respect to $\wi$ and $\wt$ with
integrands $S(\bv(t))$ and $G(\bd(t))$, respectively.
These will be the subjects of the following two propositions.
\begin{prop}\label{CONV-MART-2-A}
	We have the following facts:
	\begin{enumerate}
		\item \label{P1} the stochastic process $\left(\wi(t)\right)_{t\in [0,T]}$ (resp. $\left(\wt(t)\right)_{t\in [0,T]}$) is a  {$\mathrm{K}_1$-cylindrical $\mathrm{K}_2$-valued }Wiener process
		(resp. $\mathbb{R}$-valued standard Brownian motion) on $(\Omega^\prime, \mathcal{F}^\prime, \mathbb{P}^\prime)$.
		\item\label{P2}  For any $s$ and $t$ such that $0\le s<t\le T$, the
		increments $\wi(t)-\wi(s)$ and $\wt(t)-\wt(s)$ are independent of the $\sigma$-algebra generated by $\bv(r), \,\,\bd(r), \,\, \wi(r),\,\, \wt(r) $ for $r\in[0,s]$.
		\item \label{P3} Finally, $\wi$ and
		$\wt$ are mutually independent.
	\end{enumerate}

\end{prop}
\begin{proof}
	We will just establish the proposition for $\wi$, the same method applies to $\wt$. To this end we closely follow \cite{ZB+al-2012}, but see also \cite[Lemma 9.9]{Ondrejat} for an alternative proof.
	
\noindent \textit{Proof of item \eqref{P1}.}	By Proposition \ref{SKO-Lem} the laws of $(\v_{\mk}, \d_{\mk}, W_1, W_2)$ are equal to those of the stochastic process $(\bvmk,\bdmk, \bwi,\bwd)$ on $\mathfrak{S}$.
 Hence it is easy to
 check that $\bwi$ (resp. $\bwd$) form a sequence of {  $\mathrm{K}_1$-cylindrical $\ku_2$-valued } Wiener process (resp. $\mathbb{R}$-valued Wiener process).
  Moreover, for $0\le s<t\le T$ the increments
 $\bwi(t)-\bwi(s)$ (resp. $\bwd(t)-\bwd(s)$) are independent of the $\sigma$-algebra generated by the stochastic process $\left(\bvmk(r),\bdmk(r), \bwi(r), \bwd(r)\right)$,
 for $r\in [0,s]$.

\noindent  Now, we will  check that $\wi$ is a {$\mathrm{K}_1$-cylindrical $\ku_2$-valued} Wiener
 process by showing that the characteristic function of its
 finite dimensional distributions is equal to the characteristic function of a Gaussian random variable. For this purpose let  $k\in \mathbb{N}$ and $s_0=0<s_1<\dots<s_k\le T$ be a partition of $[0,T]$. For each $\bu\in \ku_2^\ast$ we have
 \begin{equation*}
  \mathbb{E}^\prime\biggl[e^{i\sum_{j=1}^k  \fourIdx{}{\ku_2^\ast}{}{\ku_2} {\lb  \,\bu, \bwi(s_j)-\bwi(s_{j-1})\rb}
  }\biggr]
  =e^{-\frac12 \sum_{j=1}^k \,(s_j-s_{j-1})
  \vert \bu\vert^2_{\ku_1} },
 \end{equation*}
where $i^2=-1$. Thanks to \eqref{SKO-5} and the Lebesgue Dominated
Convergence Theorem, we have
\begin{equation*}
\begin{split}
 \lim_{m\rightarrow \infty}\mathbb{E}^\prime\biggl[e^{i\sum_{j=1}^k 
 \fourIdx{}{\ku_2^\ast}{}{\ku_2} {\lb  \,\bu, \bwi(s_j)-\bwi(s_{j-1})\rb}
  }\biggr]= &
 \mathbb{E}^\prime\biggl[e^{i\sum_{j=1}^k 
 \fourIdx{}{\ku_2^\ast}{}{\ku_2} {\lb  \,\bu, \wi(s_j)-\wi(s_{j-1})\rb}
 }\biggr]\\
 = & e^{-\frac12 \sum_{j=1}^k (s_j-s_{j-1})\vert \bu\vert^2_{\ku_1}
 }
\end{split}
\end{equation*}
from which we infer that the finite dimensional distributions of
$\wi$ follow a Gaussian distribution. The same idea can be carried out to prove that
the finite dimensional distributions of $\wt$ are Gaussian.

\noindent \textit{Proof of item \eqref{P2}.} Next,
we prove that the increments $\wi(t)-\wi(s)$ and $\wt(t)-\wt(s)$,
$0\le s<t\le T$ are independent of the $\sigma$-algebra generated
by $\left(\bv(r), \,\,\bd(r), \,\, \wi(r),\,\, \wt(r)\right)$ for
$r\in[0,s]$. To this end, let us consider $\{\phi_j:
j=1,\dots,k\}\subset C_b(\ve^{-\beta}\times \h^1)$ and $\{\psi_j:
j=1,\dots, k\}\subset C_b(\ku_2\times \mathbb{R})$, where for any Banach space $\mathbf{B}$ the space $C_b(\mathbf{B})$ is defined
\begin{equation*}
 C_b(\mathbf{B})=\{\phi: \mathbf{B}\rightarrow \mathbb{R}, \phi \text{ is continuous and bounded}\}.
\end{equation*}
 Also, let $0\le r_1<\dots<r_k\le
s<t\le T$, $\psi\in C_b(\ku_2)$, and $\zeta\in C_b(\mathbb{R})$. For
each $m\in \mathbb{M}$, there holds
\begin{equation*}
\begin{split}
 & \mathbb{E}^\prime\biggl[\biggl(\prod_{j=1}^k \phi_j(\bvmk(r_j),\bdmk(r_j))\prod_{j=1}\psi_j(\bwi(r_j), \bwd(r_j) ) \biggr)\\ & \qquad \times
 \psi(\bwi(t)-\bwi(s))\zeta(\bwd(t)-\bwd(s))  \biggr]\\
 =& \mathbb{E}^\prime\biggl[\prod_{j=1}^k \phi_j(\bvmk(r_i), \bdmk(r_j))\prod_{j=1}\psi_j(\bwi(r_j), \bwd(r_j))\biggr]\\ & \qquad \times
 \mathbb{E}^\prime\left(\zeta(\bwi(t)-\bwi(s))\right)
 \mathbb{E}^\prime\left(\psi(\bwd(t)-\bwd(s))\right).
 \end{split}
\end{equation*}
Thanks to \eqref{SKO-2}, \eqref{SKO-4}, \eqref{SKO-5},
\eqref{SKO-6} and the Lebesgue Dominated Convergence Theorem,
 the same identity is true with $(\bv, \bd, \wi, \wt)$ in place of
 $(\bvmk, \bdmk, \bwi, \bwd)$.
 This completes the proof of the second item of the proposition.

 \noindent \textit{Proof of item \eqref{P3}.}
 One can argue as in the proof of the first part of the proposition, thus we omit the details.

\end{proof}
\begin{prop}\label{CONV-MART-2-B}
	 For each $t\in (0,T]$ we have
		\begin{align}
	&	\mathcal{M}^1(t)=\int_0^t S(\bv(s))d\wi(s) \text{ in } L^2(\Omega^\prime, \ve^\ast),\label{ident-mart-1}\\
	&	\mathcal{M}^2(t)=\int_0^t (\bd(s)\times h) d\wt(s) \text{ in } L^2(\Omega, \bx_{\beta-\frac12}).\label{ident-mart-2}
		\end{align}
\end{prop}
\begin{proof}
	We will only prove the identity \eqref{ident-mart-1} since the same argument can be used to establish \eqref{ident-mart-2}, but see also  \cite{ZB+al-2012} for an alternative proof.
	
Let us fixed $t\in (0,T]$ and for any  $\eps>0$ let $\eta_\eps:\mathbb{R} \to \mathbb{R}$ a standard mollifier with support in $(0,t)$. For $R\in \{S, S_{\mk}\}$, $\bu \in \{\bvmk, \bv\}$	and $s\in (0,t]$ let us set
\begin{align*}
R^\eps(\bu(s))=& (\eta_\eps \star R(\bu(\cdot)))(s)\\
=& \int_{-\infty}^{\infty} \eta_\eps(s-r) R(\bu(r)) dr.
\end{align*}
We recall that, since $R$ is Lipschitz, $R^\eps$ is Lipschitz. We also have the following two important facts, see, for instance, \cite[Section 1.3]{Arendt et al}:
 \begin{enumerate}[(a)]
 	\item for any $p \in [1,\infty)$ there exists a constant $C>0$ such that for any $\eps>0$ we have
 	\begin{equation}\label{Mollif-a}
 	\int_0^t \Vert R^\eps(\bu(s)) \Vert^p_{\mathcal{T}_2(\ku_1,\h)} ds \le C \int_0^t \Vert R(\bu(s)) \Vert^p_{\mathcal{T}_2(\ku_1,\h)} ds.
 	\end{equation}
 	\item For any $p \in [1,\infty)$, we have
 	\begin{equation}\label{Mollif-b}
 \lim_{\eps \to 0}	\int_0^t \Vert R^\eps(\bu(s))-R(\bu(s)) \Vert^p_{\mathcal{T}_2(\ku_1,\h)} ds=0.
 	\end{equation}
 \end{enumerate}
 Now, let $\mathcal{M}^\eps_{\mk}$ and $\mathcal{M}^\eps$ be respectively defined by
 \begin{align*}
 &\mathcal{M}^\eps_{\mk}(t)= \int_0^t S^\eps_{\mk}(\bvmk(s)) d\bwi(s),\\
 &\mathcal{M}^\eps(t)=\int_0^t S^\eps(\bv(s))d\wi(s),
 \end{align*}
	From the It\^o isometry, \eqref{Mollif-a} and some elementary calculations we infer that there exists a constant $C>0$ such that for any $\eps>0$ and $\mk \in \mathbb{N}$
	\begin{align}
	\mathbb{E}^\prime \Vert \mathcal{M}_{\mk}(t)-\mathcal{M}^\eps_{\mk}(t)\Vert^2=& \mathbb{E}^\prime \int_0^t \Vert S(\bvmk(s))-S_{\mk}^\eps(\bvmk(s)) \Vert^2_{\mathcal{T}_2(\ku_1,\h)} ds,\nonumber \\
	 \le & C \mathbb{E}^\prime \int_0^t \Vert S(\bvmk(s))-S(\bv(s)) \Vert^2_{\mathcal{T}_2(\ku_1,\h)} ds\\
	 &\quad \quad \quad + C\mathbb{E}^\prime \int_0^t \Vert S(\bv(s))-S^\eps(\bv(s)) \Vert^2_{\mathcal{T}_2(\ku_1,\h)} ds.\label{epsEq2}
	\end{align}
	From Assumption \ref{HYPO-ST-weak} and \eqref{conv-v-h} we derive that the first term of in the right hand side of the last estimate converges to 0 as $\mk \to \infty$. Owing to \eqref{Mollif-a} and \eqref{weakstinh} the sequence in the second term of    \eqref{epsEq2} is uniformly integrable with respect to the probability measure $\mathbb{P}^\prime$. Thus, from \eqref{Mollif-b} and the Vitali Convergence Theorem we infer that
	$$\lim_{\eps\to 0 }\mathbb{E}^\prime \int_0^t \Vert S(\bv(s))-S^\eps(\bv(s)) \Vert^2_{\mathcal{T}_2(\ku_1,\h)} ds=0 .$$ Hence, for any $t\in (0,T]$
\begin{equation}\label{ILA1}
\lim_{\eps \to 0}\lim_{\mk \to \infty}	\mathbb{E}^\prime \Vert \mathcal{M}_{\mk}(t)-\mathcal{M}^\eps_{\mk}(t)\Vert^2=0.
\end{equation}
In a similar way, we can prove that
\begin{equation} \label{ILA2}
\lim_{\eps \to 0}\lim_{\mk \to \infty}	\mathbb{E}^\prime \Vert \mathcal{M}^2 (t)-\mathcal{M}^\eps (t)\Vert^2=0.
\end{equation}
Next, we will prove that
\begin{equation}
\lim_{\eps \to 0}\lim_{\mk \to \infty}	\mathbb{E}^\prime \Vert \mathcal{M}^\eps_{\mk} (t)-\mathcal{M}^\eps (t)\Vert^2=0.
\end{equation}
To this end, we first observe that
\begin{equation*}
\begin{split}
\mathcal{M}^\eps_{\mk} (t)-\mathcal{M}^\eps (t)= & \int_0^t S^\eps_{\mk}(\bvmk(s)) \bwi(s)- \int_0^t S^\eps_{\mk}(\bv(s)) d\wi(s) \\
& \qquad + \int_0^t S^\eps_{\mk}(\bv(s)) d\wi(s) -\int_0^t S^\eps(\bv(s)) d\wi(s)\\
= & I^{\eps}_{\mk,1} + I^\eps_{\mk,2}.
\end{split}
\end{equation*}
Second,  by integration by parts we derive that
\begin{equation*}
\begin{split}
I^{\eps}_{\mk,1}= & \int_0^t [\eta_\eps^\prime \star S_{\mk}(\bv(\cdot)) ](s) \wi(s) ds-\int_0^t [\eta_\eps^\prime \star S_{\mk}(\bvmk(\cdot))](s) \bwi(s) ds\\
= & \int_0^t [\eta_\eps^\prime \star S_{\mk}(\bvmk(\cdot))](s)[\bwi(s)-\wi(s)] ds + \int_0^t [S^\eps_{\mk}(\bvmk(s))-S^\eps_{\mk}(\bv(s))]d\wi(s)\\
= & J^\eps_{\mk,1} + J^\eps_{\mk,2}.
\end{split}
\end{equation*}
On one hand, by Proposition \ref{CONV-MART-2-A} the processes $\bwi$ and $\wi$ are both {$\mathrm{K}_1$-cylindrical $\ku_2$-valued Wiener processes}, thus, for any integer $p\ge 4$ there exists a constant $C>0$ such that
$$\sup_{\mk \in \mathbb{N}} \mathbb{E}^\prime \sup_{s \in [0,T]} (\Vert \bwi (s)\Vert^p_{\mathrm{K}_2} + \Vert \wi(s) \Vert^p_{\mathrm{K}_2}) \le C Q T^\frac{p}2.$$ Hence, the sequence $\int_0^t \Vert \bwi(s) -\wi(s)\Vert^2_{\mathrm{K}_2} ds$ is uniformly integrable with respect to the probability measure $\mathbb{P}^\prime$, and from \eqref{SKO-5} and the Vitali Convergence Theorem we infer that
$$\lim_{\mk \to \infty} \mathbb{E}^\prime \int_0^t \Vert \bwi(s) -\wi(s)\Vert^2_{\mathrm{K}_2} ds=0.$$
 On the other hand, for any $\eps>0$ there exists a constant $C(\eps)$ such that
 \begin{equation*}
 \begin{split}
 \mathbb{E}^\prime \int_0^t \Vert [\eta^\prime_\eps \star S_{\mk}(\bvmk(\cdot))](s)\Vert^2_{\mathcal{T}_2(\ku_1,\h)} ds \le C(\eps) T \EE \sup_{t\in [0,T]} \Vert S_{\mk}(\bvmk(t))  \Vert^2_{\mathcal{T}_2(\ku_1,\h)},
 \end{split}
 \end{equation*}
 from which along with Assumption \ref{HYPO-ST-weak} and \eqref{EST-VD-2-B-A} we infer that for any $\eps >0$ there exists a constant $C>0$ such that for any $\mk \in \mathbb{N}$ we have
 \begin{equation*}
 \begin{split}
 \mathbb{E}^\prime \int_0^t \Vert [\eta^\prime_\eps \star S_{\mk}(\bvmk(\cdot))](s)\Vert^2_{\mathcal{T}_2(\ku_1,\h)} ds \le C(\eps) T .
 \end{split}
 \end{equation*}
Thus, from these two observation we derive that
$$ \lim_{\eps \to 0}\lim_{\mk \to \infty } \mathbb{E}^\prime \Vert   J^\eps_{\mk,1}  \Vert^2 = 0.$$
In a straightforward way we can also  show that
$$ \lim_{\eps \to 0}\lim_{\mk \to \infty }\left(  \mathbb{E}^\prime \Vert   J^\eps_{\mk,2}  \Vert^2 + \mathbb{E}^\prime \Vert I^\eps_{\mk,2} \Vert^2  \right)=0.$$
Therefore, we have just proved that
\begin{equation}\label{ILA3}
\lim_{\eps \to 0}\lim_{\mk \to \infty } \mathbb{E}^\prime \Vert \mathcal{M}^\eps_{\mk} (t)-\mathcal{M}^\eps (t) \Vert^2=0.
\end{equation}
The identities \eqref{ILA1}, \eqref{ILA2} and \eqref{ILA3} imply that for any $t\in (0,T]$
\begin{equation}\label{ILA3-A}
\lim_{\mk \to \infty}\mathbb{E}^\prime \Vert \mathcal{M}^1_{\mk}(t) -\mathcal{M}^1(t) \Vert^2=0.
\end{equation}
To conclude the proof of the proposition we need to show that $\mathbb{P}^\prime$-a.s.
\begin{equation}\label{ILA4}
\mathcal{M}^1_{\mk}(t) -\int_0^t S_{\mk}(\bvmk(s))d\bwi(s)=0,
\end{equation}
for any $t\in (0,T]$. To this end, let $\mathcal{M}^1_m$ and $\mathcal{M}^\eps_m$ be the analog of $\mathcal{M}^1_{\mk}$ and $\mathcal{M}^\eps_{\mk}$ with $\mk$ and $\bvmk$ replaced by $m$ and $\bvm$, respectively. For any $\bu \in L^2(0,T; \ve^{\ast})$ we set
$$ \varphi(\bu) =\frac{\int_0^T \Vert \bu(s) \Vert^2_{\ve^\ast} ds }{1+ \int_0^T \Vert \bu(s) \Vert^2_{\ve^\ast}ds  }.$$
Since $(\bvmk, \bdmk, \bwi)$ and $(\bvm, \bdm, W_1)$ have the same law and $\varphi(\cdot)$ is continuous as a map from $\mathfrak{S}_1\times \mathfrak{S}_2\times C([0,T];\rK_1)$ into $\mathbb{R}$, we infer that
$$\mathbb{E} \varphi(\mathcal{M}^1_m - \mathcal{M}^\eps_m)=\mathbb{E}^\prime \varphi(\mathcal{M}^1_{\mk} - \mathcal{M}^\eps_{\mk}).$$
Note that arguing as above we can show that  as $\eps \to 0$ we have
$$\mathbb{E} \varphi(\mathcal{M}^1_m - \mathfrak{M}^1_m)=\mathbb{E}^\prime \varphi(\mathcal{M}^1_{\mk} - \mathfrak{M}^1_{\mk}),$$
where  $$\mathfrak{M}^1_m(\cdot) =\int_0^{\cdot} S_m(\bvm(s))dW_1(s).$$ Since $\bvm$ and $\bdm$ are the solution of the Galerkin approximation, we have $\mathbb{P}$-a.s. $\varphi(\mathcal{M}^1_m - \mathfrak{M}^1_m)=0$, from which we infer that
 $$ \mathbb{E}^\prime \varphi(\mathcal{M}^1_{\mk} - \mathcal{M}_{\mk})=0.$$ This last identity implies that $\mathbb{P}^\prime$-a.s.
  $ \mathcal{M}^1_{\mk}(t) - \mathcal{M}_{\mk}(t)=0 $ for almost everywhere $t\in (0,T]$. Since the maps $ \mathcal{M}^1_{\mk}(\cdot)$ $ \mathcal{M}_{\mk}(\cdot)$ are continuous in $\ve^\ast$ and agree for almost everywhere $t \in (0,T]$, necessarily they agree for all $t\in (0,T]$. Thus, we have proved the identity \eqref{ILA4} which along with \eqref{ILA3-A} implies the desired equality \eqref{conv-mart-1}.
\end{proof}
Now we give the promised proof of the existence of weak martingale solution.
\begin{proof}[Proof of Theorem Theorem \ref{thm-main-weak-mart}]
Endowing the complete probability space $(\Omega^\prime, \mathcal{F}^\prime, \mathbb{P}^\prime)$ with the filtration $\mathbb{F}^\prime=(\mathcal{F}_t^\prime)_{t\ge 0}$ which satisfies the usual condition, and combining Propositions \ref{CONV-MART},
\ref{CONV-MART-2-A}  and \ref{CONV-MART-2-B} we have just constructed a complete filtered probability space and stochastic  processes $\bv(t), \bd(t),$ $\wi(t), \wt(t)$ which satisfy all the items of Definition \ref{WEAK-MART}.
\end{proof}

\subsection{Proof of the pathwise uniqueness of the weak solution in the 2-D case} \label{Proof-of-Uniq}
This subsection is devoted to the proof of the uniqueness stated in Theorem \ref{UNIQUENESS}. Before proceeding to the the actual proof of this pathwise uniqueness, we state and prove the following lemma.
 \begin{lem}\label{POLY-UNI}
 	For any $\alpha_8>0$ and $\alpha_9>0$ there exist $C(\alpha_8)>0$, $C_1(\alpha_9)>0$ and $C_2(\alpha_9)>0$ such that
 	\begin{align}\label{UNI-f1}
 	\lvert \langle f(\d_1)-f(\d_2), \d_1-\d_2\rangle \rvert\le & \alpha_8 \Vert \nabla \d_1-\nabla \d_2\Vert^2 + C(\alpha_8) \Vert \d_1-\d_2\Vert^2 \varphi(\d_1,\d_2),
 	\end{align}
 	\begin{equation}\label{UNI-f2}
 	\begin{split}
 	\lvert \langle f(\d_1)-f(\d_2), \Delta \d_1-\Delta \d_2\rangle \rvert\le \alpha_9 \Vert \Delta \d_1-\Delta \d_2\Vert^2 + C_1(\alpha_9)
 	\Vert \nabla \d_1-\nabla \d_2\Vert^2 \varphi(\d_1,\d_2) \\ +C_2(\alpha_9) \Vert \d_1-\d_2\Vert^2 \varphi(\d_1,\d_2),
 	\end{split}
 	\end{equation}
 	
 	where $$ \varphi(\d_1,\d_2):= C\left(1+ \Vert \d_1\Vert^{2N}_{\el^{4N+2}} +\Vert \d_2\Vert^{2N}_{\el^{4N+2}}\right)^2.$$
 \end{lem}
 \begin{proof}[Proof of Lemma \ref{POLY-UNI}]
 It is enough to prove the estimate \eqref{UNI-f1} for the particular case $f(\bd):=a_N  \lvert \bd \rvert^{2N} \bd$.
 For this purpose we recall that
 $$\lvert \bd_1\rvert^{2N}\bd_1-\lvert \bd_2\rvert^{2N}\bd_2= \vert \bd_1\vert^{2N}(\bd_1-\bd_2)+
  \bd_2(\vert\bd_1\vert-\vert\bd_2\vert)(\sum_{k=0}^{2N-1}\vert \bd_1\vert^{2N-k-1}\vert \bd_2\rvert^{k} ),$$
  from which we easily deduce that
 	\begin{equation*}
 	\lvert \langle f(\d_1)-f(\d_2), \d_1-\d_2\rangle \rvert\le C \int_\MO (1+\vert \d_1\vert^{2N}+\vert \d_2\vert^{2N} ) \lvert \d_1-\d_2\rvert^2 dx,
 	\end{equation*}
 	for any $\bd_1, \bd_2\in\el^{2N+2}(\mo)$. Now, invoking the H\"older, Gagliardo-Nirenberg and Young inequalities we infer that
 	\begin{align*}
 	\lvert \langle f(\d_1)-f(\d_2), \d_1-\d_2\rangle \rvert\le C \Vert \d_1-\d_2\Vert_{\el^4}^2(1+ \Vert \d_1\Vert_{\el^{4N+2}}^{2N}+\lVert \d_2\Vert_{\el^{4N+2}}^{2N} )\\
 	\le C \Vert \d_1-\d_2\Vert \Vert \nabla\left(\d_1-\d_2\right) \Vert  (1+ \Vert \d_1\Vert_{\el^{4N+2}}^{2N}+\lVert \d_2\Vert_{\el^{4N+2}}^{2N} )\\
 	\le \alpha_8 \Vert \nabla \left(\d_1-\d_2\right)\Vert^2 +C(\alpha_8)\Vert \d_1-\d_2 \Vert^2  (1+ \Vert \d_1\Vert_{\el^{4N+2}}^{2N}+\lVert \d_2\Vert_{\el^{4N+2}}^{2N} )^2.
 	\end{align*}
 	The last line of the above chain of inequalities implies \eqref{UNI-f1}.
 	
 	Using the fact that $\h^1\subset \el^{4N+2}$ for any $N\in \mathbb{N}$ and the same argument as in the proof of \eqref{UNI-f1} we derive that
 	\begin{align*}
 	\lvert \langle f(\d_1)-f(\d_2), \Delta\d_1-\Delta \d_2\rangle \rvert\le C \int_\MO (1+\vert \d_1\vert^{2N}+\vert \d_2\vert^{2N} )
 	\lvert \d_1-\d_2\rvert \vert \Delta\left(\d_1-\d_2\right)\vert dx\\
 	\le C  \Vert \d_1-\d_2\Vert_{\el^{4N+2}}\Vert \Delta[\d_1-\d_2]\Vert(1+ \Vert \d_1\Vert_{\el^{4N+2}}^{2N}+\lVert \d_2\Vert_{\el^{4N+2}}^{2N} )\\
 	\le C \Vert \d_1-\d_2\Vert_{\h^1} \Vert \Delta\left[\d_1-\d_2\right] \Vert  (1+ \Vert \d_1\Vert_{\el^{4N+2}}^{2N}+\lVert \d_2\Vert_{\el^{4N+2}}^{2N} )\\
 	\le \alpha_9 \Vert \Delta\left[\d_1-\d_2\right]\Vert^2 +C(\alpha_9)\Vert \d_1-\d_2 \Vert_{\h^1}^2
 	(1+ \Vert \d_1\Vert_{\el^{4N+2}}^{2N}+\lVert \d_2\Vert_{\el^{4N+2}}^{2N} )^2.
 	\end{align*}
 	From the last line we easily deduce the proof of \eqref{UNI-f2}.
 \end{proof}
Now, we give  the promised proof of the uniqueness of our solution.
 \begin{proof}[Proof of Theorem \ref{UNIQUENESS}]
 Let $\v=\v_1-\v_2$ and $\d=\d_1-\d_2$. These processes
satisfy $(\v(0),\d(0))=(0,0)$ and the stochastic equations
\begin{equation*}
\begin{split}
d\v(t)+\biggl(\rA
\v(t)+B(\v(t),\v_1(t))+B(\v_2(t),\v(t))\biggr)dt=-\biggl(M(\d(t),\d_1(t))+M(\d_2,\d)\biggr)dt\\+[S(\v_1(t))-S_2(\bv_2(t))]d\W(t),\\
\end{split}
\end{equation*}
and
\begin{equation*}
\begin{split}
d\d(t)+\biggl(\rA_1\d(t)+\tilde{B}(\v(t),\d_1(t))+\tilde{B}(\v_2(t),\d(t))\biggr)dt=
-[f(\d_2(t))-f(\d_1(t))]dt \\+\frac12
G^2(\d(t))dt+G(\d(t))dW_2(t).
\end{split}
\end{equation*}
Firstly, recall that from Young's inequality,  \eqref{B4},
\eqref{IM3-1} and  \eqref{EST-G1} we can derive that for any
$\alpha_1>0 $, $\alpha_2>0$, $\alpha_3>0$, $\alpha_4>0$, $\alpha_5>0$ and $\alpha_7>0$ there exist
$C(\alpha_1)>0$, $C(\alpha_2,\alpha_3)>0$, $C(\alpha_7,\alpha_4)>0$ and $C(\alpha_5)>0$ such that
\begin{align}
\lvert \langle B(\v,\v_1),\v\rangle\rvert\le& \alpha_1 \lve \nabla
\v\rve^2+C(\alpha_1)\lve \v_1\rve^2 \lve \nabla \v_1\rve^2 \lve
\v\rve^2,\label{UNIQ-1}\\
\lvert \langle M(\d_2,\d),\v\rangle\rvert\le &\alpha_2 \lve \nabla \v
\rve^2+\alpha_3 \lve \Delta \d\rve^2+C(\alpha_2,\alpha_3)\lve
\nabla \d_2\rve^2 \lve \Delta \d_2\rve^2 \lve
\nabla \d\rve^2,\nonumber \\
\lvert \langle M(\d,\d_1),\v\rangle\rvert\le &\alpha_7 \lve \nabla \v
\rve^2+\alpha_4 \lve \Delta \d\rve^2+C(\alpha_7,\alpha_4)\lve
\nabla \d_1\rve^2 \lve \Delta \d_1\rve^2 \lve
\nabla \d\rve^2,\nonumber \\
\lvert \langle \tilde{B}(\v_2,\d), \Delta \d\rangle \rvert\le &
\alpha_5 \lve \Delta \d\rve^2+C(\alpha_5)\lve \v_2\rve^2 \lve \nabla
\v_2\rve^2 \lve \nabla \d\rve^2.\nonumber
\end{align}
From H\"older's inequality, Gagliardo-Nirenberg's inequality
\eqref{GAG-l4} and the Sobolev embedding $\h^2\subset \el^\infty$
we infer that for any $\alpha_6>0$ there exists $C(\alpha_6)>0$
such that
\begin{align*}
\lvert \langle \tilde{B}(\v,\d_1),\d\rangle \rvert \le & \lve
\v\rve
\lve \nabla \d_1\rve \lve \d\rve_{\el^\infty},\\
\le &\alpha_6 \lve \Delta \d \rve^2+ C(\alpha_6) \lve \v\rve^2\lve
\nabla \d_1\rve^2.
\end{align*}
 From the proof of Proposition \ref{EST1} we see that there exists
 a constant $C(h)>0$ such that
 \begin{align*}
\lve \nabla G(\d)\rve^2\le & C(h)(\lve \nabla \d\rve^2+\lve
\d\rve^2),\\
 \lve \nabla G^2(\d)\rve^2\le & C(h)(\lve \nabla \d\rve^2+\lve
\d\rve^2),\\ 
\lvert \langle \nabla G^2(\d), \nabla \d\rangle \rvert\le &
C(h)(\lve \nabla \d\rve^2+\lve \d\rve^2).\nonumber
 \end{align*}
 Owing to the Lispschitz property of $S$ we have
 \begin{align*}
 \lve S(\v_1)-S(\v_2)\rve^2_{\mathcal{T}_2}\le C \lve \v\rve^2.
 \end{align*}
Now, let $\varphi(\d_1,\d_2)$ be as in Lemma \ref{POLY-UNI} and $$\Psi(t)=e^{-\int_0^t (\psi_1(s)+\psi_2(s)+\psi_3(s) )ds}, \text{ for any } t>0,$$ where
\begin{align*}
\psi_1(s):= & C(\alpha_1)\lve \v_1(s)\rve^2 \lve \nabla \v_1(s)\rve^2+
C(\alpha_6)\lve \nabla \d_1(s) \rve^2,\\
\psi_3(s):= &[C(\alpha_8) +C_2(\alpha_9) ] \varphi(\d_1(s),\d_2(s)),
\end{align*}
and
\begin{equation*}
 \begin{split}
\psi_2(s):=C(\alpha_2,\alpha_3)\lve \nabla \d_2(s)\rve^2\lve \Delta
\d_2(s)\rve^2+C(\alpha_7,\alpha_4)\lve \nabla \d_1(s)\rve^2 \lve \Delta \d_1(s) \rve^2\\ +C(\alpha_5) \lve \v_2(s)\rve^2\lve \nabla \v_2(s)\rve^2
+C_1(\alpha_9) \varphi(\d_1(s),\d_2(s)).
 \end{split}
\end{equation*}
Now applying It\^o's formula to $\lve \d(t)\rve^2$ (as in proof
of Proposition \ref{EST0}) and $\Psi(t)\lve \d(t)\rve^2$ yield
\begin{equation*}
\begin{split}
d[\Psi(t)\lve \bdt\rve^2]=-2\Psi(t)\lve \nabla \bdt\rve^2dt-2 \Psi(t)
\langle \tilde{B}(\bvt,\bd_1(t)),\bdt\rangle\\-2 \langle
f(\bd_2(t))-f(\bd_2(t)),\bdt\rangle dt+\Psi^\prime(t) \lve \bdt\rve^2.
\end{split}
\end{equation*}
Using the same argument we can show that $\Psi(t)\lve \nabla
\bdt\rve^2$ and $\Psi(t)\lve \bvt\rve^2$ satisfy
\begin{equation*}
\begin{split}
d[\Psi(t)\lve \nabla \bdt\rve^2]=\Psi(t)\biggl( -\lve \Delta
\bdt\rve^2+\langle \tilde{B}(\bvt,\bd_1(t))+\tilde{B}(\bv_2(t),\bdt), \Delta \bdt\rangle\biggl)dt\\
\quad +\Psi(t)\biggl(2  \langle
f(\bd_2(t))-f(\bd_1(t)),\Delta \bdt\rangle +\langle \nabla G^2(\bdt),\nabla
\bdt\rangle \biggr)dt\\+\lve G(\d(t))\rve^2 dt+ \Psi^\prime(t)\lve \nabla \bdt\rve^2 dt
+2\Psi(t)\langle \nabla G(\bdt),\nabla \bdt\rangle dW_2(t),
\end{split}
\end{equation*}
 and
 \begin{equation*}
 \begin{split}
 d[\Psi(t)\lve \bvt\rve^2]=-2\Psi(t)\biggl(\lve\nabla \bvt\rve^2
 +\langle B(\bvt,\bv_1(t))+M(\bdt,\bd_1(t))\biggr)dt\\
 -2\Psi(t)M(\bd_2(t),\bdt),\bvt\rangle dt+\Psi(t)\lve S(\bv_1(t))-S(\bv_2(t))\rve^2_{\mathcal{T}_2} dt+\Psi^\prime(t) \lve
 \bvt\rve^2 dt \\+2\Psi(t)\langle \bvt,[S(\bv_1(t))-S(\bv_2(t))]\rangle d\W(t).
\end{split}
 \end{equation*}
 Summing up these last three equalities side by side and using the inequalities \eqref{UNIQ-1}-\eqref{UNI-f2} imply
  \begin{equation*}
\begin{split}
d[\Psi(t)\left(\lve \bvt\rve^2+\lve \bdt\rve^2+\lve \nabla
\bdt\rve^2\right)]+2\Psi(t)\Big[\lve \nabla \bvt\rve^2+\lve \nabla
\bdt\rve^2+\lve \Delta \bdt\rve^2\Big]dt\\
\le
2\Psi(t)\biggl(C \biggl[\lve \bvt\rve^2+\lve \bdt\rve^2+\lve
\nabla \bdt\rve^2 \biggr]dt+\langle \nabla G(\bdt),\nabla \bdt\rangle
dW_2(t)\biggr)\\ +2\Psi(t)\biggl(\langle
\bvt,[S(\bv_1(t))-S(\bv_2(t))]\rangle d\W(t)
+\big[\alpha_9+ \sum_{j=3}^6 \alpha_j\big]\lve \Delta
\bdt\rve^2\biggr)\\
+\Psi(t)\Big[\psi_2(t) \lve \nabla \bdt\rve^2
+\psi_1(t) \lve \bvt \rve^2 + \psi_3(t) \lve \bdt \rve^2 \Big] dt\\
+\Psi^\prime(t)\biggl(\lve \bvt \rve^2+\lve \bdt\rve^2+\lve \nabla
\bdt\rve^2\biggr)dt\\
+(\alpha_1+\alpha_2+\alpha_7)\lve \nabla \bvt\rve^2+\alpha_8 \lve \nabla \bdt\rve^2  dt.
\end{split}
 \end{equation*}
 Notice
 that by the choice of $\Psi$ we have
 \begin{equation*}
\begin{split}
\Psi(t)\Big[\psi_2(t) \lve \nabla \bdt\rve^2
+\psi_1(t) \lve \bvt \rve^2 + \psi_3(t) \lve \bdt \rve^2 \Big]
+\Psi^\prime(t)\biggl(\lve \bvt \rve^2+\lve \bdt\rve^2+\lve \nabla
\bdt\rve^2\biggr) \le 0.
\end{split}
 \end{equation*}
Hence by choosing $\alpha_j=\alpha_9=\frac1{10}$, $j=3,\ldots,6,$ $\alpha_i=\alpha_7=\frac16, i=2,3$ and $\alpha_8=\frac12$ we see that
\begin{equation*}
\begin{split}
d[\Psi(t)\left(\lve \bvt\rve^2+\lve \bdt\rve^2+\lve \nabla
\bdt\rve^2\right)]+\Psi(t)\Big[\lve \nabla \bvt\rve^2+\lve \Delta
\bdt\rve^2+\lve \nabla \bdt\rve^2\Big]dt\\
\le
2\Psi(t)\biggl(C \biggl[\lve \bvt\rve^2+\lve \bdt\rve^2+\lve
\nabla \bdt\rve^2 \biggr]dt+\langle \nabla G(\bdt),\nabla \bdt\rangle
dW_2(t)\\
+\langle \bvt,[S(\bv_1(t))-S(\bv_2(t))]\rangle d\W(t) \biggr)
\end{split}
\end{equation*}
Next integrating and taking the mathematical expectation yield
\begin{equation*}
\begin{split}
\mathbb{E}\biggl[\Psi(t)\left(\lve \bvt\rve^2+\lve \bdt\rve^2+\lve
\nabla \bdt\rve^2\right)\biggr]+\mathbb{E} \int_0^t \Psi(s)\Big[\lve
\nabla \bvs\rve^2+\lve \Delta \bds\rve^2+2\lve \nabla
\bds\rve^2\Big]ds\\\le C \int_0^t
\mathbb{E}\biggl[\Psi(s)\left(\lve \bvs\rve^2+\lve \bds\rve^2+\lve
\nabla \bds\rve^2 \right)\biggr]ds.
\end{split}
\end{equation*}
Now we can easily conclude the proof by making use of Gronwall's lemma.
 \end{proof}

\section{Existence and uniqueness of local and global strong Solution}\label{SLC-Sect4}
In this section we are interested in the strong solution to the system \eqref{ABS-v1}-\eqref{ABt-d1}. Here strong solution is both in PDEs and Stochastic Calculus sense.
One of the main reult, the existence of a local maximal solution, of this section is a corollary of a general framework that will be introduced in Section \ref{ABST-STRONG}.
In this section we will check that the system \eqref{ABS-v1}-\eqref{ABt-d1} fits into this general framework and hence establish the existence of local maximal
solution of the stochastic nematic liquid crystal. We also prove that the maximal solution turns out to be a global one in the two dimensional case.
For these ends we will introduce some additional notations and prove several key inequalities related to the nonlinear terms of the stochastic system
\eqref{ABS-v1}-\eqref{ABt-d1}.

Hereafter, we put
 \begin{equation}\label{eqn-spaces}
 H=\h\times \bx_{0}, \;\; V=\ve_{\frac12}\times \bx_{\frac12} \mbox{ and } E=\ve_1 \times \bx_{1}.
 \end{equation}
Next we denote by $\{\mathbb{S}_1(t)\}_{t\geq 0}$ the analytic semigroup
generated by
 $-\rA$ on $\h$ where $\rA$ is  the Stokes operator.
The operator $-\rrA$ is the generator of an analytic
semigroup $\{\mathbb{T}(t)\}_{t\geq 0}$ on $\el^2$  satisfying
\begin{equation}\label{sem-rep}
\mathbb{T}(t) \bu=\sum_{k\in \mathbb{N}} e^{-\lambda_k t}u_k \phi_k,\;\;\; \bu=\sum_{k\in \mathbb{N}}u_k \phi_k \in \el^2.
\end{equation}
 By using the representation
\eqref{frac-space} we can show without any difficulty that
the space $\bx_{0}$ is invariant with respect to this semigroup. The restriction of $\{\mathbb{T}(t)\}_{t\geq 0}$ to $\bx_{0}$ is also an analytic semigroup which will be denoted in the sequel by $\{\mathbb{S}_2(t)\}_{t\geq 0}$.
 The minus infinitesimal
generator $\trrA$ of $\{\mathbb{S}_2(t)\}_{t\geq 0}$ is the part
of $\rrA$ on $\bx_{0}$, that is,
\begin{align*}
D(\trrA)=\{\bu \in D(\rrA): \rrA\bu\in \bx_{0}\},\;\;
\trrA\bu=\rrA\bu \text{ for any } \bu \in D(\trrA).
\end{align*}
Note that $\bx_{1}\subset D(\trrA).$
With all the above notation, the
 stochastic equations for nematic
 liquid crystal \eqref{eqn-SLQE-v}-\eqref{eqn-SLQE-d} can be rewritten as the following stochastic evolution equation in the space $H$,
\begin{equation}\label{ABSTRACT-LC}
d\y(t) +\mathbf{A}\y(t) dt+\mathbf{F}(\y(t))
dt+\mathbf{L}(\y(t))dt=\mathbf{G}(\y(t)) d{W}(t),
\end{equation}
where, for  $\y=(\bv, \d)\in E$ and $k=(k_1,k_2)\in \rK$,

\begin{equation}\label{eqn-def-A-F}
\A\y=\begin{pmatrix} \rA & 0\\
0 & \rrA
\end{pmatrix}\begin{pmatrix}
\bv\\ \bd
\end{pmatrix},\;\;
\f(\y)=\begin{pmatrix} B(\bv,\bv)+M(\bd)\\
\tilde{B}(\bv,\bd)+f(\bd)
\end{pmatrix},
 \end{equation}

  \begin{equation}\label{eqn-def-L-G}
\mathbf{L}(\y)=\begin{pmatrix} 0\\ -\frac 1 2 G^2(\bd)
\end{pmatrix}, \mathbf{G}(\y)k=\begin{pmatrix}S(\bu)k_1\\
G(\bd)k_2 \end{pmatrix}.
 \end{equation}

Below we will also use the $C_0$ analytic semigroup  $\{\mathbb{S}(t)\}_{t\geq 0}$ on  $H=\h\times \bx_{0}$
defined by
\begin{equation*}
\mathbb{S}(t)\begin{pmatrix}\bv\\\d
\end{pmatrix}=\begin{pmatrix}\mathbb{S}_1(t)\bv\\
\mathbb{S}_2(t)\d
\end{pmatrix}, \;\; (\bv,\d)\in H.
\end{equation*}
 Its
infinitesimal generator  is  $-\mathbf{A}$, where $\mathbf{A}$ is
defined in  \eqref{eqn-def-A-F}.
Some properties of $\{\mathbb{S}(t): t\ge 0\}$ will be given in Lemmata \ref{SEM-1}-\ref{SEM-3}.

%
\begin{assum}\label{HYPO-ST}
We assume that $S: \h\to \mathcal{T}_2(\rK_1,\ve)$ is a globally Lipschitz map. In particular,
there exists $\ell_5\geq 0$
such that
\begin{equation*}
\lve  S(\bu)\rve^2_{\mathcal{T}_2(\rK_1,\ve)}\leq \ell_5 (1+\lve \bu \rve^2),\;\; \mbox{ for any } \bu \in \h.
\end{equation*}
\end{assum}
Let us recall the following notations/definition which are borrowed
from \cite{Kunita-90} (see also \cite{Brz+Elw_2000}).

 \begin{Def}
 For a probability  space  $(\Omega,\mathcal{F}, \mathbb{
P})$      with  given right-continuous filtration $
\mathbb{F}=\big(\mathcal{F}_t\big)_{t\ge  0}$,   a  stopping time
$\tau$ is called accessible iff there exists an  increasing
sequence of stopping times  $\tau_n$ such that a.s. $\tau_n <
\tau$ and $\lim_{n\to \infty} \tau_n =\tau$, see \cite{Kunita-90}.
\end{Def}
\textbf{Notation}.  For  a  stopping  time  $\tau$  we  set
\[ \Omega_t(\tau) =
\{ \omega \in \Omega : t < \tau(\omega)\},
\]
\[
 [0,\tau)\times \Omega =
\{ (t,\omega) \in [0,\infty)\times \Omega: 0\le t < \tau(\omega)
\}.
\]

\begin{Def}
A process $\eta : [0,\tau) \times \Omega \to X$ (we will also write $ \eta(t)$, $t< \tau$), where $X$ is    a
metric space,     is
admissible iff
\begin{trivlist}
\item[(i) ] it is adapted, i.e.  $\eta|_{\Omega_t(\tau)}: \Omega_t(\tau) \to
X$ is $\mathcal{F}_t$ measurable, for any $t\ge 0$; \item[(ii)]
for almost all $\omega \in \Omega$, the function $[0,
\tau(\omega))\ni t \mapsto \eta(t, \omega) \in X$ is continuous.
\end{trivlist}
A process $\eta : [0,\tau) \times \Omega \to X$ is  progressively
measurable  iff, for any $t> 0$,  the  map $$[0,t\wedge \tau)
{
\times}  \Omega \ni  (s,\omega) \mapsto  \eta(s,\omega) \in  X$$ is
$\mathcal{B}_{t\wedge \tau} \times \mathcal{F}_{t\wedge \tau}$ measurable.\\
Two  processes $\eta_i: [0,\tau_i)   \times \Omega  \to X$,
$i=1,2$ are called equivalent (we will write $(\eta_1,\tau_1) \sim
(\eta_2,\tau_2)$)   iff $\tau_1=\tau_2$  a.s. and  for any  $t>0$
the  following holds
\[
\eta_1(\cdot,\omega)= \eta_2(\cdot,\omega) \mbox{ on } [0,t]
\]
for a.a. $\omega \in \Omega_t(\tau_1)\cap \Omega_t(\tau_2)$.
\\
Note that if processes  $\eta_i  :  [0,\tau_i)  \times  \Omega \to
X$,  $i=1,2$ are admissible and for any $t>0$
$\eta_1(t)|_{\Omega_t(\tau_1)}= \eta_2(t)|_{\Omega_t(\tau_2)}$
a.s. then they are also equivalent.
\end{Def}

We now define some concepts of solution to  \eqref{ABS-SPDE}, see \cite[Def. 4.2]{Brz+Millet_2012} or \cite[Def.
2.1]{Mikulevicius}.

\begin{Def}\label{def-local solution} Assume that a $V$-valued  $\mathcal{F}_0$ measurable random variable  $\y_0$ with $\mathbb{E} \Vert \y_0\Vert^2<\infty$ is given. A local mild
solution to problem \eqref{ABS-SPDE-strong} (with the initial time
$0$) is a pair $(\y,\tau)$ such that
\begin{enumerate}
\item $\tau$ is an accessible $\mathbb{F}$-stopping time, \item
$\y: [0,\tau)\times \Omega \to V$ is an admissible process, \item there
exists an approximating sequence  $(\tau_m)_{m\in \mathbb{N}}$ of
 finite $\mathbb{F}$-stopping times  such that $\tau_m \toup \tau$
a.s. and, for every $m\in \mathbb{N}$ and $t\ge 0$,
 we have
\begin{eqnarray}\label{eq-locsol_01}
&&\hspace{-3truecm}\lefteqn{\mathbb{E}\Big(  \sup_{s\in [0,t\wedge \tau_m]} \Vert
\y(s)\Vert^2 +\int_0^{t\wedge \tau_m} \vert \y(s)\vert_E^2 \,
ds\Big)<\infty,}
\\
\label{eq-locsol_01-b} \y(t\wedge \tau_m)&=& \mathbb{S}(t\wedge
\tau_m)\y_0-\int_0^{t\wedge \tau_m}
\mathbb{S}(t\wedge\tau_m-s)[ \mathbf{F}(\y(s))+\mathbf{L}(\y(s)] ds\\
\nonumber &+&\int_0^\infty
\mathds{1}_{[0,t\wedge\tau_m)}\mathbb{S}(t\wedge\tau_m-s)\mathbf{G}(\y(s))\,
d{W}(s).
\end{eqnarray}
\end{enumerate}
Along the lines of the paper \cite{Brz+Elw_2000}, we say that a
local solution $\y(t)$, $t < \tau$ is  global iff
$\tau=\infty$ a.s.
\end{Def}

 We will check that the nonlinear map $\f$ satisfies the assumption
 of Proposition \ref{prop-global Lipschitz-F} with $H=\h\times \bx_{0}$, $V=\ve\times D(\rrA)$ and $E=D(\rA)\times \bx_{1}$.  For
 this purpose we will prove several estimates.
 \begin{lem}\label{Local-LIP-Lem}
There exist some positive constants $c_1$ and $c_2$ such that for
any $(\bv_i, \bd_i)\in E$, i$=1,2$ we have, with $a=\frac d4$,
\begin{equation}\label{local-Lip-F-1}
\begin{split}
\lve B(\bv_1,\bv_1)-B(\bv_2,\bv_2)\rve \le c_1 \biggl(&\lve
\nabla(\bv_1-\bv_2)\rve \lve
\nabla\bv_1\rve^{1-a}\lve
\Delta\bv_1\rve^a\\ &+\lve \nabla
(\bv_1-\bv_2)\lve^{1-a}\lve \Delta (\bv_1-\bv_2)\lve^{a}\lve
\nabla\bv_2\lve\biggr)
\end{split}
\end{equation}
and
\begin{equation}\label{local-Lip-F-2}
\begin{split}
\lve M(\bd_1)-M(\bd_2)\rve \le c_2\biggl(&\lve \d_1-\d_2\rve_2
\lve \d_1\rve_2^{1-a} \lve \d_1\rve_3^a\\ &+ \lve
\d_1-\d_2\rve^{1-a}_2 \lve \d_1-\d_2\rve_3^a \lve \d_2\rve_2
\biggr).
\end{split}
\end{equation}
Note that we used the shorthand notation $M(\d):=M(\d,\d)$ and $B(\bv):=B(\bv,\bv)$.
 \end{lem}
 \begin{proof}
Set $(\w, \db)=(\bv_1-\bv_2, \d_1-\d_2)$. We start with the proof
of the estimate \eqref{local-Lip-F-1}. Notice that the
left-hand-side of \eqref{local-Lip-F-1} is equal to
\begin{equation*}
\lve B(\w,\bv_1)+B(\bv_2, \w)\rve.
\end{equation*}
Now we estimate the last identity as follows
\begin{equation*}
\lve B(\w,\bv_1)+B(\bv_2, \w)\rve\le C  \lve\w\rve_{\el^4}\lve
\nabla \bv_1\rve_{\el^4} + \lve \bv_2\rve_{\el^4}\lve \nabla \w
\rve_{\el^4},
\end{equation*}
from which along with \eqref{GAG-l4} and the embedding
\eqref{SOB-EM} we easily derive the estimate
\eqref{local-Lip-F-1}.

Next we show that \eqref{local-Lip-F-2} holds. From elementary
calculi we infer the existence of a constant $C>0$ such that
\begin{equation*}
\lve M(\mathbf{f}, \mathbf{g})\rve \le C \lve D^2 \mathbf{f}\rve
\lve \nabla \mathbf{g}\rve_{\el^\infty}+\lve \nabla \mathbf{f}
\rve_{\el^4}\lve D^2 \mathbf{g}\rve_{\el^4}.
\end{equation*}
Owing to the embedding \eqref{SOB-EM} it is not difficult to check
that
\begin{equation*}
\lve M(\mathbf{f}, \mathbf{g})\rve \le C \lve
\mathbf{f}\rve_2\biggl(\lve \nabla \mathbf{g}\rve_{\el^\infty}+
\lve D^2 \mathbf{g}\rve_{\el^4}\biggr).
\end{equation*}
Owing to  \eqref{GAG-l4},  \eqref{GAG-LInf-2} and the
embedding \eqref{SOB-EM} we obtain that
\begin{equation}\label{EST-MF}
\lve M(\mathbf{f}, \mathbf{g})\rve \le C \lve \mathbf{f}\rve_2
\lve \mathbf{g} \rve^{1-a}_2 \lve \mathbf{g} \rve^a_3, \quad
a=\frac d 4.
\end{equation}
Note that
\begin{equation*}
 M(\d_1)-M(\d_2)=M(\d_1-\d_2,\d_1)+M(\d_2,\d_1-\d_2).
\end{equation*}
From this last identity and  \eqref{EST-MF} we easily deduce
the inequality \eqref{local-Lip-F-2}. This ends the proof of Lemma
\ref{Local-LIP-Lem}.
 \end{proof}
 \begin{lem}\label{Local-LIP-Lem-2}
There exist a constant $c_3>0$ such that for any $(\bv_i,
\bd_i)\in  E$, i$=1,2$ we have, with $a=\frac d4$,
\begin{equation}\label{local-Lip-F-3}
\begin{split}
\lve \tilde{B}(\bv_1,\bd_1)-\tilde{B}(\bv_2,\bd_2)\rve_1 \le c_3
\biggl(&\lve \nabla(\bv_1-\bv_2)\rve \lve
\bd_1\rve_2^{1-a}\lve \d_1\rve_3^a\\
&+\lve(\bd_1-\bd_2)\lve^{1-a}_ 2 \lve(\bd_1-\bd_2)\lve^{a}_3\lve
\nabla\bv_2\lve\biggr)
\end{split}
\end{equation}
 \end{lem}
 \begin{proof}
As in the proof of Lemma \ref{Local-LIP-Lem} we set $(\w,
\db)=(\bv_1-\bv_2, \d_1-\d_2)$ and notice that the left hand side
of  \eqref{local-Lip-F-3} is equal to
\begin{equation*}
\tilde{B}(\w,\d_1)+\tilde{B}(\d_2, \w):=J_1+J_2.
\end{equation*}
Now we want to estimate $\lve J_i\rve_1=\sqrt{\lve J_i\rve^2 +\lve
\nabla J_i \rve^2}$ for $i=1,2.$ Since estimating $\lve J_i\rve$
is easy we will just focus on the term $\lve \nabla J_i \rve$.
There exists a constant $C>0$ such that
\begin{align*}
\lve \nabla J_1\rve\le & C \biggl(\lve \nabla \w \nabla
\d_1\rve+\lve\Delta \d_1 \w\rve\biggr),\\
\le & C \biggl(\lve \nabla \w\rve \lve \nabla
\d_1\rve_{\el^\infty}+\lve \w\rve_{\el^4} \lve \Delta
\d_1\rve_{\el^4}\biggr).
\end{align*}
Invoking  \eqref{GAG-l4},  \eqref{GAG-LInf} and the
embedding \eqref{SOB-EM} we infer that with $a=\frac d4$,
\begin{equation*}
\lve \nabla J_1\rve\le C \lve \nabla
\w\rve\biggl(\lve \nabla \d_1\rve^{1-a}_1 \lve \Delta \d_1\rve^a_1 +
\lve \Delta \d_1\rve^{1-a} \lve \d_1\rve_{3}^a \biggr).
\end{equation*}
This last inequality implies that there exits $\tilde{c}>0$ such
that with $a=\frac d4$,
\begin{equation*}
\lve \nabla J_1\rve\le \tilde{c} \lve \nabla
\w\rve \lve \d_1\rve^{1-a}_2 \lve \d_1\rve^a_3.
\end{equation*}
Using similar argument we can also prove that (again with $a=\frac d4$)
\begin{equation*}
\lve \nabla J_2\rve\le \tilde{c} \lve \nabla
\bv_2 \rve \lve \d_1-\d_2\rve^{1-a}_2 \lve \d_1-\d_2\rve^a_3.
\end{equation*}
The inequality \eqref{local-Lip-F-3} easily follows from these
last two estimates.
 \end{proof}

 \begin{lem}\label{Local-LIP-Lem-3}
If Assumption \ref{eqn-f} is verified, then  there exists $c_4>0$
 such that for any $\bd_i\in \bx_{\frac 12}\cap \bx_1$ with $i=1,2$.
 \begin{equation}\label{local-Lip-F-4}
\begin{split}
\lve f(\bd_1)-f(\bd_2)\rve_1 \le c_4\Big[1+\lve \bd_1\rve_2^{2N}+\lve \bd_2\rve_2^{2N}\Big]\lve \bd_1-\bd_2\rve_2.
\end{split}
\end{equation}
 \end{lem}
 \begin{proof}
 Since $\h^2 $ is and algebra we have $\lvert \bd_1 \rvert^{2k} \bd_2$ for any $\bd_1, \bd_2\in \h^2$ and $k\ge 0$ and
 there exists $C>0$ such that
 \begin{equation}\label{POL-EST-0}
  \lve\lvert \bd_1 \rvert^{2k} \bd_2\rve_2\le C  \lve \bd_1 \rve_2^{2k} \lve \bd_2\rve_2.
 \end{equation}
By Young's inequality, there exists $C>0$ such that
\begin{equation}\label{POL-EST}
\lve \bd_1 \rve_2^{2k} \lve \bd_2 \rve_2\le C \lve \bd_2\rve_2(1+\lve \bd_1 \lve_2^{2N} ) ,\text{
 for $k=0,\ldots, N$.}
\end{equation}
 Thus it is sufficient to prove the lemma for the leading term $a_N \lvert \bd \rvert^{2N}\bd$.\\
 We have $$\lvert \bd_1\rvert^{2N}\bd_1-\lvert \bd_2\rvert^{2N}\bd_2= \vert \bd_1\vert^{2N}(\bd_1-\bd_2)+
 \bd_2(\vert\bd_1\vert-\vert\bd_2\vert)(\sum_{k=0}^{2N-1}\vert \bd_1\vert^{2N-1-k}\vert \bd_2\rvert^{k} ),$$
 from which along with \eqref{POL-EST-0} and \eqref{POL-EST} we easily infer that the lemma is true from the leading term  $a_N \lvert \bd \rvert^{2N}\bd$.
 \end{proof}
For two Banach spaces $(B_i, \lve \cdot \rve_{B_i})$ with $i=1, 2$
we endow the product space $B_1\times B_2$ with the norm $$
|(b_1,b_2)|=\sqrt{\lve b_1\rve_{B_1}^2 +\lve b_2\rve_{{B}_2}^2}.$$
\begin{prop}\label{SLC-ST}
If Assumption \ref{eqn-f} is satisfied, then there exists a
constant $C_0>0$ such that for any $\y_i=(\bv_i,\d_i)$, \,\,
$i=1,2$, with  $\alpha=\frac d4$,  we have
\begin{equation*}
\begin{split}
\lve \f(\y_1) -\f(\y_2)\rve_H \le C_0 \Vert \y_1-\y_2\Vert_V^{{1-\alpha}}
 \Big[{\Vert \y_1-\y_2\Vert^{\alpha}_V}  \Vert \y_1\Vert^{1-\alpha}_V \Vert \y_1\Vert_E^\alpha + \Vert
\y_1-\y_2\Vert_E^\alpha  \Vert
\y_2\Vert_V\Big]\\
+c_4\lve \y_1-\y_2\rve_V \Big[1+\lve \y_1\rve_V^{2N}+\lve \y_2\rve_V^{2N}\Big].
\end{split}
\end{equation*}
\end{prop}
\begin{proof}
 The proposition is a consequence of Lemma \ref{Local-LIP-Lem}, Lemma
 \ref{Local-LIP-Lem-2} and Lemma \ref{Local-LIP-Lem-3}. Its proof
 is easy and we omit it.
 \end{proof}

For
any integer $k>1$ let
\begin{equation}\label{STOP}
\tau_k=\inf\{t\ge 0: \lve \nabla \v(t)\rve^2 +\lve \Delta \d(t)\rve^2 >
k^2\}.
\end{equation}
 Hereafter, for a
vector-valued function $\bu:[0,t\wedge \tau_k]\rightarrow \mathbf{B}$ we will
write  $\int_0^{t\wedge \tau_k} \bu\,\, ds :=\int_0^{t\wedge \tau_k} \bu(s) ds$ for any
$t>0.$

 Our first main result is contained in the following theorem. It
 is basically a corollary of two general theorems (see Theorem
 \ref{thm_local} and Theorem \ref{thm_maximal-abstract}) that we will state
 and prove in Section \ref{ABST-STRONG}.
\begin{thm}\label{LC-Local-Sol}
If Assumptions \ref{HYPO-ST} and \ref{eqn-f} are satisfied and $(\v_0,\d_0)\in \el^2(\Omega;D(\rA^\frac12)\times \bx_{\frac 12})$ is $\mathcal{F}_0$-measurable, $h\in \h^2$ and $d=2,3$,   then the
stochastic equation \eqref{ABSTRACT-LC} for the liquid crystals has a local-maximal strong solution $((\bv;\bd), \tilde{\tau}_\infty)$ satisfying the following properties:
\begin{enumerate}
\item\label{item-1} given $R>0$ and  $\varepsilon >0$ there exists
$\tau(\varepsilon,R)>0$,  such that for every
$\mathcal{F}_0$-measurable $D(\rA^\frac12)\times \bx_{\frac12}$-valued random variable $(\v_0,\d_0)$
satisfying  $\mathbb{E}\Vert (\v_0,\d_0) \Vert^{2}_{D(\rA^\frac12)\times \bx_{\frac12}} \leq R^{2}$, one has
\[{\mathbb P}\big(\tilde{\tau}_\infty \geq \tau(\varepsilon,R)\big) \geq
1-\varepsilon.\]
\item \label{THM-ii} We also have
\begin{align}
\mathbb{P}\left(\{ \tilde{\tau}_\infty <\infty \}\cap \{ \lve \nabla \bv(t) \rve + \lve \Delta \bd(t) \rve<\infty  \} \right)=0,\label{MAX-P1}\\
\limsup_{t\toup\tilde{\tau}_\infty } \lve \nabla \bv(t) \rve + \lve \Delta \bd(t) \rve=\infty \text{ a.s. on } \{\tilde{\tau}_\infty<\infty \}\label{MAX-P2}.
\end{align}
\end{enumerate}

\end{thm}
\begin{proof}
Lemma \ref{SEM-1}-\ref{SEM-3} show that  $\{\mathbb{S}(t)\}_{t\geq 0}$ on
$H=\h\times \bx_{0}$ satisfies Assumption \ref{assum-01}.
Thanks to Proposition \ref{SLC-ST} we can infer from Theorem
\ref{thm_local} and Theorem \ref{thm_maximal-abstract} that problem
\eqref{ABSTRACT-LC} has a local and maximal strong solution satisfying items \eqref{item-1} and \eqref{THM-ii}. This concludes the proof of the theorem.
\end{proof}
The second result is about global resolvability of the stochastic
equation for two dimensional nematic liquid crystal. This result relies on the Khashminskii test for non-explosions (see \cite[Theorem
III.4.1]{Kh_1980} for the finite-dimensional case) and some arguments from \cite{ZB et al 2005}.
\begin{thm}\label{GLOBAL-ST}
If in addition to Assumptions \ref{HYPO-ST} and  \ref{HYPO-ST-weak}, we also suppose that $d=2$, $h \in \h^2$ and $(\v_0, \d_0) \in \el^2(\Omega; D(\rA^\frac12)\times  \bx_{\frac 12})$ such that
\begin{equation}\label{E_0}
\mathbb{E}(\mathscr{E}_0^{4N+2}):=\mathbb{E} \biggl(\Vert \bv_0\Vert^2 + \Vert \bd_0\Vert^2+\Vert \nabla \bd_0\Vert^2 + \int_{\mo}F(\bd_0(x)) dx \biggr)^{4N+2}<\infty,
\end{equation}
then the
stochastic equation \eqref{ABSTRACT-LC} for nematic liquid
crystals has a global strong solution.
\end{thm}
To prove this theorem we need several auxiliary results some of whose proofs will be postponed till the appendix.
Let us start with the following 3 lemmata.
\begin{lem}\label{LEMMA_1}
	For any $h\in \mathbb{H}^2$, there exists two constants $\kappa_1>0$ and $\kappa_2>0$, which depends only on the $\mathbb{H}^2$-norm of $h$, such that for any $\bd\in \h^1 $ satisfying $ \Delta \bd- f(\bd) \in \el^2$we have
	\begin{equation}\label{ST10-b}
	\lvert\langle\Delta \bd -f(\bd), \Delta \bd \times h \rangle + \langle\Delta \bd -f(\bd), \Delta \bd \times \Delta h +(\bd \times h)\times \Delta h \rangle \rvert \le \kappa_1 (\Vert \Delta \bd -f(\d) \Vert^2 + \Vert \bd \Vert^{4N+2}_{\h^1} +1).
	\end{equation}
	\begin{equation}\label{ST12}
	\Vert \Delta (\bd \times h)-f^\prime(\bd)(\bd \times h)\Vert^2 + \vert \langle \Delta \bd -f(\bd), f^{\prime \prime}(\bd)[G(\bd), G(\bd)]\rangle \vert \le \kappa_2  (\Vert \Delta \bd -f(\d) \Vert^2 + \Vert \bd \Vert^{4N+2}_{\h^1} +1).
	\end{equation}
\end{lem}
\begin{proof}
	The proof of the lemma will be given in Appendix \ref{AppB}.
\end{proof}
\begin{lem}\label{LEMMA_3}
	For  $N\in I_2$, there exists a constant $\kappa_3>0$ such that for any $\bn\in \h^1 $ satisfying $ \Delta \bd- f(\bd) \in \el^2$ we have
	\begin{equation} \label{ST6-B}
	\vert \langle f^\prime(\bd) (\Delta \bd -f(\bd)), \Delta \bd -f(\bd)\rangle \vert \le \frac14 + \kappa_3 (\ell_1 + \ell_2 \Vert \bd \Vert^{2N}_{\h^1} )^2 \Vert \Delta \bd -f(\bd) \Vert^2,
	\end{equation}
	where the constants $\ell_1$ and $\ell_2$ are defined in \eqref{ST6-B-0} and \eqref{ST6-B-1}, respectively.
\end{lem}
\begin{proof}
	We will carry out the proof of this lemma in Appendix \ref{AppB}.
\end{proof}
\begin{lem}\label{LEMMA_4}
	There exists a constant $\kappa_4$ such that
	\begin{equation}\label{ST120}
	\begin{split}
	2 \lvert\langle \nabla \bd (\Delta \bd -f(\bd)), \rA \bv \rangle -\langle \nabla \bd (\Delta \bv ), \Delta \bd -f(\bd) \rangle \rvert\le
	\frac14 \Vert \rA \bv \Vert^2 + \frac14 \Vert \nabla(\Delta \bd -f(\bd)) \Vert^2 \\ + \kappa_4 \Vert \nabla \bd \Vert^2 \Vert \Delta \bd \Vert^2 \Vert \Delta \bd -f(\bd) \Vert^2,
	\end{split}
	\end{equation}
	for any $\bv \in D(\rA)$, $\bd \in \h^2$.
\end{lem}
\begin{proof}
Let us call $LHS$ the right-hand side of the inequality in the statement of the lemma.	From the Cauchy-Schwarz inequality we infer that there exists a constant $C>0$ such that
	\begin{equation*}
	LHS\le 4 C \Vert \rA \bv \Vert \Vert \nabla \bd \Vert_{\el^4} \Vert \Delta \bd -f(\bd) \Vert_{\el^4},
	\end{equation*}
	which, along with the Gagliardo-Nirenberg and the Cauchy inequalities, implies
	\begin{equation*}
	LHS\le \frac14 \Vert \rA \bv \Vert^2 + 16 C^2 \Vert \nabla \bd \Vert \Vert \Delta \bd \Vert \Vert \Delta \bd -f(\bd)\Vert \Vert \nabla (\Delta \bd -f(\bd))\Vert.
	\end{equation*}
	Invoking again the  Cauchy  inequality yields
		\begin{equation*}
		LHS\le \frac14 \Vert \rA \bv \Vert^2 + \frac14 \Vert \nabla (\Delta \bd -f(\bd))\Vert^2 +16^2 C^4 \Vert \nabla \bd \Vert^2 \Vert \Delta \bd \Vert^2 \Vert \Delta \bd -f(\bd)\Vert^2,
		\end{equation*}
		which completes the proof of the lemma.
\end{proof}
We also need the following lemma.
\begin{lem}\label{LEMMA_5}
	There exists a constant $\kappa_5$ such that for any $\bv \in D(\rA)$ we have
	\begin{equation}\label{MOST-DIFF}
	\vert (B(\bv,\bv), \rA \bv)\vert \le \frac14 \Vert \rA \bv \Vert^2 + \kappa_5 \Vert \bv \Vert^2 \Vert \nabla \bv \Vert^4.
	\end{equation}
\end{lem}
\begin{proof}[Proof of Lemma \ref{LEMMA_5}]
	 First, by Cauchy-Schwarz inequality we obtain
	 \begin{equation}\label{Est-max}
	 \lvert \langle B(\v,\v), \rA\v\rangle\rvert\le \lve
	 B(\v,\v)\rve \lve \rA\v \rve.
	 \end{equation}
	From  the H\"older and the Gagliardo-Nirenberg inequalities, we derive  that there exists a constant $C_1>0$
	 such that
	 \begin{equation*}
	 \lve B(\bv, \bv)\lve \le C_1  \lve \bv\rve^{1-\frac d4}\rve
	 \nabla \bv\rve^\frac d4 \lve \nabla\bv\rve^{1-\frac d4} \lve\Delta
	 \bv \rve^{\frac d4}, \text{ for any } \bv \in \ve, \bv \in D(\rA),
	 \end{equation*}
	 from which along with  \eqref{Est-max} we derive that
	 \begin{align}
	 \lvert \langle B(\v,\v), \rA\v\rangle\rvert\le& c \lve
	 \v\rve^{1-\frac d4}\lve \nabla \v\rve^\frac d4 \lve
	 \nabla \v
	 \rve^{1-\frac d4}\lve \Delta \v\rve^{\frac d4}\lve \Delta \v\rve,\nonumber\\
	 \le&c \lve \v\rve^{1-\frac d4}\lve \nabla \v\rve \lve \a
	 \v\rve^{1+\frac d4}.\label{ST6}
	 \end{align}
	 Let us recall  Young's inequality  $ab\le C(\alpha,p,q)a^p+\alpha
	 b^q$ for $p^{-1}+q^{-1}=1$ and arbitrary  $\alpha>0$. Let us choose $p=\frac{8}{d+4}$ and
	 $q=\frac{8}{4-d}$.
	 Applying Young's inequality with the above $p$ and $q$ in
	 \eqref{ST6} we obtain
	 \begin{equation*}
	 \lvert \langle B(\v,\v), \rA\v\rangle\rvert\le \alpha \lve \rA \v
	 \rve^2 + C(\alpha,p,q)\lve \v\rve^{2} \lve \nabla \v
	 \rve^{2+\frac{2d}{4-d}}.
	 \end{equation*}
	 We easily conclude the proof of the lemma by setting $d=2$ and $\alpha=\frac14$ in the last estimate.
\end{proof}
Now, we proceed to the statement of  proposition which is very crucial for the proof of Theorem \ref{GLOBAL-ST}. But before continuing further, let us fix some notations.
Let  $\ell_1$ and $\ell_2$ be the positive constants defined in \eqref{ST6-B-0} and \eqref{ST6-B-1}, respectively.
Let $\left((\bv;\bd),\tilde{\tau}_\infty\right)$ be the maximal solution of the problem \eqref{ABSTRACT-LC} obtained in Theorem
\ref{LC-Local-Sol}, and let us set
\begin{align*}
\phi(s)=&\kappa_3 (\ell_1 + \ell_2\sup_{r\in [0,s]}\lve \d(r) \rve^{2N}_{\h^1})^2+\kappa_4 \Vert \nabla\bd (s) \Vert^2 \Vert \Delta \bd(s) \Vert^2   + \kappa_5 \lve \v(s)\rve^{2} \lve \nabla \v(s)
\rve^{2},\; s\ge 0,\\
\Phi(s)=&e^{-\int_0^{s} \phi(r) dr}, \; s\ge 0,
\end{align*}
where the constants $\kappa_i$, $i=3,\ldots,5$, are given in Lemma \ref{LEMMA_3}-\ref{LEMMA_5}. Now, we are ready to state the aforementioned proposition.
\begin{prop}\label{STRONGER-NORM}
Let all the assumptions of Theorem \ref{GLOBAL-ST} be satisfied and	 $(\v,\d)$ be the local strong solution of \eqref{ABSTRACT-LC} obtained in Theorem
	\ref{LC-Local-Sol}. Let $$\Psi(\bd)=\lve \Delta \d-
	f(\d)\rve^2.$$
Then, there exist 2 increasing functions $\kappa_0,\kappa_1:[0,T] \to (0,\infty)$ such that
	for any $t>0$
	\begin{equation}\label{ST15}
	\begin{split}
	\E \biggl[\Phi(t\wedge \tau_k)\left(\lve \nabla
	\v(t\wedge \tau_k)\rve^2+\Psi(\d(t\wedge \tau_k))\right)\biggr] + \E \int_0^{t\wedge \tau_k}
	\Phi(s)\Big(2\lve \Delta\v(s)
	\rve^2+ \lve \nabla (\Delta \d(s) -f(\d(s)) )\rve^2 \Big)ds\\
	\le \kappa_0(t)\mathbb{E}(\rve \nabla \v_0\rve^2+\Psi(\d_0))+ \kappa_1(t) \mathbb{E}(\mathscr{E}_0^{2N+1}),
	\end{split}
	\end{equation}
	where $\Phi(\cdot)$ is defined as above and $\mathscr{E}_0$ is as in \eqref{E_0}.
\end{prop}
\begin{proof}
	The proof will be given in Appendix \ref{AppB}.
\end{proof}
After all these preparations we proceed to the promised proof of Theorem \ref{GLOBAL-ST}.
\begin{proof}[Proof of Theorem \ref{GLOBAL-ST}]
Let $\{ \tau_k; k \in \mathbb{N}\}$ be the sequence of stopping times defined in \eqref{STOP}. By the item \eqref{THM-ii} of Theorem \ref{LC-Local-Sol} the solution is global if we are able to show that   $$ \mathbb{P}\Big(\tilde{\tau}_\infty<\infty\Big)=0 .$$
For this aim, we first derive the following chain of inequalities
\begin{equation*}
\begin{split}
\mathbb{P}\left(\tau_k<t\right)&\le  \mathbb{E}\left(1_{\{\tau_k<t\}} 1_{\{e^{\int_0^{t\wedge \tau_k}\phi(r)dr}\le k\} } \right)+  \mathbb{E}\left(1_{\{\tau_k<t\}} 1_{\{e^{\int_0^{t\wedge \tau_k}\phi(r) dr}> k\} } \right),\\
&\le \frac1k \mathbb{E}\biggl(1_{\{\tau_k<t\}} k^2 e^{-\frac 12
\int_0^{t\wedge \tau_k}\phi(s)ds }\biggr)  +\mathbb{E}\biggl(1_{\{\tau_k<t\}} 1_{\{\int_0^{t\wedge \tau_k} \phi(s) ds
> \log{k}\} } \biggr),\\
\le & \mathrm{I}+\mathrm{II}.
\end{split}
\end{equation*}
Now, we will estimate $\mathrm{I}$ and $\mathrm{II}$ separately.
First, from the definition of $\tau_k$ and $\Phi$ we have
\begin{align*}
\mathrm{I}\le \frac1k \mathbb{E} \left(1_{\{\tau_k < t\}} \Phi(t\wedge \tau_k)\left( \lve \rA \bv(t\wedge \tau_k) \rve^2 + \lve \Delta \bd(t\wedge \tau_k) \rve^2\right) \right)\\
\le  \frac1k \mathbb{E} \left(\Phi(t\wedge \tau_k)\left( \lve \rA \bv(t\wedge \tau_k) \rve^2 + \lve \Delta \bd(t\wedge \tau_k) \rve^2\right) \right)
\end{align*}
From Proposition \ref{EST1}, Remark \ref{REM-H2},
\eqref{ST15} and  \eqref{bigdanh2}, we infer that
\begin{equation*}
\mathrm{I}\le \frac1k\left(\kappa_0(t)\mathbb{E}(\rve \nabla \v_0\rve^2+\Psi(\d_0))+ \kappa_1(t) \mathbb{E}(\mathscr{E}_0^{2N+1})\right),
\end{equation*}
for any $k\in \mathbb{N}$.
Second, we estimate $\mathrm{II}$ as follows
\begin{align*}
\mathrm{II}=& \mathbb{E}\left(1_{\{\tau_k < t\}} 1_{\{\int_0^{t\wedge \tau_k} \phi(r) dr > \log{k} \} }\right)\\
\le & \int_{\{\tau_k<t\} \cap \{\int_0^{t\wedge \tau_k} \phi(r) dr>\log{k} \}}\frac{ \int_0^{t\wedge \tau_k} \phi(r) dr}{\log{k}} d\mathbb{P}\\
\le & \frac1{\log{k}} \int_{\{\tau_k < t\}} \int_0^{t\wedge \tau_k} \phi(r) dr\;\;\; d\mathbb{P} \\
\le & \frac{1}{\log{k}} \mathbb{E} \int_0^{t\wedge \tau_k} \phi(r) dr.
\end{align*}
Now, from Proposition \ref{EST1} we infer that there exists an increasing function $\tilde{\kappa}:[0,\infty) \to (0,\infty)$ such that
\begin{align*}
&\E\int_0^{t\wedge \tau_k} \lve \v(s)\rve^2\lve \nabla \v(s)\rve^2 ds
\le \E\sup_{0\le s\le t\wedge \tau_k}\lve \v(s)\rve^4 + \E \biggl[\int_0^{t\wedge \tau_k}
\lve \nabla \v(s)\rve^2 ds
\biggr]^2\le \tilde{\kappa}(t)(1+\mathfrak{G}_1(t,4N+2)),\\
& \E\int_0^{t\wedge \tau_k}(\ell_1+\ell_2\lve \d(s)\rve^{2N}_{\h^1})^2 ds\le\tilde{\kappa}(t)(1+\mathfrak{G}_1(t,4N+2)),\\
& \E \int_0^{t\wedge \tau_k} \Vert \nabla \bd(s) \Vert^2 \Vert \Delta \bd (s) \Vert^2 ds \le \E \sup_{0\le s\le t\wedge \tau_k} \Vert \nabla \bd(s)\Vert^4 + \E \biggl[\int_0^{t\wedge \tau_k} \Vert \Delta \bd(s) \Vert^2 ds \biggr]^2\le \tilde{\kappa}(t)(1+\mathfrak{G}_1(t,4N+2)),
\end{align*}
where $\mathfrak{G}_1(t,4N+2)$ is the constant in \eqref{G_1} with $T=t$ and $p=4N+2$.
Thus, there exists a constant $C>0$ such that for any $k\in \mathbb{N}$
\begin{equation*}
\E \int_0^{t\wedge \tau_k} \phi(s)ds\le C \tilde{\kappa}(t)(1+\mathfrak{G}_1(t,2))
\end{equation*}
and
\begin{equation*}
\mathrm{II} \le \frac1{\log{k}} C \tilde{\kappa}(t)(1+\mathfrak{G}_1(t,2)).
\end{equation*}
Collecting the information about $\mathrm{I}$ and $\mathrm{II}$ together, we infer that there exists an increasing function $C: [0,\infty) \to (0,\infty)$ such that for any $k \in \mathbb{N}$
\begin{equation*} 
 \begin{split}
  \mathbb{P}\left(\tau_k<t\right)&\le  C(t) [\frac 1 k  +\frac{1}{\log{k}}],
 \end{split}
\end{equation*}
which implies that

$$\lim_{k\rightarrow \infty}\mathbb{P}\left(\tau_k<t\right)=0.$$
 Now, we easily deduce from the last identity and Theorem \ref{LC-Local-Sol}\eqref{THM-ii} that $\mathbb{P}\Big(\tilde{\tau}_\infty<\infty\Big)=0 $ which implies the global resolvability of our system for the case $d=2$. 
\end{proof}
\section{Strong solution for an abstract stochastic equation}\label{ABST-STRONG}
The goal of this section is to prove a general result about the existence of local and maximal solution to an
abstract stochastic partial differential equations with locally
Lipschitz continuous coefficients. This is achieved by using some
truncation and fixed point methods.
\subsection{Notations and Preliminary}\label{abstract-framework}
Let $V$,  $E$ and $H$ be separable Banach spaces such that
$E\subset V$ continuously. We denote the norm in $V$ by $\Vert
\cdot \Vert$ and we put
\begin{equation}\label{eqn-X_T}
 X_T:= C([0,T];V) \cap L^2(0,T;E)
 \end{equation}
with the norm
\begin{equation}\label{eqn-X_T-norm}
 \vert u\vert_{X_T}^2= \sup_{s \in [0,T]} \Vert u(s)\Vert^2+\int_0^T \vert u(s) \vert_E^2\, ds.
 \end{equation}
Let $F$ and $G$ be two nonlinear mappings satisfying the following
sets of conditions.
\begin{assum}\label{assum-F}
 Suppose that $F: E \to H$ is such that
$F(0)=0$ and there exist $p ,q \geq 1$, $\alpha, \gamma \in [0,1)$ and $C>0$
such that
\begin{equation}\label{eqn-local Lipschitz-F}
\begin{split}
\vert F(y)-F(x) \vert_H \leq C \Big[ \Vert y-x\Vert \Vert
y\Vert^{p-\alpha} \vert y\vert_E^\alpha + \vert y-x\vert_E^\alpha
\Vert y-x\Vert^{1-\alpha} \Vert x\Vert^p\Big]\\
+ C \Big[ \Vert y-x\Vert \Vert
y\Vert^{q-\gamma} \vert y\vert_E^\gamma + \vert y-x\vert_E^\gamma
\Vert y-x\Vert^{1-\gamma} \Vert x\Vert^q\Big],
\end{split}
\end{equation}
for any $x, y\in E$.
\end{assum}
\begin{assum}\label{assum-G}
Assume that $G: E \to V$  such that $G(0)=0$ and there exists $k
\geq 1$, $\beta \in [0,1)$ and $C_G>0$ such that
\begin{equation}\label{eqn-local Lipschitz-G}
\Vert G(y)-G(x) \Vert \leq C_G \Big[ \Vert y-x\Vert \Vert
y\Vert^{k-\beta} \vert y\vert_E^\beta + \vert y-x\vert_E^\beta
\Vert y-x\Vert^{1-\beta} \Vert x\Vert^k\Big],
\end{equation}
for any $x, y\in E$.
\end{assum}
 Let $(\Omega,
\mathcal{F}, \mathbb{P})$ be a complete probability space equipped
with a filtration $\mathbb{F}=\{\mathcal{F}_t: t\geq 0\}$
satisfying the usual condition.   By $M^2(X_T)$ we denote
the space of all progressively measurable $E$-values processes
whose trajectories belong to $X_T$ almost surely endowed with a
norm
\begin{equation}\label{eqn-M^2X_T}
\vert u \vert_{M^2(X_T)}^2 = \mathbb{E}\Big[ \sup_{s \in [0,T]}
\Vert u(s)\Vert^2+\int_0^T \vert u(s) \vert_E^2\, ds\Big].
\end{equation}
Let us also formulate the following assumptions.
\begin{assum}\label{assum-01}
Suppose that the embeddings $E\subset V \subset H$ are continuous. Consider (for
simplicity) a one-dimensional Wiener process $W(t)$.\\
Assume that $S(t)$, $t\in [0,\infty)$, is a family of bounded linear operators on the space $H$ such that there exists two positive constants $C_1$ and $C_2$ with the following properties . \\
(i) For every $T>0$ and every  $f\in L^2(0,T;H)$ a function
$u=S\ast f$ defined by
\[ u(t)=\int_0^T S(t-r) f(r)\, dt, \;\; t \in [0,T]\]
belongs to $X_T$ and
\begin{equation}\label{ineq-dc}
\vert u \vert _{X_T}\leq C_1 \vert f \vert _{L^2(0,T;H)}.
\end{equation}
(ii) For every $T>0$ and every process  $\xi\in M^2(0,T;V)$ a
process $u=S \diamond \xi$ defined by
\[ u(t)=\int_0^T S(t-r) \xi(r)\, dW(r), \;\; t \in [0,T]\]
belongs to $M^2(X_T)$ and
\begin{equation*}
\vert u \vert _{M^2(X_T)}\leq C_2 \vert \xi \vert _{M^2(0,T;V)}.
\end{equation*}
(iii) For every $T>0$ and every $u_0\in V$, a function $u=Su_0$
defined by
\[ u(t)= S(t)u_0,  \;\; t \in [0,T]\]
belongs to $X_T$. Moreover, for every $T_0>0$ there exist $C_0>0$
such that for all $T\in (0,T_0]$,
\begin{equation}\label{ineq-dc-2}
\vert u \vert _{X_T}\leq C_0 \Vert u_0 \Vert.
\end{equation}
\end{assum}
Now let us consider a semigroup $S(t)$, $t\in [0,\infty)$ as above
and the abstract SPDEs
\begin{equation}\label{ABS-SPDE-1}
u(t)=S(t)u_0+\int_0^t S(t-s) F(u(s)) ds+\int_0^t S(t-s) G(u(s))
dW(s),\;\; \mbox{ for any }t>0
\end{equation}

which is a mild version of the problem

\begin{equation}\label{ABS-SPDE-strong}
\left\{\begin{array}{rl} du(t)&= Au(t)\,dt+  F\big(u(t)\big)\, dt+
G\big(u(t)\big)
dW(t),\;\;t>0,\\
u(0)&=u_0.
\end{array}
\right.
\end{equation}
\begin{Def}\label{def-local solution-2}
Assume that a $V$-valued  $\mathcal{F}_0$ measurable random variable  $u_0$ is given. A local mild
solution to problem \eqref{ABS-SPDE-strong} (with the initial time
$0$) is a pair $(u,\tau)$ such that
\begin{enumerate}
\item $\tau$ is an accessible $\mathbb{F}$-stopping time, \item
$u: [0,\tau)\times \Omega \to V$ is an admissible\footnote{This
also follows from  condition (3) below.} process, \item there
exists an  approximating sequence  $(\tau_m)_{m\in \mathbb{N}}$ of
$\mathbb{F}$ finite stopping times  such that $\tau_m \toup \tau$
a.s. and, for every $m\in \mathbb{N}$ and $t\ge 0$,
 we have
\begin{eqnarray}\label{eq-locsol_00}
&&\mathbb{E}\Big(  \sup_{s\in [0,t\wedge \tau_m]} \Vert
u(s)\Vert^2 +\int_0^{t\wedge \tau_m} \vert u(s)\vert_E^2 \,
ds\Big)<\infty,
\\
\label{eq-locsol_00-b} u(t\wedge \tau_m)&=&S(t\wedge
\tau_m)u_0+\int_0^{t\wedge \tau_m}
S(t\wedge\tau_m-s) F(u(s)) ds\\
\nonumber &+&\int_0^\infty
\mathds{1}_{[0,t\wedge\tau_m)}S(t\wedge\tau_m-s)G(u(s))\,  dW(s).
\end{eqnarray}
\end{enumerate}
Along the lines of the paper \cite{Brz+Elw_2000}, we said that a
local solution $u(t)$, $t < \tau$ is called global iff
$\tau=\infty$ a.s.
\end{Def}
\begin{Rem}\label{Rem-independence}
{\rm  The  Definition \ref{def-local solution-2} of  a local
solution  is independent  of the choice of the  sequence
$\big(\tau_n\big)$.  A proof  of this fact follows from the
continuity of trajectories of the process $u$ (what is a consequence
of admissibility of $u$) and is based on the following three
principles.
\begin{trivlist}
\item[{ \bf  (i)}] If $\tau$ is  an accessible stopping time, then
there exist an increasing sequence $\tau_n$ of {\em discrete}
stopping times such that a.s. $\tau_n < \tau$ and $\tau_n \toup
\tau$; \item[{ \bf  (ii)}] if $\tau$  is an accessible  stopping
time and $\sigma \le \tau$ is a stopping time, then $\sigma$ is
also accessible. \item[{  \bf (ii)}]  if a pair $(u,\tau)$  is  a
local solution to \eqref{ABS-SPDE-1}, then \eqref{eq-locsol_00-b} holds with $t$ being any
discrete stopping time.
\end{trivlist}
It follows that the following is an equivalent definition of
a local solution. \\} A pair  $(u,\tau)$,  where $\tau$ be an
accessible stopping  time and $u:[0,\tau)\times \Omega\to V$ is an
admissible   process, is a local  mild solution to  equation
\eqref{ABS-SPDE-strong}  iff  for every accessible stopping time
$\sigma$ such that $\sigma < \tau$, for every $t \ge 0$,   a.s.
\begin{eqnarray}
\label{eq-locsol-2} u(t\wedge \sigma)&=&S(t\wedge
\sigma)u_0+\int_0^{t\wedge \sigma}
S(t\wedge\sigma-s) F(u(s)) ds\\
&+&\int_0^\infty
\mathds{1}_{[0,t\wedge\sigma]}S(t\wedge\sigma-s)G(u(s)) dW(s).
\nonumber\end{eqnarray}
\end{Rem}
Let us first formulate the following useful result.
\begin{prop}
\label{prop-local solution} Assume that a pair $(u,\tau)$ is a
local mild solution to problem \eqref{ABS-SPDE-strong}, where
$u_0$ is an   $V$-valued  $\mathcal{F}_0$ measurable random
variable. Then for
every finite stopping time $\sigma$, a pair $(u_{\vert [0, \tau
\wedge \sigma)\times \Omega}, \tau \wedge \sigma)$ is also a local
mild solution to problem \eqref{ABS-SPDE-strong}.
\end{prop}

Secondly, let us recall the following result,  see \cite[Lemmata III 6A and
6B]{Elw_1982}.
\begin{lem} \label{L4.2} \textbf{(The Amalgamation
Lemma) } Let $\Delta$ be a family of accessible stopping
times with values in $[0, \infty ]$. Then a function
\[ \tau := \sup \{ \alpha : \alpha \in \Delta \}\]
is an accessible stopping time with values in $[0, \infty ]$ and
there exists an $\Delta $-valued  increasing sequence $ \{
\alpha_{n} \}_{ n= 1 }^\infty $ such that
$ \tau $ is  the poitwise limit of $\alpha_n$.\\
Assume also that for each  $ \alpha \in \Delta $, $  I_{
\alpha } : [ 0, \alpha ) \times \Omega \to V $ is an admissible
process such that for all $ \alpha , \beta \in \Delta $ and
every $t>0$,
 \begin{equation} \label{eqn-amalgamation_01} I_{ \alpha }(t)= I_{ \beta } (t) \mbox{ a.s. on } \Omega_t( \alpha \wedge \beta) .
\end{equation}
Then, there exists an admissible process $ \mathbf{I}: [0, \tau )
\times \Omega \to V $,
 such that every $t>0$,
\begin{equation} \label{eqn-amalgamation_02}
\mathbf{I}(t)= I_{ \alpha } (t) \; \mbox{a.s. on} \; \Omega_t( \alpha ).
\end{equation}
Moreover, 
if ${\tilde I} :[0,\tau) \times \Omega \to X$ is any process
satisfying \eqref{eqn-amalgamation_02} then the process $\tilde I$
is a version of the process $I$, i.e. for any $t\in [0,\infty) $
\begin{equation} \label{eqn-amalgamation_03}
\mathbb{ P}\left(\left\{\omega \in \Omega: t < \tau(\omega) , \;
I(t,\omega)\not=  {\tilde I}(t,\omega) \right\} \right) =0 .
\end{equation}
In particular, if in addition ${\tilde I}$ is an admissible
process,  then
\begin{equation} \label{eqn-amalgamation_03'}
\mathbf{I}=  {\tilde I}.
\end{equation}


\end{lem}
\begin{Rem}\label{rem-amalgamation-equiv}
Let us note that because both processes $ \mathbf{I}: [0, \tau )
\times \Omega \to V $ and $ {I}_\alpha: [0, \alpha ) \times \Omega
\to V $ are admissible (and hence with almost sure continuous
trajectories), and since $\alpha \leq \tau$,  condition
\eqref{eqn-amalgamation_02} is equivalent to the following one:
\begin{equation} \label{eqn-amalgamation_02'}
\mathbf{I}_{\vert [0,\alpha)\times \Omega }= I_{ \alpha } .
\end{equation}
Similarly,  condition  \eqref{eqn-amalgamation_01} is equivalent
to the following one
 \begin{equation} \label{eqn-amalgamation_01'} {I_{ \alpha }}_{\vert [0,\alpha \wedge \beta )\times \Omega } = {I_{ \beta }}_{\vert [0,\alpha \wedge \beta )\times \Omega }.
\end{equation}

\end{Rem}

\begin{Def}\label{Def-maxsol}
Consider  a family  $\mathcal{ LS}$ of all local  mild solution
$(u,\tau)$ to  the problem \eqref{ABS-SPDE-strong}. For two
elements $(u,\tau), (v,\sigma) \in \mathcal{ LS} $ we write that
$(u,\tau)\preceq (v,\sigma)$ iff $\tau \leq \sigma$ a.s. and
$v_{\vert [0,\tau)\times \Omega} \sim u$. Note that if
$(u,\tau)\preceq (v,\sigma)$ and $(v,\sigma)\preceq (u,\tau)$,
then  $(u,\tau)\sim (v,\sigma)$. We write $(u,\tau)\prec
(v,\sigma)$ iff $(u,\tau)\preceq (v,\sigma)$ and $(u,\tau)\not\sim
(v,\sigma)$. Then, the pair $(\mathcal{ LS},\preceq)$ is a
partially ordered set in which, according to the  Amalgamation
Lemma \ref{L4.2},   every non-empty chain has an upper bound.
 \\
Each such a maximal element $(u,\tau)$ in the set $(\mathcal{
LS},\preceq)$
is called a maximal local  mild solution    to  the problem \eqref{ABS-SPDE-strong}. \\
If $(u, \tau)$ is a  maximal local mild solution to equation
\eqref{ABS-SPDE-strong}, the stopping time $\tau$ is called its
lifetime.

\end{Def}
A priori, there may be many maximal elements in $(\mathcal{
LS},\preceq)$ and hence many maximal local  mild solutions    to
the problem \eqref{ABS-SPDE-strong}. However, as we will see
later, if  the uniqueness of local solutions holds,  then the
uniqueness of the maximal local  mild solution will follow.

\begin{Rem} The following is an equivalent version of Definition \ref{Def-maxsol}.
 For  a local mild solution $(u,\tau)$ the following
conditions are equivalent.
\begin{trivlist}
\item[(nm1)] The pair  $(u,\tau)$ is not a maximal  local mild
solution to problem \eqref{ABS-SPDE-strong}. \item[(nm2)] There
exists a local mild solution $(v,\sigma)$  to problem
\eqref{ABS-SPDE-strong} such that $(u, \tau) \prec (v,\sigma)$.
\item[(nm3)] There exists a local mild solution $(v,\sigma)$  to
problem \eqref{ABS-SPDE-strong} such that $\tau \leq \sigma$ a.s.,
$v_{\vert [0,\tau)\times \Omega} \sim u$ and $\mathbb{P}(\tau <
\sigma)>0$. \item[(nm4)] Every  local mild solution $(v,\sigma)$
to problem \eqref{ABS-SPDE-strong} such that $(u, \tau) \not\sim
(v,\sigma)$ satisfies $(u, \tau) \not\prec (v,\sigma)$.
\end{trivlist}
\end{Rem}

\begin{Def}\label{Def-uniq} A local solution $(u,\tau)$
to problem \eqref{ABS-SPDE-strong}  is unique iff for any other
local solution $(v,\sigma)$  to \eqref{ABS-SPDE-strong} the
restricted processes $u_{[0, \tau\wedge \sigma)\times \Omega}$ and
$v_{[0, \tau\wedge\sigma)\times \Omega}$ are equivalent.
\end{Def}

\begin{prop}\label{prop-loc-implies-max}
Suppose that $u_0$ is a $V$-valued random variable and $\mathcal{F}_0$-measurable. Suppose that there
exist at least one local solution $(u^0,\tau^0)$ to problem
\eqref{ABS-SPDE-strong} and that for any two local solutions
$(u^1,\tau^2)$ and $(u^1,\tau^2)$, and every $t>0$,
 \begin{equation} \label{eqn-amalgamation_04} u^1(t)= u^2 (t) \mbox{ a.s. on }
\{ t< \tau^1 \wedge \tau^2\} .
\end{equation}
Then there exists a maximal  local mild solutions to the same
problem.
\end{prop}

\begin{Rem}\label{rem--loc-implies-max-equiv}
Let us note that similarly to Remark \ref{rem-amalgamation-equiv},
because both the local solutions $u^1$ and $u^2$ are admissible
(and hence with almost sure continuous trajectories),   condition
\eqref{eqn-amalgamation_04} is equivalent to the following one:
\begin{equation} \label{eqn-amalgamation_04'}
u^1_{\vert [0,\tau^1 \wedge \tau^2)\times \Omega }= u^2_{\vert
[0,\tau^1 \wedge \tau^2)\times \Omega } .
\end{equation}
\end{Rem}
\begin{proof}[Proof of Proposition \ref{prop-loc-implies-max}]
Consider  the family  $\mathcal{ LS}$ introduced in Definition
\ref{Def-maxsol}.  Note that in view of Proposition
\ref{prop-local solution}, if $(u,\tau)\in \mathcal{LS}$, then for
any $T>0$, $(u_{\vert [0,T\wedge \tau)\times \Omega},T\wedge
\tau)$ also belongs to $ \mathcal{LS}$. Let $\mathcal{ LC}$ be
any nonempty chain in $\mathcal{ LS}$ containing $(u^0,\tau^0)$.
Set
\[
\hat{\tau} := \sup\left\{ \tau: (u,\tau)\in \mathcal{ LC}\right\}
.
\]
Since  the chain $\mathcal{ LC}$ is non-empty,  from the
Amalgamation Lemma  \ref{L4.2} it follows that $\hat \tau$ is an
accessible stopping time and that there exists an admissible
$V$-valued process ${\hat u}(t)$, $t< {\hat \tau}$ such that for
all $(u,\tau) \in \mathcal{ LC}$, $(\hat{u}_{\vert [0,\tau)\times
\Omega},\tau) \sim (u,\tau)$. In view of the Kuratowski-Zorn Lemma it remains to prove that
the pair  $({\hat u},{\hat \tau})$ belongs to $\mathcal{ LS}$. To
prove this let us consider an $\mathcal{LC}$-valued sequence,
which exists in view of Lemma  \ref{L4.2},  $(u^n,\alpha_n)$ of
local mild solutions to problem \eqref{ABS-SPDE-strong} such that
$\alpha_n \toup \tau$ a.s. Moreover, by the above comment on
Proposition \ref{prop-local solution} we can assume that
$\alpha_n$ is a bounded from above stopping time. Moreover, since
each $(u^n,\alpha_n)$ is a local mild solution, we can find a
sequence $(\sigma_n)$ of local times such that $\sigma_n <
\alpha_n$, $\sigma_n \toup \tau$ a.s. and
for each $n$, the condition \eqref{eq-locsol_00} and equality  \eqref{eq-locsol_00-b} are satisfied by $u^n$. \\
Since $(u^n,\alpha_n)\in \mathcal{LC}$, we infer that
$(\hat{u}_{\vert [0,\alpha_n)\times \Omega},\alpha_n) \sim
(u^n,\alpha_n)$ and thus we infer that also for each $n$, the
condition \eqref{eq-locsol_00} and equality  \eqref{eq-locsol_00-b}
are satisfied by $\hat{u}$. This proves that
$(\hat{u},\hat{\tau})\in \mathcal{LC}$. The proof is complete.
\end{proof}

Now we shall prove a counterparts of Corollary 2.29 and Lemma 2.31
from \cite{Brz+Elw_2000}.
\begin{prop}\label{prop-t_infty}
 Assume that the Assumption \ref{assum-F}-Assumption \ref{assum-01} hold. If a pair  $(u, \tau)$ is a  maximal  local solution  then
 \begin{equation}\label{eqn-t_infty-limit}
\mathbb{ P}\Big(\{\omega \in\Omega :\tau(\omega)< \infty , \exists
\lim_{t\nearrow \tau(\omega)} u(t)(\omega)\in V\}\Big)=0
 \end{equation}
and
 \begin{equation}\label{eqn-t_infty}
\lim_{t\toup \tau}|u|_{X_t}=\infty \;\;\; \mathbb{P}-\mbox{a.s. on
} \{\tau<\infty\}.
 \end{equation}
\end{prop}
\begin{proof}[Proof of Proposition \ref{prop-t_infty} (by contradiction)]
\textbf{Part I.} Assume that there exists $\Omega_1 \subset \Omega
$ such that $\mathbb{ P}(\Omega_1)> 0$ and such that for any
$\omega\in\Omega_1$ we have $\tau(\omega) < \infty$ and
$\lim_{t\nearrow \tau(\omega)}u(t)(\omega) =\bar {v}(\omega)\in
V$.  Let us take an increasing   sequence $\{
\sigma_n\}_{n=1}^\infty$ of stopping times such that $\sigma_n <
\tau$ a.s. and $\sigma_n \toup \tau$ a.s. Since the function $f_n:
\Omega \ni (\omega) \mapsto u(\sigma_n(\omega))(\omega)\in V$, is
$\mathcal{F}_{\sigma_n}$--measurable, see e.g. \cite[Proposition
I.2.18]{Kar-Shr-96}, and hence $\mathcal{F}_{\tau}$--measurable,
the set
\[\Omega_2:=\{ \omega \in \Omega :  \tau(\omega) < \infty \; \mbox{\rm and }
 \lim_{n}f_n(\omega):=\bar
{v}(\omega) \mbox{\rm\, exists}\}\]
 belongs to the $\sigma$-algebra
  $\mathcal{F}_{\tau}$
and the function $\bar{v}$ is $\mathcal{F}_{\tau}$--measurable.
Note that obviously $\Omega_1 \subset \Omega_2$ so that $\mathbb{
P}(\Omega_2)> 0$. \\
Let   $v(t)$, $\tau (\omega)\le t < \tau_2(\omega)$,
  be the solution to \eqref{ABS-SPDE-strong} with
initial condition $v_{\tau(\omega)}(\omega)={\bar
v}(\omega)1_{\Omega_2}(\omega)$. Then,  by  \cite[Corollary 2.9]{Brz+Elw_2000}
 the process  $\tilde{u}(t)$ defined by
\[
\tilde{u}(t)(\omega)=
\begin{cases}
u(t)(\omega),& \text{if } t < \tau(\omega), \cr v(t)(\omega),&
\text{ if } \omega\in\Omega_2 \text{ and } t \in
[\tau(\omega),\tau_2(\omega))\cr
\end{cases}
\]
is a solution to \eqref{ABS-SPDE-strong}. Obviously this
contradicts the maximality of the solution $u(t)$.

\textbf{Part II.}
   Let us  assume that there exists $\eps>0$ such that \[{\mathbb P}\big( \{\tau<\infty\} \cap \{\limsup_{ t \toup \tau} \vert u\vert_{X_t}<\infty \}\big) =4\varepsilon.\]
   Since the function $t\mapsto \vert u\vert_{X_t}$ is increasing, it means that
\[ {\mathbb P}\big( \{\tau<\infty\} \cap \{\sup_{ t \in[0, \tau)} \vert u\vert_{X_t}<\infty \}\big) =4\varepsilon >0.\]
Hence there exists  $R>0$  such that
\[ {\mathbb P}\big( \{\tau <\infty\} \cap \{\sup_{ t<\tau_\infty} |u(t)|_{X_t}< R \}\big) =3\varepsilon >0.\]
Define (our definition implies that $\sigma_{R}=\tau$, possibly
$\infty$, when the set $\{ \omega\in\Omega:  |u|_{X_t}< R, t\in
[0,\tau) \}$ is empty)  \[\sigma_{R}= \inf\{ t\in [0,\tau] :
|u|_{X_t}\geq R\}. \] It is known that  $\sigma_{R}$ is a
predictable stopping time and  the  set
 $\tilde{\Omega}=\{ \sigma_{R}= \tau\} \cap\{ \tau <\infty\}$ is
 ${\mathcal F}_{\sigma_{R}}$-measurable. Note also that $\tilde{\Omega}=\{\tau <\infty\} \cap \{\sup_{ t<\tau_\infty} |u(t)|_{X_t}< R \} $ and so
 \[{\mathbb P}(\tilde{\Omega}) \geq 3\eps.\]
Next, define two processes $x$  and $y$ by
 \[
 x(t)=1_{\{t < \sigma_R\}}G(u(t)),\;\;\; y(t)=1_{\{t < \sigma_R\}}F(u(t)) \;\; t \geq 0.
 \]
Note that $x(t)=0$ or $\vert u\vert_{X_t}< R$.

Since $F(0)=0$ by assumptions, from \eqref{eqn-local Lipschitz-F}
we infer that

\begin{equation*}
\vert F(x) \vert_H \leq C   \Vert x\Vert^{p+1-\alpha} \vert
x\vert_E^\alpha+C \Vert x\Vert^{q+1-\gamma} \vert
x\vert_E^\gamma, \;\; x\in E.
\end{equation*}
Similarly, from \eqref{eqn-local Lipschitz-G}

\begin{equation*}
\Vert G(x) \Vert \leq C \Big[ 1 +\vert x\vert_E^\beta \Vert
x\Vert^{k+1-\beta} \Big], \;\; x\in E.
\end{equation*}

Let us now fix $T>0$. From elementary calculus we obtain
\begin{eqnarray*}
&&\int_0^T \vert y(t)\vert_{H}^2\, dt =  \mathbb{E} \int_0^T
1_{[0, \sigma_R)}(t) \vert F(u(t)) \vert_H^2\, dt
\leq J_{1,p}+J_{2,q},
\end{eqnarray*}
where $$J_{1,p}:=  C^2 \int_0^{T \wedge \sigma_R} \Vert u(t)\Vert^{2(p+1-\alpha)} \vert u(t)\vert_E^{2\alpha} \, dt,$$ and
$$ J_{2,q}:=c^2 \int_0^{T \wedge \sigma_R} \Vert u(t)\Vert^{2(q+1-\gamma)} \vert u(t)\vert_E^{2\gamma} \, dt.$$
We can easily check that
\begin{align*}
 J_{1,p} \leq & C^2 \Big[ \sup_{t \in [0,T \wedge \sigma_R)} \Vert
u(t)\Vert^{2(p+1-\alpha)} \int_0^{T \wedge \sigma_R}
\vert u(t)\vert_E^{2\alpha} \, dt \Big]\\
 \leq  & C^2  \Big[ \sup_{t \in [0,T \wedge \sigma_R)} \Vert
u(t)\Vert^{2(p+1-\alpha)} \big(\int_0^{T \wedge \sigma_R}
\vert u(t)\vert_E^{2} \, dt\big)^\alpha T^{1-\alpha} \Big]\\
\leq &  C^2 T^{1-\alpha}  \Big[ (1-\alpha) \sup_{t \in [0,T \wedge \sigma_R)} \Vert u(t)\Vert^{2(1+ \frac{p}{1-\alpha})} + \alpha \int_0^{T \wedge \sigma_R} \vert u(t)\vert_E^{2} \, dt  \Big]\\
\leq &  C^2 T^{1-\alpha}  \Big[ \vert u \vert_{X_{T \wedge
\sigma_R}}^{2(1+ \frac{p}{1-\alpha})} \Big] \leq  C^2 T^{1-\alpha}
R^{2(1+ \frac{p}{1-\alpha})}.
\end{align*}
With similar argument we can prove that
\begin{equation*}
 J_{2,q} \leq  C^2 T^{1-\gamma}
R^{2(1+ \frac{q}{1-\gamma})}.
\end{equation*}
In particular, we obtain that

\[ \mathbb{E} \int_0^T \vert y(t)\vert_{H}^2\, dt \leq  \mathbb{E} C^2 [T^{1-\alpha} R^{2(1+ \frac{p}{1-\alpha})}+T^{1-\gamma}
R^{2(1+ \frac{q}{1-\gamma})}]<\infty.
\]
In a very similar manner we can show that
\[ \mathbb{E} \int_0^T \Vert x(t)\Vert_{V}^2\, dt <\infty.
\]
This allows us to define a process $v$  by

\begin{eqnarray}\label{eqn-v}
v(t)&=&S(t)u_0+\int_0^{t} S(t-s) y(s) ds +\int_0^\infty
\mathds{1}_{[0,t)}S(t-s)x(s) \,  dW(s), \;\; t \geq 0.
\end{eqnarray}
The process $v$ is defined globally and it is $V$-valued
continuous.  Moreover, since $y=F(u)$ and $x=G(u)$, on
$[0,\sigma_R)$, we infer by \cite[Remark 2.27]{Brz+Elw_2000} that $u=v$ on
$[0,\sigma_R)$. In particular, there exist, a.s.
\[ \lim_{t\toup  \sigma_R} u(t) \mbox{ in } V.\]
In particular, the above limit  exists on the set
$\tilde{\Omega}$. Thus, a.s. on $\tilde{\Omega}$,
\[ \lim_{t\toup  \tau} u(t) \mbox{ in } V.\] Clearly, this contradicts the first part of the proposition.
\end{proof}

\subsection{An abstract result}
In this subsection we want to establish a general theorem of
existence of mild solution to the following abstract SPDEs
\begin{equation}\label{ABS-SPDE}
u(t)=S(t)u_0+\int_0^t S(t-s) F(u(s)) ds+\int_0^t S(t-s) G(u(s))
dW(s),\;\; \mbox{ for any }t>0,
\end{equation}
where $S(t)$, $t\in [0,\infty)$ is a semigroup, $F$ and $G$ are
nonlinear map satisfying Assumption \ref{assum-01}, Assumption
\ref{assum-F}, and Assumption \ref{assum-G}, respectively.

 Let
$\theta:\mathbb{R}_+\to [0,1]$ be a ${\mathcal C}^\infty_c$ non
increasing function
 such that
\begin{equation}\label{eqn-theta} \inf_{x\in\mathbb{R}_+}\theta^\prime(x)\geq -1, \quad \theta(x)=1\;
\mbox{\rm  iff } x\in [0,1]\quad \mbox{\rm  and } \theta(x)=0 \;
\mbox{\rm  iff } x\in [2,\infty).
\end{equation}
and for $n\geq 1$ set  $\theta_n(\cdot)=\theta(\frac{\cdot}{n})$.
Note that if $h:\mathbb{R}_+\to\mathbb{R}_+$ is a non decreasing
function, then for every $x,y\in {\mathbb R}$,
\begin{equation}\label{ineq-theta}
 \theta_n(x)h(x) \leq h(2n),\quad
\vert \theta_n(x)-\theta_n(y)\vert \leq \frac1n |x-y| . 
\end{equation}
\begin{prop}\label{prop-global Lipschitz-F}
Let $F$ be a nonlinear mapping satisfying Assumption
\ref{assum-F}. Let us consider a map
\[
\Phi_T^n=\Phi_T: X_T \ni u \mapsto \theta_n( \vert u
\vert_{X_\cdot}) F(u) \in L^2(0,T;H).
\]
Then $\Phi_T^n$ is globally Lipschitz and moreover, for any
$u_1,u_2 \in X_T$,
\begin{eqnarray}\label{eqn-global Lipschitz-F}
\vert \Phi_T^n(u_1)-\Phi_T^n(u_2)\vert_{L^2(0,T;H)} &\leq & C\Big[   2n C +1\Big]
\Big[(2n)^{p+1} T^{(1-\alpha)/2} +(2n)^{q+1} T^{(1-\gamma)/2} \Big]\vert  u_1 -u_2
\vert_{X_T}.
\end{eqnarray}
\end{prop}
The proof is based on a proof from the paper
\cite{Brz+Millet_2012} which in turn was based on a proof from
papers by De Bouard and Debussche \cite{deBouard+Deb_1999,deBouard+Deb_2003}.
For simplicity, we will write $\Phi_T$ instead of $\Phi_T^n$.
\begin{proof}
Note that $\Phi_T(0)=0$. Assume that $u_1,u_2 \in X_T$. Denote,
for $i=1,2$,
\[
\tau_i= \inf\{t \in [0,T]: \vert u_i \vert_{X_t} \geq 2n\}.
\]
Note that by definition, if the set on the RHS above is empty,
then $\tau_i=T$.

Without loss of generality we can assume that $\tau_1 \leq
\tau_2$. We have the following chain of inequalities/equalities
\begin{eqnarray*}
\vert \Phi_T(u_1)-\Phi_T(u_2)\vert_{L^2(0,T;H)} &=& \Big[ \int_0^T \vert  \theta_n( \vert u_1 \vert_{X_t}) F(u_1(t))-\theta_n( \vert u_2 \vert_{X_t}) F(u_2(t))\vert_H^2\,dt\Big]^{1/2}\\
&&\mbox{ because for } i=1,2, \;\; \theta_n( \vert u_i
\vert_{X_t}) =0 \mbox{ for } t \geq \tau_2
\\
&=& \Big[ \int_0^{\tau_2} \vert  \theta_n( \vert u_1 \vert_{X_t})
F(u_1(t))-\theta_n( \vert u_2 \vert_{X_t})
F(u_2(t))\vert_H^2\,dt\Big]^{1/2}
\\
&&\hspace{-5truecm}\lefteqn{= \Big[ \int_0^{\tau_2} \vert \big[
\theta_n( \vert u_1 \vert_{X_t}) - \theta_n( \vert u_2
\vert_{X_t}) \big] F(u_2(t))+ \theta_n( \vert u_1
\vert_{X_t})\big[  F(u_1(t))- F(u_2(t)) \Big]
\vert_H^2\,dt\Big]^{1/2} }
\\
& &\hspace{-3truecm}\lefteqn{ \leq
\Big[ \int_0^{\tau_2} \vert  \big[ \theta_n( \vert u_1 \vert_{X_t}) - \theta_n( \vert u_2 \vert_{X_t}) \big] F(u_2(t))\vert_H^2 \,dt \Big]^{1/2} }\\
&&\hspace{-2truecm}\lefteqn{+ \Big[ \int_0^{\tau_2} \vert
\theta_n( \vert u_1 \vert_{X_t})\big[  F(u_1(t))- F(u_2(t)) \big]
\vert_H^2\,dt\Big]^{1/2} =:A+B}
\end{eqnarray*}
Next, since $\theta_n$ is Lipschitz with Lipschitz constant $2n$
we have
\begin{eqnarray*}
A^2&=& \int_0^{\tau_2} \vert  \big[ \theta_n( \vert u_1 \vert_{X_t}) - \theta_n( \vert u_2 \vert_{X_t}) \big] F(u_2(t))\vert^2 \,dt\\
&\leq& 4n^2 C^2 \int_0^{\tau_2} \big[ \vert
\vert u_1 \vert_{X_t} -  \vert u_2 \vert_{X_t} \vert\big]^2 \vert F(u_2(t))\vert_H^2 \,dt\\
&&\mbox{by Minkowski inequality}\\
&\leq& 4n^2 C^2 \int_0^{\tau_2}  \vert
 u_1 -u_2 \vert_{X_t}^2 \vert F(u_2(t))\vert_H^2 \,dt \leq  4n^2 C^2 \int_0^{\tau_2}  \vert
 u_1 -u_2 \vert_{X_T}^2 \vert F(u_2(t))\vert_H^2 \,dt \\
&\leq & 4n^2 C^2 \vert
 u_1 -u_2 \vert_{X_T}^2 \int_0^{\tau_2}   \vert F(u_2(t))\vert_H^2 \,dt
\end{eqnarray*}
Next, by assumptions
\begin{eqnarray*}
\int_0^{\tau_2}   \vert F(u_2(t))\vert_H^2 \,dt  &\leq & C^2  \int_0^{\tau_2}  \Vert  u(t)\Vert^{2p+2-2\alpha}  \vert  u(t)\vert_E^{2\alpha} \,dt+
C^2  \int_0^{\tau_2}  \Vert  u(t)\Vert^{2q+2-2\gamma}  \vert  u(t)\vert_E^{2\gamma} \,dt \\
&& \quad \quad := I_{1,p}+ I_{2,q}.
\end{eqnarray*}
By some elementary calculations we obtain
\begin{align*}
I_{1,p} \leq & C^2 \sup_{t \in [0,\tau_2]} \Vert u(t)\Vert^{2p+2-2\alpha}
\big( \int_0^{\tau_2} \vert u(t)\vert_E^{2} \,dt\big)^\alpha
\tau_2^{1-\alpha}
\\
\leq & C^2  \tau_2^{1-\alpha} \vert u \vert_{X_{\tau_2}}^{2p+2} \leq  C^2  \tau_2^{1-\alpha}(2n)^{2p+2},
\end{align*}
and
\begin{align*}
I_{2,q} \leq & C^2 \sup_{t \in [0,\tau_2]} \Vert u(t)\Vert^{2q+2-2\gamma}
\big( \int_0^{\tau_2} \vert u(t)\vert_E^{2} \,dt\big)^\gamma
\tau_2^{1-\gamma}
\\
\leq & C^2  \tau_2^{1-\gamma} \vert u \vert_{X_{\tau_2}}^{2q+2} \leq  C^2  \tau_2^{1-\gamma}(2n)^{2q+2}.
\end{align*}
Therefore,
\begin{eqnarray*}
A&\leq & C^2  [ \tau_2^{(1-\alpha)/2}(2n)^{p+2} + \tau_2^{(1-\gamma)/2}(2n)^{q+2} ]\vert  u_1 -u_2
\vert_{X_T} .
\end{eqnarray*}

Also, because $ \theta_n( \vert u_1 \vert_{X_t}) =0$ for  $ t \geq
\tau_1$, and $\tau_1 \leq \tau_2$, we have
\begin{eqnarray*}
B&=& \Big[ \int_0^{\tau_2} \vert   \theta_n( \vert u_1 \vert_{X_t})\big[  F(u_1(t))- F(u_2(t)) \big] \vert_H^2\,dt \Big]^{1/2}\\
&=& \Big[ \int_0^{\tau_1} \vert   \theta_n( \vert u_1 \vert_{X_t})\big[  F(u_1(t))- F(u_2(t)) \big] \vert_H^2\,dt\Big]^{1/2}\\
&&\mbox{because } \theta_n( \vert u_1 \vert_{X_t}) \leq 1 \mbox{ for } t\in [0,\tau_1) \\
&\leq& \Big[ \int_0^{\tau_1} \vert     F(u_1(t))- F(u_2(t))  \vert_H^2\,dt\Big]^{1/2}\\
&\leq & \tilde{B}_{,p}+\tilde{B}_{,q},
\end{eqnarray*}
where
\begin{equation*}
 \begin{split}
  \tilde{B}_{,p}:=
C \Big[ \int_0^{\tau_1}  \vert u_1(t)-u_2(t)\vert_E^{2\alpha} \Vert u_1(t)-u_2(t)\Vert^{2-2\alpha} \Vert u_2(t)\Vert^{2p} \,dt\Big]^{1/2}\\
+C \Big[ \int_0^{\tau_1}     \Vert u_1(t)-u_2(t)\Vert^2 \Vert u_1(t)\Vert^{2p-2\alpha} \vert u_1(t)\vert_E^{2\alpha}   \,dt\Big]^{1/2},
 \end{split}
\end{equation*}
and
\begin{equation*}
 \begin{split}
  \tilde{B}_{,q}:=
C \Big[ \int_0^{\tau_1}  \vert u_1(t)-u_2(t)\vert_E^{2\gamma} \Vert u_1(t)-u_2(t)\Vert^{2-2\gamma} \Vert u_2(t)\Vert^{2q} \,dt\Big]^{1/2}\\
+C \Big[ \int_0^{\tau_1}     \Vert u_1(t)-u_2(t)\Vert^2 \Vert u_1(t)\Vert^{2p-2\gamma} \vert u_1(t)\vert_E^{2\gamma}   \,dt\Big]^{1/2}.
 \end{split}
\end{equation*}
Now we have
\begin{eqnarray*}
\tilde{B}_{,p} &\leq& C \sup_{t\in [0,\tau_1]} \Vert u_1(t)-u_2(t)\Vert \Vert
u_1(t)\Vert^{p-\alpha} \Big[ \int_0^{\tau_1} \vert
u_1(t)\vert_E^{2\alpha}   \,dt\Big]^{1/2}\\
&+& C \sup_{t\in [0,\tau_1]} \Vert u_1(t)-u_2(t)\Vert^{1-\alpha}
\Vert u_2(t)\Vert^p \Big[ \int_0^{\tau_1} \vert
u_1(t)-u_2(t)\vert_E^{2\alpha} \,dt\Big]^{1/2}
\\
&\leq&  C \sup_{t\in [0,T]} \Vert u_1(t)-u_2(t)\Vert \sup_{t\in
[0,\tau_1]} \Vert u_1(t)\Vert^{p-\alpha} \Big[ \int_0^{\tau_1}
\vert u_1(t)\vert_E^{2}   \,dt\Big]^{\alpha/2}
\tau_1^{(1-\alpha)/2}
\\
&+& C \sup_{t\in [0,T]} \Vert u_1(t)-u_2(t)\Vert^{1-\alpha}
\sup_{t\in [0,\tau_1]}  \Vert u_2(t)\Vert^p \Big[ \int_0^{\tau_1}
\vert u_1(t)-u_2(t)\vert_E^{2}   \,dt\Big]^{\alpha/2}
\tau_1^{(1-\alpha)/2}
\\
&\leq&  C \vert u_1-u_2\vert_{X_T}  \vert u_1 \vert_{X_{\tau_1}}^p \tau_1^{(1-\alpha)/2} +C \vert u_1-u_2\vert_{X_T}  \vert u_2 \vert_{X_{\tau_1}}^p \tau_1^{(1-\alpha)/2}\\
&&\mbox{ because} \vert u_1 \vert_{X_{\tau_1}} \leq 2n \mbox{ and } \vert u_2 \vert_{X_{\tau_1}} \leq \vert u_2 \vert_{X_{\tau_2}}\leq 2n\\
&\leq&   C \tau_1^{(1-\alpha)/2} \vert u_1-u_2\vert_{X_T} \Big[
\vert u_1 \vert_{X_{\tau_1}}^p  +   \vert u_2
\vert_{X_{\tau_1}}^p\Big] \leq C (2n)^{p+1} \tau_1^{(1-\alpha)/2}
\vert u_1-u_2\vert_{X_T}.
\end{eqnarray*}
In exactly the same way,
\begin{equation*}
 \tilde{B}_{,q}\le C (2n)^{q+1} \tau_1^{(1-\gamma)/2}
\vert u_1-u_2\vert_{X_T}.
\end{equation*}

Summing up, we proved

\begin{eqnarray*}
\vert \Phi_T(u_1)-\Phi_T(u_2)\vert_{L^2(0,T;H)} &\leq & C^2[
\tau_2^{(1-\alpha)/2}(2n)^{p+2}+\tau_2^{(1-\gamma)/2}(2n)^{q+2}] \vert  u_1 -u_2 \vert_{X_T}\\
&& +C [(2n)^{p+1} \tau_1^{(1-\alpha)/2} +(2n)^{q+1} \tau_1^{(1-\gamma)/2} ]\vert u_1-u_2\vert_{X_T}\\
&=& C\Big[   2n C +1\Big] \Big[(2n)^{p+1}  \tau_2^{(1-\alpha)/2} +(2n)^{q+1}  \tau_2^{(1-\gamma)/2} \Big]\vert
u_1 -u_2 \vert_{X_T}
\end{eqnarray*}
The proof is complete.
\end{proof}
In what follows the function $\Phi_T^n$ from Proposition
\ref{prop-global Lipschitz-F} will be denoted by $\Phi_{F,T}^n$.
The following result is a special case of Proposition
\ref{prop-global Lipschitz-F} with $H=V$.
\begin{cor}\label{cor-global-lip-G}
Let $G$ be a nonlinear mapping satisfying Assumption
\ref{assum-G}. Define a map $\Phi_{G}^n$ by
\begin{equation}\label{eqn-Phi_G}
\Phi_{G}=\Phi_{G}^n=\Phi_{G,T}^n: X_T \ni u \mapsto \theta_n(
\vert u \vert_{X_\cdot}) G(u) \in L^2(0,T;V).
\end{equation}
Then $\Phi_{G,T}^n$ is globally Lipschitz and moreover, for any
$u_1,u_2 \in X_T$,
\begin{eqnarray}\label{eqn-global LipschitzG}
\vert \Phi_{G,T}^n(u_1)-\Phi_{G,T}^n(u_2)\vert_{L^2(0,T;V)} &\leq
& C_G (2n)^{k+1}\big[   2n C_G +1\big]  T^{(1-\beta)/2} \vert  u_1
-u_2 \vert_{X_T}.
\end{eqnarray}
\end{cor}

\begin{prop}\label{prop-Psi_T-Lipschitz}
Let Assumption \ref{assum-F}-Assumption \ref{assum-01} hold.
Consider a map
\begin{equation}\label{eqn-Psi_T}
\Psi_T^n=\Psi_T: M^2(X_T) \ni u \mapsto  Su_0 + S \ast \Phi_{F,T}^n
(u)+ S\diamond \Phi_{G,T}^n(u) \in M^2(X_T)
\end{equation}
Then $\Psi_T$ is globally Lipschitz and moreover, for any $u_1,u_2
\in M^2(X_T)$,
\begin{eqnarray}\label{eqn-global Lipschitz}
\vert \Psi_{T}^{n}(u_1)-\Psi_{T}^{n}(u_2)\vert_{M^2(X_T)} &\leq &\hat{C}(n)\Big[T^{(1-\alpha)/2}\vee T^{(1-\gamma)/2}\vee T^{(1-\beta)/2} \Big]\vert  u_1 -u_2
\vert_{M^2(X_T)},
\end{eqnarray}
where
\[ \hat{C}(n)= C_1 C_F\big[   2n C_F +1\big] \Big[(2n)^{p+1}+(2n)^{q+1}\Big]+ C_2C_G (2n)^{k+1}\big[   2n C_G +1\big].
\]
\end{prop}
\begin{proof}For simplicity of notation we will write
$\Psi_T$ instead of $\Psi_{T}^n$. We will also write $\Psi_F$ (resp. $\Psi_G$) instead of $\Psi_{F,T}$ (resp. $\Psi_{G,T}$).

Obviously in view of Assumption \ref{assum-01} the map $\Psi_T$ is
well defined. Let us fix  $u_1,u_2 \in M^2(X_T)$. Then
\begin{eqnarray*}
\vert \Psi_T(u_1)-\Psi_T(u_2)\vert_{M^2(X_T)} &\leq & \vert S\ast \Phi_F(u_1)-S\ast \Phi_F(u_2)\vert_{M^2(X_T)}\\
&+& \vert S \diamond  \Phi_G(u_1)-S\diamond \Phi_G(u_2)\vert_{M^2(X_T)}\\
&\leq& C_1 \vert  \Phi_F(u_1)- \Phi_F(u_2)\vert_{M^2(0,T;H)}+C_2
\vert  \Phi_G(u_1)-\Phi_G(u_2)\vert_{M^2(X_T)}
\\
&\leq&\Big[ C_1 C_F \big[   2n C_F +1\big]\Big[(2n)^{p+1} T^{(1-\alpha)/2}+ (2n)^{q+1} T^{(1-\gamma)/2}\Big] \\
 &+&  C_2C_G (2n)^{k+1}\big[   2n C_G +1\big] T^{(1-\beta)/2} \Big]
\\
&\leq& \hat{C}(n)\Big[T^{(1-\alpha)/2}\vee T^{(1-\gamma)/2}\vee T^{(1-\beta)/2} \Big],
\end{eqnarray*}
where
\[ \hat{C}(n)= C_1 C_F\big[   2n C_F +1\big] \Big[(2n)^{p+1}+(2n)^{q+1}\Big]+ C_2C_G (2n)^{k+1}\big[   2n C_G +1\big].
\]
The proof is complete.
\end{proof}
 The
first main result of this subsection is given in the following
theorem.
\begin{thm}\label{thm_local} Suppose that Assumption \ref{assum-F}-Assumption \ref{assum-01} hold. Then for every
$\mathcal{F}_{0}$-measurable $V$-valued square integrable random
variable $u_0$    there exits  a  local process $u=\big(u(t),
t\in[0,T_1) \big) $  which is the
  unique local mild solution to our  problem.
Moreover,  given $R>0$ and  $\varepsilon >0$ there exists
$\tau(\varepsilon,R)>0$,  such that for every
$\mathcal{F}_0$-measurable $V$-valued random variable $u_0$
satisfying  $\mathbb{E}\Vert u_0 \Vert^{2} \leq R^{2}$, one has
\[{\mathbb P}\big(T_1\geq \tau(\varepsilon,R)\big) \geq
1-\varepsilon.\]
 \end{thm}
\begin{proof}
Without loss of generality we can assume that $p\ge q\vee k$.
Owing to Proposition \ref{prop-Psi_T-Lipschitz} we can argue as in
\cite[Proposition 5.1]{Brz+Millet_2012}  to prove the first part
of the theorem. {For any $n\in \mathbb{N}$, $T> 0$ and $u_0\in
L^2(\Omega, \mathbb{P}; V)$ let $\Psi^n_{T,u_0}$ be the mapping
from $M^2(X_T)$ defined by
\begin{align*}
\Psi^n_{T,u_0}(u)=&S_t u_0+\int_0^t S_{t-r}[\theta_n
(|u|_{X_r})F(u(r))]dr+\int_0^t S_{t-r}[\theta_n
(|u|_{X_r})G(u(r))]dW(r).
\end{align*}
It follows from Proposition \ref{prop-global Lipschitz-F},
Corollary \ref{cor-global-lip-G} and Assumption
 \ref{assum-01} that $\Psi^n_{T,u_0}$ maps
$M^2(X_T)$ into itself. From
 Proposition \ref{prop-Psi_T-Lipschitz} we deduce that $\Psi^n_{T,u_0}$ is globally Lipschitz, moreover it is a strict contraction for small $T$.
 Therefore we can find $\delta>0$ such that $\Psi^n_{T,u_0}$ is $\frac 12 $-contraction. For $k\in \mathbb{N}$ let $(t\wedge \tau_k)_{k\in \mathbb{N}}$
  be a sequence of times defined by $t\wedge \tau_k=k \delta$. By the $\frac12$-contraction property of $\Psi^n_{T,u_0}$ we} {can find $u^{[n,1]}\in M^2(X_\delta)$ such that
  $u^{[n,1]}=\Psi^n_{\delta,u_0}(u^{[n,1]})$. Since $u^{[n,1]}\in M^2(X_\delta)$ it follows that $u^{[n,1]}(t)$ is $\mathcal{F}_t$-measurable  and
  $u^{[n,1]}(t)\in L^2(\Omega, \mathbb{P};V)$ for any $t\in [0,\delta]$. Thus replacing $u_0$ with $u^{[n,1]}(\delta)$ and using the same argument as
   above we can find
  $u^{[n,2]}\in M^2(X_\delta)$ such that $u^{[n,2]}=\Psi^n_{\delta, u^{[n,1]}(\delta)}(u^{[n,2]})$.
  By induction we can construct a sequence $u^{[n,k]}\subset M^2(X_\delta)$ such that $u^{[n,k]}=\Psi_{\delta, u^{[n,k-1]}}(u^{[n,k]})$.
  Now let  $u^n$ be the process defined by $u^n(t)=u^{[n,1]}(t)$, $t \in [0,\delta)$, and for $k=[\frac T \delta]+1$ and $0\le t<\delta$,
  let $u^{n}(t+k\delta)=u^{[n,k]}(t)$. By construction $u^n\in M^2(X_T)$ and $u^n= \Psi^n_{T,u_0}(u^n)$, consequently $u^n$ is a global solution to
  the truncated equation
$$u(t)=S_t u_0+\int_0^t S_{t-r}[\theta_n (|u|_{X_r})F(u(r))]dr+\int_0^t
S_{t-r}[\theta_n (|u|_{X_r})G(u(r))]dW(r).$$ Arguing exactly as in
\cite[Theorem 4.9]{Brz+Millet_2012} we can show that it is unique.
Now let $(\tau_n)_{n\in \mathbb{N}}$ be the sequence of stopping
times defined by
\begin{equation}\label{Loc-stop}
\tau_n=\inf\{t\in [0,T]: |u^n|_{X_t}\ge n\}.
\end{equation}
By definition $\theta_n(|u^n|_{X_{r}})=1$ for $r\in [0,t\wedge
\tau_n)$, hence
$$\theta_n (|u^n|_{X_r})F(u^n(r)) = F(u^n(r)), r\in [0,t\wedge
\tau_n).$$ From \cite[Lemma A.1]{ZB et al 2005} we infer that
$$\int_0^{t\wedge \tau_n} S_{t\wedge \tau_n-r}[\theta_n
(|u|_{X_r})G(u(r))]dW(r)=\int_0^{t\wedge \tau_n} S_{t\wedge \tau_n
-r} G(u^n(r)) dW(r). $$ From these remarks we easily deduce that
$u^n$ satisfies
\begin{equation*}
u^n(t\wedge \tau^n)=S_{t\wedge \tau_n}u_0+\int_0^{t\wedge \tau_n}
S_{t\wedge \tau_n-r} F(u^n(r)) dr + \int_0^{t\wedge \tau_n}
S_{t\wedge \tau_n-r} G(u^n(r))dW(r).
\end{equation*}
Since $\tau_n$ is an accessible stopping time it follows that the
process $(u^n(t), t<\tau_n)$ is a local solution to
\eqref{ABS-SPDE-strong}. This ends the proof of the first part of
the theorem.}

For the second part we will follow the lines of \cite[Theorem
5.3]{Brz+Millet_2012}. For this we fix $\varepsilon
>0$ and choose $N$
 such that $N\geq 2 \varepsilon^{-1/2}$.
 Thanks to Proposition \ref{prop-global Lipschitz-F}, Corollary \ref{cor-global-lip-G} and Assumption
 \ref{assum-01} we can deduce
 that there exist some positive constants $\tilde{C}_i, i=1,\dots,
 4$ such that for any $u\in M^2(X_T)$ we have
\begin{align*}
 \lvert S\ast \Phi_T^n(u)\rvert_{M^2(X_T)}
\le \tilde{C}_1 \tilde{C}_2(2n \tilde{C}_2+1)\Big[ (2n)^{p+2}T^{(1-\alpha)/2}+(2n)^{q+2}T^{(1-\gamma)/2}\Big],\\
\lvert S\diamond \Phi^n_G(u)\rvert_{M^2(X_T)}\le \tilde{C}_3
\tilde{C}_4 (2n)^{k+2} [2n \tilde{C}_4+1] T^{(1-\beta)/2}.
\end{align*}
Since $\alpha, \beta, \gamma \in [0,1)$ we infer from these inequalities
that there exists a sequence $(K_n(T))_n$ of numerical functions
with $\lim_{T\rightarrow 0} \sup_n  K_n(T)=0$ and
\begin{equation*}
\lvert S\ast \Phi^n_T(u)+ S\diamond
\Phi^n_G(u)\rvert_{M^2(X_T)}\le K_n(T),
\end{equation*}
for any $u\in M^2(X_T)$.
Let us put $n=N R$    for some ``large" $N$ to be chosen later
and
 choose $\delta_{1}(\varepsilon,R) >0$ such that
$K_{n}(\delta_{1}(\varepsilon,R)) \leq  R $. Let $\Psi^n_T$ be the
mapping defined by \eqref{eqn-Psi_T}. Since $\me \lve u_0\rve^2\le
R^2$, we infer by the Assumption \ref{assum-01} (namely
\eqref{ineq-dc}) that
\begin{align*}
\lvert \Psi^n_T(u)\rvert_{M^2(X_T)}\le & C_0 R+ K_n(T),\\
\le & (C_0+ 1)R,
\end{align*}
for any $T\le \delta_1(\eps,R)$. That is, for $T\le
\delta_1(\eps,R)$
 the range of $\Psi_T^{n}$ is included in the ball centered at 0 and of radius  $(C_0+1)R$ of ${M}^2(X_T)$. Furthermore, Propositions \ref{prop-Psi_T-Lipschitz}
implies that there exists $C>0$ such that for any $u_1,u_2\in
M^2(X_T)$
\begin{equation*}
\lvert \Psi^n_T(u_1)-\Psi^n_T(u_2)\rvert_{M^2(X_T)}\le C
N^{p+1}R^{p+1} (NRC+1)\Big[T^{(1-\alpha)/2} \vee T^{(1-\gamma )/2}\vee T^{(1-\beta)/2}\Big] \lvert
u_1-u_2\rvert_{M^2(X_T)}.
\end{equation*}
Hence we can find $\delta_2(\eps,R)>0$ such that $\Psi^n_T$ is a
strict contraction for any $T\le \delta_2(\eps,R)$. Thus if one
puts $\tau(\eps,R)=\delta_1(\eps,R)\wedge \delta_2(\eps,R)$, the
mapping $\Psi^n_T$ has a unique fixed point $u^n$ which satisfies
\begin{equation*}
\me \lvert u^n\rvert^2_{X_{\tau(\eps,R)}}\le (C_0+1)^2 R^2.
\end{equation*}
As in \cite[Proposition 5.1]{Brz+Millet_2012} we can show that
$(u^n(t), t\le \tau_n)$ is a local solution to problem
\eqref{ABS-SPDE}. By the definition of the stopping time $\tau_n$
the set $\{\tau_n \le \tau(\eps,R)\}$ is contained in the set $\{
\lvert u^n\rvert_{X_{\tau(\eps,R)}}\ge n \}$. By Chebychev's
inequality we have
\begin{align*}
\mathbb{P}(\tau_n \le \tau(\eps,R))\le &\mathbb{P}(\lvert
u^n\rvert^2_{X_{\tau(\eps,R)}}\ge n ),\\
\le & (C_0+1)^2 N^{-2}, \\
\le & \eps,
\end{align*}
for $N\ge (C_0+1) \eps^\frac 12$. Therefore for $N\ge (C_0+1)
\eps^\frac 12$ we have $\mathbb{P}(\tau_n \ge \tau(\eps,R))\ge
1-\eps$
 and the stopping time $T_1=\tau_{n}$ satisfies the requirements of the theorem. This  concludes the proof.
\end{proof}
Let $(\tau_n)_{n\in \mathbb{N}}$ be the sequence of stopping times defined by \eqref{Loc-stop} and
$$\tau_\infty(\omega)=\lim_{n\rightarrow \infty} \tau_n(\omega), \,\, \omega \in \Omega.$$ Let also $(u(t), t<\tau_\infty)$ be the  local process defined by $$ u(t,\omega)=u^n(t,\omega) \text{ if } t<\tau_n(\omega), \omega \in \Omega.$$
The next result is about the existence and uniqueness of a maximal solution and the characterization of its lifespan.

 \begin{thm}\label{thm_maximal-abstract}
For every      $u_0\in L^2(\Omega,\mathcal{F}_0,V)$,
 the process $u=(u(t)\, ,\, t<  \tau_\infty) $ defined above
   is the unique local maximal solution to our equation.
  Moreover,
$ {\mathbb P }\big(\{\tau_\infty <\infty\} \cap \{\sup_{
t<\tau_\infty} |u(t)|_{V}<\infty \}\big)=0$ and on $\{
\tau_\infty<\infty \}$, $\limsup_{t\to \tau_\infty} |u(t)|_{V} =
+\infty$ a.s.
\end{thm}
\begin{proof}
One can use similar argument as in \cite[Theorem 4.10]{ZB-97} to show that $(u,\tau_\infty)$ is a local solution to our problem.
Since $\tau_n\toup \tau_\infty$ we have $\tau_n\le
\tau_\infty<\infty$ on $\{\tau_\infty<\infty\}$  for any $n$. This
implies in particular that for any $n>0$ there exists a constant $\delta$ such that $\lvert u\rvert_{X_{t}}\ge \lvert u \rvert_{X_{\tau_n}}\ge  n$ whenever $t-\tau_\infty<\delta $ on $\{ \tau_\infty<\infty\}$. This yields that $\lim_{t\toup \tau_\infty}\lvert u\rvert_{X_{t}}\rightarrow \infty$ on $\{\tau_\infty<\infty \}$.  This concludes that $u=(u(t)\, ,\, t< \tau_\infty)
$ is a maximal solution.

To prove the second statement we argue by contradiction. Assume
that for some $\eps>0$,  $$ {\mathbb P }\big(\{\tau_\infty
<\infty\} \cap \{ t\in [0,\tau_\infty) |u(t)|_{V}<\infty
\}\big)=4\eps>0.$$ Let $R>0$ and assume that
\[ {\mathbb P }\big(\{\tau_\infty <\infty\} \cap \{ |u(t)|_{V}<R \mbox{ for all } t\in [0,\tau_\infty)
\}\big) \geq 3\eps.\]
 Let $\sigma_R=\inf \{t \in [0,\tau_\infty) : \lvert u(t)\rvert_{V}\ge
R\}  $ and $\tilde{\Omega}=\{\sigma_R=\tau_\infty<\infty\}.$

Note  that
\[ \tilde{\Omega}= \{\tau_\infty <\infty\} \cap \{ |u(t)|_{V}<R \mbox{ for all } t\in [0,\tau_\infty)
\} \]

With this $R$ and previous $\eps$ we can choose a number
$\tau(\eps,R)>0$ as in Theorem \ref{thm_local}. Let us now choose
$\alpha$ such that  $\alpha= \frac 12 \tau(\eps,R)$. By
construction $\tau_n\toup \tau_\infty$ almost surely and  hence
for arbitrary  $\delta>0$  there exists $n_0>0$ such that
$\mathbb{P}(\Omega_0)\geq (1-\delta)\mathbb{P}(\tilde{\Omega})$,
where $\Omega_0:=\{\omega \in \tilde{\Omega}:
\tau_\infty-\tau_{n_0}<\alpha\}$. Choosing $\delta=\frac13$ we get
$\mathbb{P}(\Omega_0)\geq 2\eps$.

Let $T_0=\tau_{n_0}$ and
\begin{equation*}
y_0=\begin{cases} u(T_0) \text{ on }
{\Omega_0},\\
0 \text{ otherwise}.
\end{cases}
\end{equation*}
Note that $\mathbb{E}(\vert y_0\vert_{V}^2 )\leq R^2$.
Then, thanks to Theorem \ref{thm_local} the problem
\eqref{ABS-SPDE} with initial condition (starting at $T_0$) $y_0$
has a unique local solution denoted by $y(t)$, $t\in
[T_0,T_0+T_1)$. Moreover, the lifespan $T_1$ of the process
$y(\cdot)$ satisfies $\mathbb{P}(T_1\geq \tau(\eps,R))>1-\eps$.
From the first part of our theorem $y$ is the unique maximal
solution of \eqref{ABS-SPDE} with initial condition $y_0$.
 Note that
$\mathbb{P}(\tau_\infty-T_0< \frac 12 \tau(\eps,R))\geq 2\eps$.
Hence, we infer that
\[\mathbb{P}\big(\Omega_1  \big) \geq \eps>0,\]
where
\[\Omega_1:= \Omega_0 \cap \{ T_1\geq \tau(\eps,R)\}.\]

Next we define a local stochastic process $v$ by
\begin{equation*}
v(t,\omega)=\begin{cases} u(t,\omega) \text{ if } \omega\in \Omega_1^c,\\
y(t,\omega) \text{ if }\omega\in \Omega_1 \mbox{ and } t> T_0,\\
u(t,\omega) \text{ if }\omega\in \Omega_1 \mbox{ and } t\in [0,
T_0],
\end{cases}
\end{equation*}
The process $v(\cdot)$ is a local solution of \eqref{ABS-SPDE}
with initial condition $u_0$. Since $$\mathbb{P}(T_0+\frac 12
\tau(\eps,R)>\tau_\infty)\ge \eps,$$ the process $v(\cdot)$
satisfies
\begin{align*}
\mathbb{E}\left(\lvert v \rvert_{X_{\tau_\infty+\frac 12
\tau(\eps,R)}}\cdot \mathds{1}_{\Omega_1 } \right)\le &
\mathbb{E}\left(\lvert v \rvert_{X_{T_0+
\tau(\eps,R)}}\cdot \mathds{1}_{\Omega_1 } \right) ,\\
\le & \mathbb{E}\biggl[\left(\lvert u \rvert_{X_{T_0}}+ \lvert
\mathds{1}_{[T_0,T_0+\tau(\eps,R)] }y \rvert_{X_{T_0+
\tau(\eps,R)}} \right)\cdot
\mathds{1}_{\Omega_1 } \biggr],\\
\leq &\mathbb{E} \left(\lvert u \rvert_{X_{T_0}}
\mathds{1}_{\Omega_1 } \right)+ \mathbb{E} \left(
\mathds{1}_{\Omega_1 }\lvert \mathds{1}_{[T_0,T_0+ \tau(\eps,R)]
}y \rvert_{X_{T_0+ \tau(\eps,R)}}\right),
\\
\leq &\mathbb{E} \left(\lvert u \rvert_{X_{T_0}}
\mathds{1}_{\Omega_0 } \right)+ \mathbb{E} \left(
\mathds{1}_{\Omega_1 }\lvert y \rvert_{X_{[T_0,T_0+ \tau(\eps,R)]}}\right),
\end{align*}
where the space $X_{[a,b]}$ and the norm $\vert
\cdot\vert_{X_{[a,b]}}$ are defined similarly to the space $X_T$
and norm $\vert \cdot\vert_{X_T}$.
Since $T_0=\inf\{t\in[0,\tau_\infty): \lvert u \rvert_{X_t} \ge
n_0\}=\tau_{n_0}$, we infer that $T_0$ is finite ($T_0<\tau_\infty$
to be precise)  on the set $\Omega_0$ and hence $\lvert u
\rvert_{X_{T_0}}=n_0$ on the set $\Omega_0$. Therefore,  the first
expected value $\mathbb{E} \left(\lvert u \rvert_{X_{T_0}}
\mathds{1}_{\Omega_0 }\right)$ is finite.
Since the solution $y(\cdot)$ is such that
$\mathds{1}_{[T_0,T_0+T_1) }y \in M^2(X_{T_0+T_1})$  and  the
lifespan $T_1$ of $y(\cdot)$ satisfies $\mathbb{P}(T_1+T_0\ge
T_0+\tau(\eps,R) )\ge \eps $   we infer that
\begin{align*}
\mathbb{E} \left(\lvert \mathds{1}_{[T_0,T_0+ \tau(\eps,R)] }y
\rvert_{X_{T_0+ \tau(\eps,R)}}\right)\le &  \mathbb{E}
\left(\lvert \mathds{1}_{[T_0,T_0+ \tau(\eps,R)] }y
\rvert^2_{X_{T_0+ T_1}}\right)<  \infty.
\end{align*}
Therefore
\begin{equation*}
\mathbb{E}\left(\lvert v \rvert_{X_{\tau_\infty+\frac 12
\tau(\eps,R)}}\cdot \mathds{1}_{\Omega_1 } \right)< \infty.
\end{equation*}
 This implies in
particular that $\lvert v(t)\rvert_{X_{\tau_\infty+\frac 12
\tau(\eps,R)}}<\infty$ on $\Omega_1\subset \{\tau_\infty<\infty\}$
which contradicts Proposition \ref{prop-t_infty}.
\end{proof}

\begin{Rem}
\begin{enumerate}[(i)]
 \item
The last result is very important since a'priori we only know that
$\vert u\vert_{X_t} \to \infty$ as $t\toup \tau_\infty$ on the set
$\{ \tau_\infty<\infty \}$.

\item
The proof of the existence of a global solution could then follow
the proof in \cite[Theorem 8.12]{Brz+Millet_2012}.

\end{enumerate}

\end{Rem}
\section{Maximum Principle type Theorem}\label{MAX-PRIN-SEC}
In this section we replace in the system \eqref{ABS-v1}-\eqref{ABt-d1} the general polynomial $f(\d)$ by the bounded Ginzburg-Landau function
$\mathds{1}_{\lvert \bd\rvert \le 1}(\lvert \d\rvert^2-1)\d$.
All our previous result remains true and the analysis are even easier. In the case $f(\bd)=\mathds{1}_{\lvert \bd\rvert \le 1}(\lvert \d\rvert^2-1)\d$, we will show that if the initial value $\bd_0$ is in
the unit ball, then so are the values of the vector director $\bd$.
That is, we must show that $\lvert \bd(t)\rvert^2 \le 1$ almost
all $(\omega, t,x)\in \Omega\times [0,T]\times \MO$.
In fact we have the following theorem.
\begin{thm}Assume that  $n \leq 3$ and that a process  $(\v,\d)=(\v(t), \d(t))$, $t\in [0,T]$, is a
	solution to problem \eqref{ABS-v1}-\eqref{ABt-d1} with  initial condition $(\v_0,\d_0)$ such that
	$\lvert \d_0\rvert^2\le 1$ for almost all  $(\omega,x)\in \Omega \times \MO$. Then
	$\lvert \bd(t)\rvert^2 \le 1$ for almost all $(\omega, t,x)\in
	\Omega\times [0,T]\times \MO$.
\end{thm}

\begin{proof}
	We
	follow the idea in \cite[Lemma 2.1]{Rojas-Medar2} and \cite[Proof
	of Theorem 4, Page 513]{Matoussi}. Let
	$\varphi:\mathbb{R}\rightarrow [0,1]$ be an increasing function of class
	$\mathcal{C}^\infty$ such that
	\begin{align*}
	\varphi(s) = 0 \text{ iff } s\in (-\infty, 1],\\
	\varphi(s)=1 \text{ iff } s\in [2,+\infty).
	\end{align*}
	Let $\{\varphi_m; m\in \mathbb{N}\}$ and $\{\phi_m, m\in
	\mathbb{N}\}$ be two sequences of smooth function from
	$\mathbb{R}^n$ defined by
	\begin{align*}
	\varphi_m(\d)=&\varphi(m(\lvert \d\rvert^2-1)),\\
	\phi_m(\d)=&(\lvert \d\rvert^2-1){\varphi_m(\d)}, \d\in
	\mathbb{R}^n.
	\end{align*}
	Define a sequence of function $\{\Psi_m, m \in \mathbb{N}\}$ by
	\begin{align*}
	\Psi_m(\d)=&\lve \phi_m(\d)\rve^2,\\
	=&\int_\MO (\lvert \d\rvert^2-1)^2 [\varphi_m(\d)]^2 dx,\,\, \d\in
	\el^4(\MO),
	\end{align*}
	for any $m\in \mathbb{N}$. It is clear that $\Psi_m:
	\h^2\rightarrow \mathbb{R}$ is twice (Fr\'echet) differentiable
	and its first and second derivatives satisfy
	\begin{equation*}
	\Psi_m(\d)(h)=4\int_\MO \left((\lvert
	\d\rvert^2-1)\varphi_m(\d) \d\cdot h  \right)dx+ 2m \int_\MO
	(\lvert \d\rvert^2-1)^2 \varphi_m(\d) (\d\cdot h)dx ,
	\end{equation*}
	and
	\begin{equation*}
	\begin{split}
	\Psi^{\prime \prime}_m(\d)(k,h)=8\int_\MO \biggl[\varphi_m(\d)
	(\d\cdot k) (\d \cdot h)\biggr]dx+4 \int_\MO \left(\varphi_m(\d)
	(\lvert \d\rvert^2-1) (k\cdot h)\right)dx\\
	+16m \int_\MO \left((\lvert \d\rvert^2-1)\varphi_m(\d) (\d\cdot k) (\d \cdot
	h)\right)dx\\
	+ 4m^2 \int_\MO \left((\lvert \d\rvert^2-1)^2 \varphi_m^{\prime\prime}(\d) (\d\cdot k) (\d \cdot
	h)\right)dx\\
	+ 2m \int_\MO (\lvert \d\rvert^2 -1)^2 \varphi_m (\d)
	(k\cdot h) dx,
	\end{split}
	\end{equation*}
	for any $\d \in \h^2$ and $h, k\in \el^2(\MO)$. In particular, for
	any $k,h$ such that $k\perp \d$ and $h\perp \d$
	\begin{align*}
	\Psi_m(\d)(h)=&0,\\
	\Psi_m^{\prime\prime}(\d)(k,h)=&4\int_\MO (\lvert
	\d\rvert^2-1)\varphi_m(\d) (k\cdot h) dx+2m \int_\MO (\lvert
	\d\rvert^2-1)^2 \varphi_m(\d) (k\cdot h)dx.
	\end{align*}
	It
	follows from It\^o's formula (see \cite[Theorem I.3.3.2, Page
	147]{Pardoux}) that
	\begin{equation*}
	\begin{split}
	d[\Psi_m(\d)]=\Psi_m (\d) \left(-\rA_1
	\d-\tilde{B}(\v,\d)-\frac{1}{\eps^2}f(\d)+\frac 12 G^2(\d)\right)
	dt +\frac 12 \Psi_m^{\prime\prime}(\d) (G(\d), G(\d)) dt.
	\end{split}
	\end{equation*}
	The integral stochastic vanishes because $G(\d)\perp\d$. Owing to
	the identity $$-\lvert \d\times h\rvert^2=\d\cdot \left((\d\times
	h)\times h\right),$$ we have
	\begin{equation*}
	\frac 1 2 \Psi_m^{\prime\prime}(G(\d), G(\d))+\frac 12
	\Psi_m^{\prime}(G^2(\d))=0.
	\end{equation*}
	Hence
	\begin{equation}\label{approx}
	d[\Psi_m(\d)]=\Psi_m (\d) \left(-\rA_1
	\d-\tilde{B}(\v,\d)-\frac{1}{\eps^2}f(\d)\right) dt
	\end{equation}
	Noticing that from the definition of $\varphi_m$ and the Lebesgue
	Dominated Convergence Theorem we have for $\d\in \h^2, h\in \el^2$
	\begin{align*}
	\lim_{m\rightarrow \infty}\Psi_m(\d)=&\lve \left(\lvert
	\d\rvert^2-1\right)_{+}\rve^2,\\
	\lim_{m\rightarrow
		\infty}\Psi_m(\d)(k)=&4 \int_\MO[ \left(\lvert
	\d\rvert^2-1\right)_{+}\d\cdot h] \,dx.
	\end{align*}
	Hence, we obtain from letting $m\rightarrow \infty$ in
	\eqref{approx} that for almost all $(\omega,t)\in
	\Omega \times [0,T]$
	\begin{equation*}
	y(t)-y(0)+4\int_0^t \biggl(\int_\MO \biggl[+\rA_1 \d+(\v\cdot
	\nabla)\d+\frac{1}{\eps^2} f(\d) \biggr]\cdot \biggl[\d
	\left(\lvert \d\rvert^2-1\right)_{+}\biggr] dx \biggr)ds=0,
	\end{equation*}
	where $y(t)=\lve \left(\lvert \d(t)\rvert^2-1\right)_{+}\rve^2$.
	Let us set $\xi=\left(\lvert \d\rvert^2-1\right)_{+}$, it follows
	from \cite[Exercise 7.1.5, p 283]{Atkinson} that $\xi\in \h^1$ if
	$\d\in \h^1$. Thus, since $\frac{\partial \d}{\partial
		\mathbf{n}}=0 $ we derive from integration-by-parts that
	\begin{align*}
	-4 \int_0^t \biggl(\int_\MO \Delta \d \cdot \d \left(\lvert
	\d\rvert^2-1\right)_{+} dx \biggr)ds= \int_0^t \biggl(\int_\MO
	\left(2 \nabla(\lvert \d\rvert^2)\cdot \nabla \xi+4 \xi \lvert
	\nabla
	\d\rvert^2 \right)dx \biggr) ds,\\
	\end{align*}
	Since $\xi\ge 0$ and $\lvert \nabla \d\rvert^2\ge 0$ a.e.
	$(t,x)\in \MO\times[0,T]$ we easily derive from the above identity
	that
	\begin{align*}
	-4 \int_0^t \biggl(\int_\MO \Delta \d \cdot \d \left(\lvert
	\d\rvert^2-1\right)_{+} dx \biggr)ds
	\ge 2 \int_0^t \biggl(\int_\MO  \nabla(\lvert
	\d\rvert^2-1)\cdot \nabla \xi dx \biggr)ds.
	\end{align*}
	Thanks to \cite[Exercise 7.1.5, p 283]{Atkinson} we have
	\begin{equation*}
	\int_0^t \biggl(\int_\MO  \nabla(\lvert
	\d\rvert^2-1)\cdot \nabla \xi dx \biggr)ds= \int_0^t \int_\MO
	\lvert \nabla (\lvert \d\rvert^2-1) \rvert^2 \mathds{1}_{\{\lvert
		\d\rvert^2>1\}}dx \,\, ds,
	\end{equation*}
	which implies that
	\begin{equation*}
	-4 \int_0^t \biggl(\int_\MO \Delta \d \cdot \d \left(\lvert
	\d\rvert^2-1\right)_{+} dx \biggr)ds\ge \int_0^t \int_\MO \lvert
	\nabla (\lvert \d\rvert^2-1) \rvert^2 \mathds{1}_{\{\lvert
		\d\rvert^2>1\}}dx \,\, ds.
	\end{equation*}
	
	We also have
	\begin{align*}
	4\int_0^t \biggl(\int_\MO [(\v\cdot \nabla)\d ]\cdot [\d
	\left(\lvert \d\rvert^2-1\right)_{+}] dx \biggr)ds=&2\int_0^t
	\biggl(\int_\MO [(\v\cdot \nabla)(\lvert\d\rvert^2) ][
	\left(\lvert \d\rvert^2-1\right)_{+}] dx \biggr)ds,\\
	=&\int_0^t \biggl(\int_\MO (\v\cdot \nabla)\xi \xi dx \biggr)ds,\\
	=&0.
	\end{align*}
	Since $f(\d)=0$ for $\lvert \d \rvert^2>1$ and $\xi f(\d)=0$ for
	$\lvert \d\rvert^2 \le 1$ we have
	\begin{equation*}
	4\int_0^t \biggl(\int_\MO \xi f(\d) \cdot \d  \,dx \biggr) ds=0.
	\end{equation*}
	Therefore we see that $y(t)$ satisfies the estimate
	\begin{equation*}
	y(t)+2\int_0^t \int_{\{\lvert \d\rvert^2>1\}}\lvert \nabla (\lvert
	\d\rvert^2-1)_{+}\rvert^2 ds \le y(0),
	\end{equation*}
	for almost all $(\omega,t)\in \Omega \times [0,T]$.
	Since the second term in the left hand side of the above
	inequality is positive and $y(0)=\lve (\lvert \d_0\rvert^2-1)_{+}\rve^2$ and by
	assumption $\lvert \d_0\rvert^2\le 1$ for almost all
	$(\omega,t,x)\in \Omega \times [0,T]\times \MO$  we
	derive that
	\begin{equation*}
	y(t)=0,
	\end{equation*}
	for almost all $(\omega, t)\in \Omega \times [0,T]$,
	$T\ge 0$. Hence we have $\lvert \d\rvert^2 \le 1$ a.e.
	$(\omega, t,x)\in \Omega \times [0,T]\times \MO$,
	$T\ge 0$.
\end{proof}
\appendix

\section{Some important estimates}
In this section we recall or establish some crucial estimates needed for the proof of our mains results.

First, let $d\in \{2,3\}$ and put $a=\frac d4$.  Then
the following estimates, valid for all $\bu\in \mathbb{W}^{1,4}$,
are special cases of Gagliardo-Nirenberg's inequalities:
\begin{align}
\lve \bu\rve_{\el^4}\le \lve \bu\rve^{1-a} \lve \nabla \bu
\rve^a,\label{GAG-l4}\\
\lve \bu \rve_{\el^\infty}\le \lve \bu \rve_{\el^4}^{1-a}\lve
\nabla \bu \rve_{\el^4}^a.\label{GAG-LInf}
\end{align}
The inequality \eqref{GAG-l4} can be written in the spirit of
the continuous embedding
\begin{equation}\label{SOB-EM}
\h^1\subset \el^4.
\end{equation}
  It follows from
\eqref{GAG-LInf} and \eqref{SOB-EM} that for $\bu\in D(A_1) $
\begin{equation}
\lve \bu \rve_{\el^\infty}\le \lve \bu \rve_1 ^{1-a}\lve \nabla
\bu \rve_{2}^a.\label{GAG-LInf-2}
\end{equation}
Next we give some properties of the bilinear form $B$ and $\tilde{B}$ defined in Section \ref{sec-spaces} (see equations \eqref{DEF-B1} and \eqref{DEF-B2} on page \pageref{DEF-B1}, respectively).
\begin{lem}\label{LEM-B}
The bilinear map $B(\cdot, \cdot)$ maps continuously  $\ve\times
\h^1$ into $\ve^\ast$ and
\begin{align}
&\langle B(\bu,\bv),\w\rangle=b(\bu,\bv,\w), \text{ for any }
\bu\in \ve, \bv\in \h^1,\w \in \ve,\label{B1}\\
&\langle B(\bu,\bv),\w\rangle=-b(\bu,\w,\bv)
\text{ for any } \bu\in \ve, \bv\in \h^1,\w \in \ve,\label{B2}\\
&\langle B(\bu,\bv),\bv\rangle=0 \,\, \text{ for any } \bu\in \ve,
\bv\in \ve,\label{B3}\\
&\lve B(\bu,\bv)\rve_{\ve^\ast}\le C_0 \lve \bu\rve^{1-\frac d4}
\lve\nabla \bu \rve^{\frac d4} \lve \bv\rve^{1-\frac d4}\rve
\nabla \bv\rve^\frac d4, \text{ for all } \bu\in \ve, \bv\in
\h^1.\label{B4}
\end{align}
\end{lem}
\begin{proof}
 This lemma is well-known and we refer to \cite[Chapter II, Section 1.2]{Temam} for its proof.
\end{proof}

With an abuse of notation, we again denote by
$\tilde{B}(\cdot,\cdot)$ the restriction of
$\tilde{B}(\cdot,\cdot)$ to $\ve \times \h^2$.
\begin{lem}\label{LEM-G1}
The bilinear operator $\tilde{B}$ maps continuously $\ve\times
\h^2$ into $\el^2$ \del{such that $$\langle G^1(\bv,\bd),
\w\rangle=\langle \tilde{B}(\bv,\bd), \w\rangle \text{ for any }
\bv\in\ve, \bd\in \h^2, \w\in \el^2,$$}  and there exists $C_1>0$
such that
\begin{equation}\label{EST-G1}
\lve \tilde{B} (\bv, \bd)\lve \le C_1  \lve \bv\rve^{1-\frac
d4}\rve \nabla \bv\rve^\frac d4 \lve \nabla\bd\rve^{1-\frac d4}
\lve\Delta \bd \rve^{\frac d4}, \text{ for any } \bv \in \ve,
\bd\in \h^2.
\end{equation}
Moreover, we have
\begin{align}
\langle \tilde{B}(\bv,\bd),\bd\rangle=&0,\label{tild-b-0}\text{
for
any } \bv\in \ve, \bd \in \h^2.
\end{align}
\end{lem}
\begin{proof}
We can argue as in the proof of \eqref{B4} (see also \cite[Chapter II, Section 1.2]{Temam}) to establish the estimate \eqref{EST-G1}.
 The identity \eqref{tild-b-0}
easily follows by integration-by-parts and by taking into account
that $\nabla \cdot \v=0$ and $\v$ is zero on the boundary.
\end{proof}

Next let $\rA$ and $A_1$ the self-adjoint linear operators defined in Section \ref{sec-spaces} on page \pageref{sec-spaces}. We recall that the space
$\bx_\alpha$, $\alpha\ge 0$ are the domain of the fractional power operator $(I+\rrA)^{frac12+\alpha}$ and $\bx_\alpha \subset \h^{1+2\alpha}$ (see the identity \eqref{frac-space} on page \pageref{frac-space}). We will give some properties of the semigroups $\{\mathbb{S}_1(t): t\ge0\}$ and $\{\mathbb{S}_2(t): t\ge0\}$
 generated by  $-\rA$ on $\h$ and $-A_1$ on $\bx_\frac12$, respectively.
\begin{lem}\label{SEM-1}
Let $T\in (0,\infty)$, $g\in L^2(0,T; \bx_{0})$ and
\begin{align}
\bu(t)=\int_0^T \mathbb{S}_2(t-s) g(s) ds,= \sum_{k\in
\mathbb{N}} \int_0^T e^{-\lambda (t-s)} g_k(s) \,ds, \;\; t\in [0,T].
\label{det-conv}
\end{align} Then
$$ \bu \in C([0,T]; \bx_{\frac 12})\cap L^2(0,T; \bx_{1}),$$
and
\begin{equation*}
\lve \bu(t)\rve_{C([0,T];\bx_{\frac 12})}+\lve \bu(t)\rve_{L^2(0,T;\bx_1) }\le (1+\max(T, 1+\frac 1 \lambda_2)) \lve g(t)\rve_{L^2(0,T;
\bx_{0})}.
\end{equation*}
 \end{lem}
\begin{proof} This result is well-known. We refer the reader to \cite{Pardoux}.
\end{proof}
Similarly we have the following result, see, for instance, \cite{Vish+Furs_1980}.
 \begin{lem}\label{SEM-2}
Let $T\in (0,\infty)$, $\tilde{g}\in L^2(0,T; \h)$ and
\begin{align*}
\bv(t)=\int_0^T \mathbb{S}_1(t-s) \tilde{g}(s) ds.
\end{align*}
We have
\begin{equation*}
\lve \bv(t)\rve_{C([0,T]; D(\rA^\frac12)}+\lve \bv(t)\rve_{L^2(0,T;D(\rA))
}\le (1+\frac 1 \mu_1) \lve \tilde{g}\rve_{L^2(0,T;\h) },
\end{equation*}
where $\mu_1$ is the smallest eigenvalues of the Stokes operator
$\rA$.
 \end{lem}
Let $\zeta_1$ be an element of  $M^2(0,T;\h)$ (resp.
$\zeta_2\in M^2(0,T; \bx_{\frac 12})$) and consider the stochastic convolutions
$$ \w_2(t)=\int_0^t \mathbb{S}_2(t-s)\zeta_2(s) dW_2(s),$$  and
$$\w_1(t)=\int_0^t \mathbb{S}_1(t-s)\zeta_1(s) dW_1(s).$$
\begin{lem}\label{SEM-3}
There exists $C>0$ such that \begin{equation*} \mathbb{E} \lve
\w_i(t)\rve_{C([0,T];\tilde{V})}+\mathbb{E}\lve
\w_i(t)\rve_{L^2(0,T;\tilde{E})}\le C \mathbb{E}\lve
\zeta_i(t)\rve_{L^2(0,T;\tilde{V})},
\end{equation*}
where $(\tilde{V}, \tilde{E})=(D(\rA^\frac12),D(A_1))$ if $i=1$,
and $(\tilde{V}, \tilde{E})=(D(A_1), \bx_{1})$ if $i=2$.
\end{lem}
\begin{proof} This result is also well-known, see for example \cite{Pardoux,Vish+Furs_1980}.
\end{proof}
\section{Proof of the Lemma \ref{LEMMA_1} and \ref{LEMMA_3}}
In this section we give the proofs of the Lemma  Lemmata \ref{LEMMA_1}-\ref{LEMMA_3} and we shall start with that of Lemma \ref{LEMMA_1}.
\begin{proof}[Proof of Lemma \ref{LEMMA_1}]

	First, we will establish \eqref{ST10-b}. For this purpose, 	we set
	 \begin{align}
	\langle \Delta \d-f(\d), \Delta\d\times
	 h\rangle +\langle\Delta \d-f(\d),\d\times \Delta h+(\d\times
	 h)\times \Delta h\rangle -2\langle \Delta
	 \d-f(\d), \fp(\d) (\d\times h)\times
	 h\rangle\nonumber \\  =:  R_1+R_2+R_3. \nonumber
	 \end{align}
	 Owing to Cauchy-Schwarz, Cauchy inequalities, the assumption on
	 $h$ and Remark \ref{bigdanh2} we have that
	 \begin{align}
	 \lvert R_1\rvert \le& C \lve \Delta
	 \d-f(\d)\rve \lve \Delta \d \rve\lve h\rve_{\el^\infty} \nonumber\\
	 \le& C \lve \Delta \d-f(\d)\rve^2+C(h)\lve \Delta
	 \d\rve^2\nonumber\\
	 \le& C (\lve \Delta \d-f(\d)\rve^2+\lve
	 \d\rve^{\bar{q}}_{\el^{\bar{q}}}+1), \nonumber 
	 \end{align}
	 where $\bar{q}=4N+2$. In a similar way we can show that
	 \begin{align}
	 \lvert R_2\rvert\le & C \lve \Delta \d-f(\d)\rve^2+ \lve
	 \d\rve^2_{\el^\infty} (\lve \Delta h\rve^2[1+\lve
	 h\rve^2_{\el^\infty}] )\nonumber \\
	 \le & C(\lve \Delta \d-f(\d)\rve^2+\lve
	 \d\rve^{\bar{q}}_{\el^{\bar{q}}}+1).\nonumber 
	 \end{align}
	 Note that we have used the fact that $\h^2$ is continuously
	 embedded in $\el^\infty$. Next, thanks to \eqref{REM-H2} we have
	 \begin{align}
	 \lvert R_3\rvert\le & C \lve \Delta \d -f(\d)\rve^2 +C
	 \int_{\mathcal{O}}\lvert \fp(\d(x))[(\d\times h)\times h] \nonumber \\
	 \le & \lve \Delta \d-f(\d)\rve^2 + C\lve h\rve^2_{\el^\infty}
	 \int_{\mathcal{O}} (1+\lvert \d(x)\rvert^{2N})^2 \vert
	 \d(x)\vert^2 dx \nonumber \\
	 \le & C(\lve \Delta \d-f(\d)\rve^2+\lve
	 \d\rve^{\bar{q}}_{\el^{\bar{q}}}+1).\nonumber 
	 \end{align}
	 Collecting all these inequalities and using the embedding $\h^1 \subset \el^{4N+2}$ we obtain  that
	 \begin{equation}
	 \vert R_1 + R_2 +R_3\vert \le C(\lve \Delta \d-f(\d)\rve^2+\lve
	 \d\rve^{\bar{q}}_{\h^1}+1), \nonumber
	 \end{equation}
	 which completes the proof of \eqref{ST10-b}.

	 Now, we proceed to the proof of \eqref{ST12}.
	 To this end, set let us take a look at
	 \begin{align*}
	 \lve \Delta
	 G(\d)-\fp(\d)G(\d)\rve^2-\langle \Delta \d-f(\d), \fpp(\d) [G(\d),
	 G(\d)]\rangle=:J_1+J_2.
	 \end{align*}
	 By using the fact that $\rvert \fpp(\bu)\rvert\le C(1+ \lvert \bu\rvert^{N-2})$ for any $\bu \in \mathbb{R}^2$, it is easy to show that
	 \begin{align}
	 \lvert J_2\rvert &\le C \lve \Delta \d -f(\d)\rve^2+
	 C \int_{\mathcal{O}}\lvert \fpp(\d(x))[G(\d(x)), G(\d(x))]\rvert^2 dx \nonumber\\
	 &\le C \lve \Delta \d -f(\d)\rve^2+ C \lve h\rve^2_{\el^\infty}
	 \int_{\mathcal{O}} \lvert \fpp(\d(x))\rvert^2 \lvert\d(x)\rvert^4
	 dx\nonumber \\
	 \le & C(\lve \Delta \d-f(\d)\rve^2+\lve
	 \d\rve^{\bar{q}}_{\el^{\bar{q}}}+1).\label{ST11}
	 \end{align}
	 To deal with $J_1$ first remark that
	 \begin{align}
	 J_1= &\lve \Delta\d \times h +\d\times \Delta h
	 -\fp(\d)(\d\times h)\lve^2\label{J_1}\\
	 \le & \quad C(  \lve h\rve^2 \lve \Delta \d\rve^2+ \lve
	 \d\rve^2_{\el^\infty} \lve \Delta h\rve^2+ \lve \fpp(\d)(\d\times
	 h)\rve^2.\nonumber
	 \end{align}
	 Now, arguing as in the proof of  \eqref{ST10-b} we  derive that
	 \begin{equation}\label{ST11-b}
	 J_1\le C(\lve \Delta \d-f(\d)\rve^2+\lve
	 \d\rve^{\bar{q}}_{\el^{\bar{q}}}+1).
	 \end{equation}
	 Owing to \eqref{ST11}, \eqref{ST11-b} and the Sobolev embedding $\h^1\subset \el^{4N+2}$ we obtain
	 \begin{equation*}
	 \begin{split}
	 \vert J_1 + J_2\vert \le C(\lve \Delta
	 \d-f(\d)\rve^2+\lve \d\rve^{\bar{q}}_{\h^1}+1),
	 \end{split}
	 \end{equation*}
	 which completes the proof of \eqref{ST12} and Lemma \ref{LEMMA_1}.
\end{proof}
Now we shall establish Lemma \ref{LEMMA_3}.
\begin{proof}[Proof of Lemma \ref{LEMMA_3}]
	For any $\bu\in \h^1$, we have
	\begin{align*}
	\lvert \langle \fp(\d)\bu,\bu\rangle=&
	\int_{\MO}\tilde{f}(\d(x))\lvert \bu(x)\rvert dx+2
	\int_\MO\tilde{f}^\prime (\d(x))\lvert
	\d(x)\cdot \bu(x)\rvert^2 dx\nonumber \\
	\le & \quad  \biggl[\int_\MO \lvert \tilde{f}(\d(x))\rvert^2
	dx\biggr]^\frac12\biggl[ \int_\MO \rvert\bu(x)\rvert^4
	dx\biggr]^\frac12 + 2 \biggl[\int_\MO \lvert
	\tilde{f}^\prime(\d(x))\rvert^2 \lvert \d(x)\rvert^4
	dx\biggr]^\frac12\biggl[ \int_\MO \rvert\bu(x)\rvert^4
	dx\biggr]^\frac12\nonumber\\
	\end{align*}
	By Gagliardo-Nirenberg, Cauchy inequalities, \eqref{ST6-B-0} and \eqref{ST6-B-1}
	we derive that
	\begin{equation*}
	\lvert \langle \fp(\d)\bu,\bu\rangle\le \frac14 \lve \nabla
	\bu\rve^2+ (\ell_1 + \ell_2\lve \d
	\rve^{2N}_{\el^{4N}})^2 \lve \bu \rve^2.
	\end{equation*}
	Putting $\bu=\Delta \d-f(\d)$ in this last inequality and using
	the Sobolev embedding $\h^1\subset \el^{4N}$ we infer that
	\begin{equation*}
	\lvert \langle \fp(\d)[\Delta \d-f(\d)],\Delta \d-f(\d) \rangle\le
	\delta_7 \lve \nabla [\Delta \d -f(\d)]\rve^2+ C(\delta_7)(\ell_1
	+ \ell_2\lve \d \rve^{2N}_{\h^1})^2 \lve \Delta \d -f(\d)\rve^2,
	\end{equation*}
	from which we easily conclude the proof of \eqref{ST6-B} and of Lemma \ref{LEMMA_3}.
\end{proof}
\section{Basic estimates for the solution  $(\bv,\bd)$}\label{AppB}
 Throughout this section we will assume that the initial data is random and have finite moments in appropriate functional spaces. These assumptions will be made precise in each of the following four crucial propositions. Obviously, the estimates we derive in this section are still valid for non-random initial data. Throughout, $(\bv,\bd)$ will denote a weak or a strong solution to \eqref{eqn-SLQE-d}.
\subsection{Estimates for $\lVert \bdt \rVert$}
Before we proceed further we should note that the forthcoming results are proved for $p\ge 4N+2$, but they will also hold for $p\in [1, 4N+2)$.
\begin{prop}\label{EST0}
Assume that for some $p\ge 4N+2$ and $t\ge0$ $$ \mathfrak{G}_0(t, p)<\infty \text{ almost surely}, $$ where $\mathfrak{G}_0(t,p)$ is defined in \eqref{G_0}. Then,    the following estimate holds almost surely
 \begin{equation}\label{EST-N}
  \lve \bd(t)\rve^p+p \int_0^t \lve \bds \rve^{p-2} \lve \nabla \bds\rve^2 ds
  + p \int_0^t \lve \bds \rve^{p-2} \lve \bds\rve^{2N+2}_{\mathbb{L}^{2N+2}} ds\le \mathfrak{G}_0(t, p).
 \end{equation}
\end{prop}
\begin{proof}
Let $\Psi(\cdot)$ be the mapping defined by $\Psi(\d)=\frac 12 \lVert \d \rVert^{p}$ for any $\d\in \el^2$. This mapping is twice Fr\'echet differentiable with first and second derivatives defined by
\begin{align*}
 \Psi^\prime(\d)[\mathbf{h}]=p \Vert \d \Vert^{p-2} \langle \d, \mathbf{h}\rangle,\\
 \Psi^{\prime\prime}[\mathbf{h},\mathbf{k}]=p(p-2)\Vert \d\Vert^{p-4}\langle \d,\mathbf{k}\rangle \langle \d, \mathbf{h}\rangle+p \Vert \d\Vert^{p-2} \langle \mathbf{h}, \mathbf{k}\rangle.
\end{align*}
 by straightforward calculation one can check that
 if $\mathbf{h}\in \el^2$ and $\mathbf{h}\perp \d$ then $\Psi^\prime(\d)[\mathbf{h}]=0$ and $\Psi''(\d)[\mathbf{h},\mathbf{h}]=p \lve \d \rve^{p-2}\lve \mathbf{h}\rve^2.$ Moreover, if $\d\in \h^1$
 then it is possible to extend $\Psi^\prime(\d)[\cdot]$ to the dual space $(\h^1)^\ast$ of $\h^1$ by setting
 $$ \Psi(\d)[\mathbf{h}]=p \lVert \d\rVert^{p-2} {}_{(\h^1)^\ast}\langle \mathbf{h}, \d\rangle_{\h^1} \text{ for any } \mathbf{h}\in (\h^1)^\ast.$$
  Note also that if $\d\in \h^1$ then it follows from the Sobolev continuous embedding
 $\h^1\subset \el^{4n+2}$ and Remark \ref{REM-H2} that $\Psi(\d)[f(\d)]$ is well defined for any $\d\in \h^1$. From these observation we can apply It\^o's formula and derive that
\begin{equation*}
\begin{split}
\Psi(\d(t))= \Psi(\d_{0})-\int_0^t \Psi'(\d(s))[A\d(s)+\tilde{B}(\v(s),\d(s)+ f(\d(s))]\\
+\frac12 \int_0^t \Big(\Psi'(\d(s))[G^2(\d(s))]
+ \Psi''(\d(s))[G(\d(s),G(\d(s)]\Big) ds\\+ \int_0^t \Psi'(\d(s))[G(\d(s))] dW_2(s)
\end{split}
\end{equation*}
The stochastic integral vanishes because $\d\times h\perp \d$
and
\begin{align*}
\langle G(\d),\d \rangle= &\langle (\d\times h),
\d\rangle,\\
=& \langle \d\times h, \d\rangle,\\
=&0.
\end{align*}
 Since $\v$ is a divergence free function it follows from  \eqref{tild-b-0} that
$$\langle \tilde{B}(\v(t),\d(t)),\d(t)\rangle=0. $$
From the identity $$\langle (b\times a)\times a, b\rangle=-\lVert
a\times b\rVert^2, $$ we deduce that
\begin{align*}
& \Psi'(\d(s))[G^2(\d(s))]
+ \Psi''(\d(s))[G(\d(s)),G(\d(s))]\\
& \quad \quad \quad =2p\lve \d(s)\rve^{2(p-1)}\Big[\langle
G^2(\d(t)), \d(t)\rangle+\lVert G(\d(t))\rVert^2\Big]\\
& \quad \quad \quad = 0.
\end{align*}
Consequently,
\begin{equation}\label{3.4}
\begin{split}
\lVert \d(t)\rVert^{p}=\lVert \d_{0}\rVert^{p}-p \int_0^t \lve \d(s)\rve^{p-2} \lVert \nabla \d(s) \rVert^2
ds\\-p\int_0^t \lve \d(s) \rve^{p-2} \langle f(\d(s)),
\d(s)\rangle ds,
\end{split}
\end{equation}
Now, by Assumption \ref{eqn-f} that there exists a polynomial $\tilde{F}(r)=\sum_{l=1}^{N+1} b_l r^l$ with $\tilde{F}(0)=0$ and $b_{N+1}>0$ such that $$ \langle f(\d), \d\rangle = \int_{\mathcal{O}} \tilde{F}(\vert \d(x)\vert^2) dx.$$
In fact, it follows from Assumption \ref{eqn-f} that
\begin{align*}
 \langle \tilde{f}(\vert \d\vert^2)\d, \d\rangle=& \int_{\mathcal{O}} \tilde{f}(\vert \d(x)\vert^2)\vert \d(x)\vert^2 dx\\
 =&  \int_{\mathcal{O}} \sum_{k=0}^{N} a_k (\vert \d(x)\vert^2)^{k+1} dx\\
 =& \int_{\mathcal{O}} \sum_{l=1}^{N+1} a_{l-1} (\vert \d(x)\vert^2)^{l} dx.
\end{align*}
Thanks to this observation we can use \cite[Lemma 8.7]{Brz+Millet_2012} to infer that there exists $c>0$ such that
\begin{equation*}
 \frac{a_{N+1}}{2} \int_{\mathcal{O}} \lvert \d(x)\rvert^{2N+2} dx-c \int_{\mathcal{O}} \lvert \d(x)\rvert^2 dx \le \langle f(\d), \d\rangle.
\end{equation*}
From this estimate and \eqref{3.4} we deduce that
\begin{equation}\label{3.5}
\begin{split}
\lVert \d(t)\rVert^{p}+p \int_0^t \lve \d(s)\rve^{p-2} \lVert \nabla \d(s) \rVert^2
ds+ p \int_0^t \lve \bds \rve^{p-2} \lve \bds\rve^{2N+2}_{\mathbb{L}^{2N+2}} ds\le C \int_0^t \Vert \d(s) \Vert^p ds + \Vert \d_0\Vert^p,
\end{split}
\end{equation}
from which along with an application of the Gronwall lemma we complete the proof of our proposition.
\end{proof}
\subsection{Estimates for $\lVert \bv \rVert$ and $\lVert \nabla
\d \rVert$}
 \begin{prop}\label{EST1}
 Suppose that Assumption \ref{HYPO-ST-weak} is verified and that
 for $p=1$ and some $p\ge 4N+2$ and for any $t\ge0$
 $$ \mathbb{E} \mathfrak{G}_1(t,p) <\infty,$$
 where $\mathfrak{G}_1(t,p)$ is defined in \eqref{G_1}.
Let $\{\tau_R: R\ge 0\}$ be an arbitrary family of stopping times such that the stochastic process $M_{\cdot \wedge \tau_R}$
defined by
\begin{equation*}
 M_{t\wedge \tau_R}:=\int_0^{t\wedge \tau_R} \langle G(\d(s)),
f(\d(s))-\Delta \d(s)\rangle dW_2(t)
\end{equation*}
is a martingale.
Then, there exists $\ell>0$ such that for any  $t\ge 0$
\begin{equation}\label{EST-VD-A}
\begin{split}
\mathbb{E} \biggl[\sup_{0\le s\le
t\wedge \tau_R}\left(\lve\bv(s)\rve^{2}+\ell \lve \d(s) \rve^2+\lve \nabla \bd(s)\rve^{2}+\int_\MO F(\d(s,x) dx \right)^p\\+\mathbb{E} \left(\int_0^{t\wedge \tau_R}
\left(\lve \nabla
\bv(s)\rve^2+ \lve \Delta \bd(s)-f(\bd(s))\rve^2\right)ds\right)^p\biggr]
\le \mathbb{E}\mathfrak{G}_1(t,p).
\end{split}
\end{equation}

\end{prop}
\begin{proof}
In what follow we fix a time $t\in [0,T\wedge \tau^R)$. Application of It\^o's formula to
$\Phi(\bvt)=\frac 12 \lVert \bvt\rVert^2$ yields that for any $t \in [0,T\wedge \tau^R) $ and $\mathbb{P}-$a.s.
\begin{equation}\label{est-vm}
\begin{split}
d\lVert \bvt\rVert^2=-\biggl\langle\ma \bvt+B(\bvt)+M(\bdt),
\bvt\biggr\rangle dt\\ +\frac 12 \lVert S(\bvt)\rVert^2_{\mathcal{T}_2} dt
+\langle S(\bvt), \bvt) dW_1(t)\rangle.
\end{split}
\end{equation}
Now we introduce the function $\Psi$ defined by
$$\Psi(\bd)=\frac{1}{2}\lVert \nabla \bd\rVert^2+\frac12 \int_\MO F(\vert \d(x) \vert^2)\,\, dx, \bd\in \h^1.$$
Arguing as in \cite{Brz+Millet_2012} we can show that the map $\Psi(\cdot)$ is twice Fr\'echet differentiables and the first and second derivatives of $\Psi$ are given by
\begin{align*}
\Psi^\prime(\bd)\mathbf{g}=&\langle \nabla \bd, \nabla
\mathbf{g}\rangle+\langle f(\d), \mathbf{g}\rangle,\\
\Psi^{\prime\prime}(\bd)(\mathbf{g}, \mathbf{g})=&\langle \nabla
\mathbf{g}, \nabla \mathbf{g}\rangle+\int_\MO \tilde{f}(\d) \lvert \mathbf{g}\rvert^2 dx +\int_\MO [\tilde{f}^\prime(\d)
][\d \cdot \mathbf{g}]^2\,\, dx,
\end{align*}
$\text{ for all } \bd, \mathbf{g} \in \h^1.$
If $\mathbf{g}\perp \d$ then
$$ \Psi^{\prime\prime}(\bd)(\mathbf{g}, \mathbf{g})=\langle \nabla
\mathbf{g}, \nabla \mathbf{g}\rangle+\int_\MO \tilde{f}(\d) \lvert \mathbf{g}\rvert^2 dx. $$
Note also that $$ \Psi^\prime(\bd)\mathbf{g}= \langle -\Delta \bd, \mathbf{g}\rangle + \langle f(\d), \mathbf{g}\rangle,\\$$for all
$\d \in \h^2$ and $\mathbf{g}\in \h^1$.

Before proceeding further we should recall that it was proved in \cite[Lemma 8.9]{Brz+Millet_2012} that there exists $\ell >0$ such that
\begin{equation}\label{Lem-BM}
 \lve \nabla \d\rve^2+\lve \d \rve^2 \le 2 \Psi(\d)+\ell \lve \d\rve^2,
\end{equation}
for any $\d\in \h^1$.

Now, by
It\^o's formula we have
\begin{equation*}
\begin{split}
d\Psi(\bdt)=\left(-\lve \Delta \bdt -f(\bdt)\rve^2 +\frac12 \int_\MO \tilde{f}(\bdt)\lvert G(\bdt)\rvert^2\right) dt \\
+\biggl\langle \biggl(\frac 12 G^2(\bdt)-\tilde{B}(\bvt,\bdt)\biggr),f(\bdt)-\Delta \bdt\biggr\rangle dt\\
+\frac 12 \lVert \nabla G(\bdt)\rVert^2+ \langle G(\bdt),
f(\bdt)-\Delta \bdt\rangle dW_2(t),
\end{split}
\end{equation*}
 which is equivalent to
\begin{equation*}
\begin{split}
d\Psi(\bdt)=\left(-\lve \Delta \bdt -f(\bdt)\rve^2 +\int_\MO [\tilde{f}(\bdt)\lvert G(\bdt)\rvert^2\,\, dx \right) dt \\
+\left(\langle \frac 12 G^2(\bdt), f(\bdt)-\Delta \bdt\rangle+\langle \tilde{B}(\bvt,\bdt),\Delta \bdt\rangle \right)dt\\
+\frac 12 \lVert \nabla G(\bdt)\rVert^2+ \langle G(\bdt),
f(\bdt)-\Delta \bdt\rangle dW_2(t).
\end{split}
\end{equation*}
Note that we have used the fact that
\begin{align*}
 \langle \bv \cdot \nabla \d, f(\d)\rangle = &\sum_{i,j} \int_\MO \bv_i(x) \frac{\partial \d_j(x)}{\partial x_i}
 \tilde{f}(\lvert \d(x)\rvert^2) \d_j(x) dx\\
 =& \frac12 \int_\MO \bv_i(x)  \frac{\partial \tilde{F}(\lvert \d(x)\rvert^2)}{\partial x_i} dx\\
 =& \langle \bv, \nabla \tilde{F }(\vert \d\rvert^2)\rangle\\
 =& 0, \text{ because $\divv\bv=0$.}
\end{align*}
Notice also that
\begin{align*}
 \frac 12 \biggl\vert \int_\MO \tilde{f}(\vert \d(x)\vert^2) \lvert G(\d(x))\rvert^2 dx\biggr\vert \le \frac 12 \lve h\rve_{\mathbb{L}^\infty}^2
 \int_\MO \vert \tilde{f}(\vert \d(x)\vert^2)\vert \,\,\vert \d(x) \vert^2 dx.
\end{align*}
By setting $\bar{f}(r)=\sum_{k=0}^N b_k r^k $ with $b_k=\vert a_k\vert$, $k=0,\ldots, N$, we derive that there exists a polynomial $\tilde{Q}$ of degree $N$ such that $$ \bar{f}(r)r= a_N r^{N+1} + \tilde{Q}(r)
.$$ From this last identity and the former estimate we derive that
\begin{align*}
 \frac 12 \biggl\vert \int_\MO \tilde{f}(\vert \d\vert^2) \lvert G(\d)\rvert^2 dx\biggr\vert \le \frac 12 \lve h\rve_{\mathbb{L}^\infty}^2 \biggl[
 a_N\int_\MO   \vert \d(x) \vert^{2N+2} dx+ \biggl\lvert \int_{\mathcal{O}} \tilde{Q}(\lvert \d(x)\rvert^2) dx \biggr\rvert\biggr],
\end{align*}
from which along with \cite[Lemma 8.7]{Brz+Millet_2012} we deduce that
\begin{equation*}
 \frac 12 \biggl\vert \int_\MO \tilde{f}(\vert \d\vert^2) \lvert G(\d)\rvert^2 dx\biggr\vert \le C(h)\Big(
 \int_\MO F(\vert \d(x) \vert^2) dx+ \lve \d \rve^2 \Big).
\end{equation*}
Thus,
\begin{align*}
 \frac 12 \biggl\vert \int_\MO \tilde{f}(\vert \d\vert^2) \lvert G(\d)\rvert^2 dx\biggr\vert \le C(h)\Big( \Psi(\d) + \lve \d \rve^2\Big),
\end{align*}
and
\begin{equation}\label{est-grdm}
\begin{split}
d\Psi(\bdt)\le \left(-\lve \Delta \bdt -f(\bdt)\rve^2 +C(h) \Big[ \Psi(\d(t)) + \lve \d(t) \rve^2\Big] \right) dt \\
+\left(\langle \frac 12 G^2(\bdt), f(\bdt)-\Delta \bdt\rangle+\langle \tilde{B}(\bvt,\bdt),\Delta \bdt\rangle \right)dt\\
+\frac 12 \lVert \nabla G(\bdt)\rVert^2+ \langle G(\bdt),
f(\bdt)-\Delta \bdt\rangle dW_2(t).
\end{split}
\end{equation}
Thanks to \eqref{G1-eq-Md} we derive that
\begin{equation*}
\langle \tilde{B}(\v,\d), \Delta \d\rangle-\langle M(\d),
\v \rangle=0.
\end{equation*}
From Lemma \ref{LEM-B} we also derive that
\begin{equation*}
\langle B (\v(t)),\v(t)\rangle=0.
\end{equation*}
 Therefore by summing \eqref{est-vm} and \eqref{est-grdm} side by side and
using the two last identities we see that
\begin{equation}\label{3.11}
 \begin{split}
  d[\Psi(\d(t))+\frac12 \Vert \v(t)\Vert^2]+ \biggl(\Vert \nabla \v(t)\Vert^2 + \frac12 \Vert \Delta\d(t) -f(\d(t))\Vert^2-C(h) \Big[ \Psi(\d(t)) + \lve \d(t) \rve^2\Big]\biggr)dt
  \\
  \le \frac12\biggl(\Vert S(\v(t))\Vert^2_{\mathcal{T}_2} + \Vert G^2(\bdt)\Vert^2
+\lVert \nabla G(\d(t))\Vert^2 \biggr)dt\\+ \langle\v(t), S(\v(t))dW_1(t)\rangle +\langle G(\bdt),
f(\bdt)-\Delta \bdt\rangle dW_2(t).
 \end{split}
\end{equation}
Now, recall that $G(\bd)=\bd\times h$. Hence
\begin{align}
\lVert \nabla G(\d(t))\rVert^2\le& \lVert \nabla (\d(t)\times h)\rVert^2=\lVert G(\d(t))\rve^2_{\h^1},\nonumber\\
\le& \lve \nabla (\d(t)\times h)\rve^2=\lve \d(t)\times h\rve^2_{\h^1}\nonumber\\
\le& \lve \nabla(\d(t)\times h)\rve^2+\lve \d(t)\times h\rve^2\nonumber\\
\le& 2[\lve \nabla \d(t) \times h\rve^2+ \lve\d(t)\times \nabla
h\rve^2]+\lve\d(t)\times h\rve^2\nonumber\\
\le & C \lve h\rve^2_{\el^\infty}(\lve \nabla \d(t)\rve^2 + \lve \d(t)\rve^2)
+\lve\d(t)\times \nabla
h\rve^2.\label{est-grG-1}
\end{align}
From H\"older's inequality and the Sobolev embedding $\h^1\subset
\el^6$ (true for $n=2,3$!) we obtain
\begin{align}
\lve\d(t)\times \nabla h\rve^2\le& \lve \d(t)\rve^2_{\el^6}\lve \nabla h\rve^2_{\el^3},\nonumber\\
\le& c(\lve \nabla \d(t)\rve^2+\lve \d(t) \rve^2)\lve \nabla
h\rve^2_{\el^3}.\nonumber
\end{align}
By plugging this last inequality into \eqref{est-grG-1} we infer the existence of a constant $C(h)>0$ such that
\begin{equation*}
\lve \nabla G(\d(t))\rve^2 \le C(h) (\lve \nabla \d(t)\rve^2 +\lve\d(t)
\rve^2).
\end{equation*}
From the definition of
$G^2$ it is easy to see that there exists $C(h)>0$ such that
\begin{align*}
\lve G^2(\d(t))\rve^2\le& C(h) \lve \d(t)\rve^2\le C(h) (\Vert \nabla \d(t)\Vert^2+\Vert \d(t)\Vert^2).
\end{align*}
From the last two estimates, \eqref{Lem-BM} and the linear growth assumption \eqref{HYPO-ST-weak} we derive that there exist two constants $C>0$ and $C(h)>0$ such that for any $t\ge 0$
\begin{equation}\label{3.14}
 \frac12\biggl(\Vert S(\v(t))\Vert^2_{\mathcal{T}_2} + \Vert G^2(\bdt)\Vert^2
+\lVert \nabla G(\d(t))\Vert^2 \biggr)\le C\Vert \v(t)\Vert^2 +2 C(h)\Psi(\d(t))+\ell C(h) \Vert \d(t)\Vert^2.
\end{equation}
From this inequality and \eqref{3.11} we derive that there exists $C>0$ such that
\begin{equation}\label{3.15}
\begin{split}
  \me \sup_{s\in [0,t\wedge \tau_R]}[\Psi(\d(s))+\frac12 \Vert \v(s)\Vert^2]+ \me \int_0^{t\wedge \tau_R}\biggl(\Vert \nabla \v(s)\Vert^2 + \frac12 \Vert \Delta\d(s) -f(\d(s))\Vert^2\biggr)ds\\
  \le
  C \me\int_0^{t\wedge \tau_R} \Big[ \lve \v(s) \rve^2 +\Psi(\d(s))\Big]ds + \me \int_0^{t\wedge \tau_R} \lve \d(s) \rve^2 ds
  +\me \sup_{s\in [0,t\wedge \tau_R] }\biggl\vert \int_0^{s\wedge \tau_R} \langle\v(r), S(\v(r))dW_1(r) \rangle\biggr\vert\\ +\me \sup_{s\in [0,t\wedge \tau_R]}\biggl\vert \int_0^{s\wedge \tau_R} \langle G(\d(r)),
f(\d(r))-\Delta \d(r) \rangle dW_2(r)\biggr\vert + \mathbb{E}\left(\frac12 \Vert \bv_0\Vert^2+ \Psi(\d_0)\right).
 \end{split}
\end{equation}
Thanks to the Burkholder-Davis-Gundy, Cauchy-Schwarz and Cauchy inequalities we infer that
\begin{align}
 \me \sup_{s\in [0,t\wedge \tau_R]}\biggl\vert \int_0^{s\wedge \tau_R} \langle G(\d(r)),
f(\d(r))-\Delta \d(r) \rangle dW_2(r)\biggr\vert\le C \me \biggl(\int_0^{t\wedge \tau_R} [\langle G(\d(s)),\Delta \d(s)-f(\d(s))\rangle]^2 ds\biggr)^\frac12\nonumber\\
\le C \me \biggl[\sup_{s\in [0,t\wedge \tau_R]}\lve G(\d(s))\rve\biggl(\int_0^{t\wedge \tau_R} \lve \Delta \d(s)-f(\d(s))\rve^2 ds\biggr)^\frac12 \biggr]\nonumber\\
\le C \me  \sup_{s\in [0,t\wedge \tau_R]}\lve G(\d(s))\rve^2 +\frac14 \me \int_0^{t\wedge \tau_R} \lve \Delta \d(s)-f(\d(s))\rve^2 ds\nonumber\\
\le C \lve h \rve^2_{\el^\infty}\me \sup_{s\in [0,t\wedge \tau_R]}\lve \d(s)\rve^2 +\frac14 \me \int_0^{t\wedge \tau_R} \lve \Delta \d(s)-f(\d(s))\rve^2 ds.\label{BDG-n}
\end{align}
By making use of a similar argument and the linear growth assumption \eqref{HYPO-ST-weak} one can prove that
\begin{equation}\label{BDG-v}
 \me \sup_{s\in [0,t\wedge \tau_R] }\biggl\vert \int_0^{s\wedge \tau_R} \langle\v(r), S(\v(r))dW_1(r) \rangle\biggr\vert\le \frac14\me \sup_{s\in [0,t\wedge \tau_R]}\lve \v(s)\rve^2 +C \me \int_0^{t\wedge \tau_R} \lve \v(s) \rve^2 ds.
\end{equation}
Note that from \eqref{Lem-BM} we easily derive that $\lve \bv \rve^2 +\lve \d \rve^2 \le\lve \v\rve^2+ 2\Psi(\d) +\ell \lve \d \rve^2.$ Hence, using \eqref{BDG-n} and \eqref{BDG-v} in  \eqref{3.15}  we infer that
\begin{equation*}
\begin{split}
 \me \sup_{s\in [0,t\wedge \tau_R]}[\Psi(\d(s))+\frac12 \Vert \v(s)\Vert^2]+ \me \int_0^{t\wedge \tau_R}\biggl(\Vert \nabla \v(s)\Vert^2 + \frac12 \Vert \Delta\d(s) -f(\d(s))\Vert^2\biggr)ds\\
  \le C \me \int_0^{t\wedge \tau_R} \left(\lve \v(s)\rve^2 +\Psi(\d(s))\right)ds + C \varphi(t\wedge \tau_R),
\end{split}
\end{equation*}
where $\varphi(\cdot)$ is the non-decreasing function defined by $$ \varphi(t)= \me \left(\lve \v_0\rve^2+\Psi(\d_0) \right)+\me \sup_{s\in [0,t]} \lve \d(s)\rve^2+\me \int_0^{t} \lve \d(s)\rve^2 ds.$$
Now it follows from Gronwall's lemma that
\begin{equation*}
\begin{split}
  \me \sup_{s\in [0,t\wedge \tau_R]}[\Psi(\d(s))+\frac12 \Vert \v(s)\Vert^2]+ \me \int_0^{t\wedge \tau_R}\biggl(\Vert \nabla \v(s)\Vert^2 + \frac12 \Vert \Delta\d(s) -f(\d(s))\Vert^2\biggr)ds\\
  \le \varphi(t\wedge \tau_R)
  \left(1+t\wedge \tau_R e^{C t\wedge \tau_R}\right).
\end{split}
\end{equation*}
This completes the proof of our proposition for the case $ p=1$.
For the case $p\ge 4N+2$ we first observe that from \eqref{3.11}, \eqref{3.14} we easily see that
\begin{equation*}
 \begin{split}
  \Psi(\d(t))+\ell \lve \d(t)\rve^2 +\lve \v(t)\rve^2 +\int_0^{t}\biggl(2 \lve \nabla \v(s) \rve^2+ \lve \Delta \d(s)-f(\d(s))\rve^2 \biggr)ds\\
  \le \Psi(\d_0)+\ell \lve \d_0\rve^2 +\lve \v_0\rve^2 + C \int_0^{t} \left(2\Psi(\d(s))+\ell \lve \d(s)\rve^2 +\lve \v(s)\rve^2\right) ds \\
  \biggl\lvert \int_0^{t} \langle\v(s), S(\v(s))dW_1(s)\rangle\biggr\rvert + \biggl\rvert \int_0^{t} \langle G(\d(s)),
f(\d(s))-\Delta \d(s)\rangle dW_2(s)\biggr\rvert.
 \end{split}
\end{equation*}
Rising both sides of this estimate to the power $p$ and taking the supremum over $s\in [0,t\wedge \tau_R]$ and the mathematical expectation imply that
\begin{equation}\label{3.21}
\begin{split}
 \me \sup_{s\in [0,t]}[\psi(s)]^p- \me [\psi(0)]^p + \me \biggl[\int_0^{t\wedge \tau_R} \biggl(2 \lve \nabla \v(s) \rve^2+ \lve \Delta \d(s)-f(\d(s))\rve^2 \biggr)ds\biggr]^p\\
 \le C t\wedge \tau_R\me \int_0^{t\wedge \tau_R} [\psi(s)]^p ds + \me \sup_{s\in [0,t\wedge \tau_R]}\biggl\lvert \int_0^{s\wedge \tau_R} \langle\v(r), S(\v(r))dW_1\rangle\biggr\rvert^p\\ +
 \me \sup_{s\in [0,t\wedge \tau_R] }\biggl\rvert \int_0^{s\wedge \tau_R} \langle G(\d(r)),
f(\d(r))-\Delta \d(r)\rangle dW_2\biggr\rvert^p,
\end{split}
\end{equation}
where, for the sake of simplicity, we have put $\psi(t)=\Psi(\d(t))+\ell \lve \d(t)\rve^2 +\lve \v(t)\rve^2$. \\
Now by making use of the Burkholder-Davis-Gundy, Cauchy-Schwarz, Cauchy inequalities and the linear growth assumption \eqref{HYPO-ST-weak} we derive that
\begin{align}
 \me \sup_{s\in [0,t\wedge \tau_R]}\biggl\lvert \int_0^{s\wedge \tau_R} \langle\v(r), S(\v(r))dW_1\rangle\biggr\rvert^p\le & C \me \biggl(\int_0^{t\wedge \tau_R} \lve \v(s)\rve^2 \lve S(\v(s))\rve^2_{\mathcal{T}_2} ds\biggr)^\frac p2\nonumber\\
 \le & C \me \biggl[\sup_{s\in [0,t\wedge \tau_R]} [\psi(s)]^\frac p2\biggl(\int_0^{t\wedge \tau_R} (1+\lve \v(s)\rve^2)ds\biggr)^\frac p2 \biggr]\nonumber\\
 \le & \frac14 \me \sup_{s\in [0,t\wedge \tau_R]} [\psi(s)]^p+C t\wedge \tau_R +C \me \int_0^{t\wedge \tau_R} [\psi(s)]^p ds.\label{BDG-psi-1}
\end{align}
By using a similar argument we obtain
\begin{align}
  \me \sup_{s\in [0,t\wedge \tau_R] }\biggl\rvert \int_0^{s\wedge \tau_R} \langle G(\d(r)),
f(\d(r))-\Delta \d(r)\rangle dW_2\biggr\rvert^p\le& C \me \biggl(\int_0^{t\wedge \tau_R} \lve G(\d(s))\rve^2 \lve \Delta \d(s)-f(\d(s))\rve^2 ds\biggr)^\frac p2\nonumber\\
\le & C \me \sup_{s\in [0,t\wedge \tau_R]}\lve \d(s)\times h\rve^p\nonumber\\
& \quad \quad \quad + \frac 12 \me \biggl(\int_0^{t\wedge \tau_R} \lve \Delta \d(s)-f(\d(s))\rve^2 ds \biggr)^p\nonumber.
\end{align}
From the last line we easily derive that
\begin{equation}\label{BDG-psi-2}
\begin{split}
 \me \sup_{s\in [0,t\wedge \tau_R] }\biggl\rvert \int_0^{s\wedge \tau_R} \langle G(\d(r)),
f(\d(r))-\Delta \d(r)\rangle dW_2\biggr\rvert^p-C(h) \me \sup_{s\in [0,t\wedge \tau_R]} \lve \d(s)\rve^{2p}\\
\le \frac 12 \me \biggl(\int_0^{t\wedge \tau_R} \left(2 \lve \nabla \v(s)\rve^2+\lve \Delta \d(s)-f(\d(s))\rve^2\right) ds \biggr)^p.
\end{split}
\end{equation}
Inserting \eqref{BDG-psi-1} and \eqref{BDG-psi-2} in  \eqref{3.21} yields that
\begin{equation*}
 \begin{split}
   \me \sup_{s\in [0,t\wedge \tau_R]}[\psi(s)]^p+ \me \biggl[\int_0^{t\wedge \tau_R} 2 \lve \nabla \v(s) \rve^2+ \lve \Delta \d(s)-f(\d(s))\rve^2 \biggr)ds\biggr]\\
   \le 2 \me [\psi(0)]^p+ C t\wedge \tau_R +C(h) \me \sup_{s\in [0,t\wedge \tau_R]} \lve \d(s)\rve^{2p}+ C(t\wedge \tau_R+1)\me \int_0^{t\wedge \tau_R} [\psi(s)]^p ds,
 \end{split}
\end{equation*}
from which altogether with the Gronwall lemma implies that
\begin{equation*}
 \begin{split}
   \me \sup_{s\in [0,t]}[\psi(s)]^p+ \me \biggl[\int_0^{t\wedge \tau_R} 2 \lve \nabla \v(s) \rve^2+ \lve \Delta \d(s)-f(\d(s))\rve^2 \biggr)ds\biggr]\\
   \le \biggl[2 \me [\psi(0)]^p+ C t\wedge \tau_R+C(h) \me \sup_{s\in [0,t\wedge \tau_R]} \lve \d(s)\rve^{2p} \biggr]\left(1+ Ct\wedge \tau_R(t\wedge \tau_R+1) e^{C(t\wedge \tau_R+1)t\wedge \tau_R }\right).
 \end{split}
\end{equation*}
This completes the proof of the proposition.
\end{proof}
\subsection{Estimate for $\me \int_0^T \lve \Delta \bds\rve^2 ds$ }\label{AppC}
 \begin{prop}
 Suppose that Assumption  \ref{eqn-f} is satisfied and let $p\ge 4N+2$. If $\mathbb{E}\mathfrak{G}_1(T,q(2N+1))$ for some $q\in [2,\frac{p}{4N+2}] $ and $T\ge 0$, then there exists a constant $C>0$ such that
 \begin{equation*}
  \me \biggl[\int_0^T \lve \Delta \bds\rve^2 ds\biggr]^q \le \mathbb{E}\mathfrak{G}(T, q(2N+1))+C T.
 \end{equation*}
 \end{prop}
\begin{proof}
 The proof follows from \eqref{bigdandel} of Remark \ref{REM-H2} and Proposition \ref{EST1}.
 In fact for $N \in I_d$ the embedding $\h^1\subset \mathbb{L}^{4N+2}$ is continuous and thanks to \eqref{bigdandel} we derive that
  \begin{align*}
   \me \biggl[\int_0^T \lve \Delta \bds\rve^2 ds\biggr]^q \le C \me \biggl[\int_0^T \lve \Delta \bds -f(\bds) \rve^2 ds\biggr]^q
   + C \me \sup_{s\in [0,T]} \lve \nabla \bds\rve^{2q(2N+1)}+C.
  \end{align*}
Now, we easily conclude the proof of our claim by using Proposition \ref{EST1} with $p=q(2N+1)$.
\end{proof}
\subsection{Proof of Proposition \ref{STRONGER-NORM}}
In this subsection we give the proof of  Proposition \ref{STRONGER-NORM}, \textit{i.e.},  we will derive the important estimate for  $\me \sup_{s\in [0,t\wedge \tau_k]}\lve \Delta \d(s)-f(\d(s))\lve^2 $
\begin{proof}[Proof of Proposition \ref{STRONGER-NORM}]
  Since the local strong solution $(\v,\d)\in D(\rA)\times
 \bx_{1}$ almost surely we can view $\Psi(\bd)$ as
 $$\Psi(\bd)= \lve -\rA_1 \d-
 f(\d)\rve^2.$$
 The functional $\Psi(\d)$ is twice differentiable with first and second Fr\'echet derivatives defined by
$$\Psi^{\prime}(\d)(h)=2 \langle -\rA_1 \d-f(\d), -\rA_1
 h-f^\prime(\d)h\rangle,$$ and $$ \pp(\d)[h,k]=2\langle -\rA_1 k -\fp(\d)k, -\rA_1 h -\fp(\d) h\rangle-2\langle -\rA_1 \d- f(\d), \fpp(\d) h \,\,k
 \rangle,
 $$ for any $h,\, k\in \bx_{\frac 12}$.

 Let us recall that the process $(\bv(t,\d(t)))$, $t<\tilde{\tau}_\infty$ solves
 \begin{eqnarray*}
d\bv(t)+\biggl(\Delta\bv(t)+B(\bv(t), \bv(t))+M(\bd(t))\biggr)dt=S(\bv(t))dW_1(t),\\
d\bd(t)+\biggl(\rA_1\bd(t)+ \tilde{B}(\bv(t),\bd(t))+ f(\bd(t))-\frac
12 G^2(\bd(t))\biggr)dt=G(\bd(t))dW_2(t).
\end{eqnarray*}
 Since $(\v,\d)\in X_t$ for
any $t\in [0,\tau_\infty)$ we infer that
\begin{equation*}
\Delta\bv(\cdot)+B(\bv(\cdot), \bv(\cdot))+M(\bd(\cdot)) \in L^2(0,t;\h),
\end{equation*}
and
\begin{equation*}
\rA_1\bd(\cdot)+ \tilde{B}(\bv(\cdot),\bd(\cdot))+ f(\bd(\cdot))-\frac 12 G^2(\bd(\cdot))
\in L^2(0,t;\bx_{0}),
\end{equation*}
for any $t\in [0,\tilde{\tau}_\infty)$. Therefore by considering the
Gelfand triples $D(\rA)\subset D(\rA^\frac12)\subset\h$ and $\bx_{1}\subset \bx_{\frac 12}\subset \bx_{0}$ It\^o's formula for the
functional $\lve \nabla \v (t\wedge \tau_k)\rve^2$ and
$\Psi(\d(t\wedge \tau_k))$ for any integer $k>0$ are applicable to
our situation (see \cite[Theorem I.3.3.2, Page 147]{Pardoux}).
 By It\^o's formula we have
\begin{equation*}
\begin{split}
\lve \nabla \bv(t\wedge \tau_k) \rve^2-\lve \nabla \v_0\rve^2=-2\int_0^{t\wedge \tau_k}
\langle \Delta\v(s) +B(\v(s),\v(s))+\Pi\nabla \cdot(\nabla \d(s)\odot
\nabla \d(s)),\Delta \v(s) \rangle ds\\+ \int_0^{t\wedge \tau_k} \lve  S(\v(s))\rve^2_{\mathcal{T}_2(\rK_1,\ve)}
ds+2\int_0^{t\wedge \tau_k} \langle \rA^\frac 12 S(\v(s)),
\rA^\frac 12 \v(s)\rangle d\W(s),
\end{split}
\end{equation*}
where $t\wedge \tau_k:= t \wedge \tau_k.$
As in Lin and Liu \cite{Lin-Liu} we use the identity
\begin{equation*}
\begin{split}
\langle \Pi \nabla \cdot(\nabla \d\odot \nabla \d),\Delta \v
\rangle=\langle \nabla \d \Delta \d , \Delta \v\rangle+\langle
\nabla\left(\frac{ \lvert \nabla \d\rvert^2}{2}\right), \Delta \v\rangle,\\
=\langle \nabla \d \Delta \d, \Delta \v\rangle.
\end{split}
\end{equation*}
To get the second line of this identity we have used the fact that
$\Delta \v$ is divergence free. Now it is easy to see that
\begin{equation*}
 \langle \Pi \nabla \cdot(\nabla \d\odot \nabla \d),\Delta \v
\rangle=\langle \nabla \d (\Delta \d-f(\d)), \Delta \v\rangle +\langle
\nabla \d, f(\d) \Delta \v\rangle.
\end{equation*}
 As a consequence of the last identity we have
that
\begin{equation}\label{ST1}
\begin{split}
\lve \nabla \bv(t\wedge \tau_k)\rve^2-\lve \nabla \v_0\rve^2=-2\int_0^{t\wedge \tau_k} \langle
\Delta \v(s) +B(\v(s),\v(s))+\nabla \d(s) f(\d(s)), \Delta \v(s)\rangle ds\\
-2\int_0^{t\wedge \tau_k} \langle \nabla \d(s)
(\Delta \d(s)-f(\d(s))), \Delta \v(s)\rangle ds
+\int_0^{t\wedge \tau_k}
\lve S(\v(s))\rve^2_{\mathcal{T}_2(\rK_1,\ve)} ds\\+2\int_0^{t\wedge \tau_k} \langle
\nabla S(\v(s)), \nabla \v(s)\rangle d\W(s).
\end{split}
\end{equation}
 At the same time we have
\begin{equation*}
\begin{split}
\Psi(\d(t\wedge \tau_k))-\Psi(\d_0)=\int_0^{t\wedge \tau_k} \ps(\d(s))[-\v(s)\cdot \nabla \d(s)-\rA_1
\d(s)-f(\d(s))+\frac12 [G^2(\d(s))]]ds\\+\frac
12\int_0^{t\wedge \tau_k} \pp(\d(s))[G(\d(s))G(\d(s))] ds
+\int_0^{t\wedge \tau_k}
\ps(\d(s))[G(\d(s))]dW_2(s).
\end{split}
\end{equation*}
Using the definition of $\ps(\d(s))$ we see from this last equation
that
\begin{equation}\label{ST2}
\begin{split}
\Psi(\d(t\wedge \tau_k))-\Psi(\d_0)=2\int_0^{t\wedge \tau_k}\langle  -\rA_1 \d(s)-f(\d(s)),-\rA_1
\Big(-\v(s)\cdot \nabla \d(s)-\rA_1 \d(s)-f(\d(s))\Big)\rangle ds
\\-2\int_0^{t\wedge \tau_k} \langle-\rA_1
\d(s)-f(\d(s)),\fp(\d(s))\Big(-\v(s)\cdot \nabla \d(s)-\rA_1 \d(s)-f(\d(s))\Big)  \rangle
ds\\
+\frac 12
\int_0^{t\wedge \tau_k} \ps(\d(s))[G^2(\d(s))] ds +\frac 12\int_0^{t\wedge \tau_k}
\pp(\d(s))[G(\d(s)),G(\d(s))] ds\\
+\int_0^{t\wedge \tau_k} \ps(\d(s))[G(\d(s))]dW_2(s)\\
=2\sum_{i=1}^4 \int_0^{t\wedge \tau_k} I_i(s)ds+\int_0^{t\wedge \tau_k}
\ps(\d(s))[G(\d(s))]dW_2(s).
\end{split}
\end{equation}
For $(\v,\d)\in D(\rA)\times
 \bx_{1}$ almost surely we easily check that $-\v\cdot\nabla \d \in D(\rA_1)$. Therefore  $I_1$
 can be rewritten in the following form
\begin{align}
I_1(s)=\langle-\nabla \d  \Delta \v -\v\cdot \nabla \Delta \d,
\Delta\d-f(\d) \rangle+\langle  \Delta \d-f(\d), -\rA_1(\Delta
\d-f(\d))\rangle\nonumber,
\end{align}
where the first term is a product of two function in $\el^2$ and
the second term is understood as the duality pairing between an
element of $\h^1$ and $\left(\h^1\right)^\ast$. Invoking
\eqref{ST0} we derive that
\begin{equation}
\begin{split}
 I_1(s)=-\langle \v\cdot \nabla [f(\d)], \Delta
\d-f(\d)\rangle -\lve \nabla(\Delta \d-f(\d))\rve^2\\
+\langle-\nabla \d \Delta \v  -\v\cdot \nabla (\Delta \d-f(\d)),
\Delta\d-f(\d) \rangle.\label{ST3}
\end{split}
\end{equation}
We derive from  \eqref{ST1},  \eqref{ST2} and
\eqref{ST3} that
\begin{equation}\label{ST4}
\begin{split}
\lve \nabla \v(t\wedge \tau_k)\rve^2+\Psi(\d(t\wedge \tau_k))-\lve \nabla \v_0\rve^2-\Psi(\d_0)+2\int_0^{t\wedge \tau_k}\Big[\lve \Delta \v(s)\rve^2+\lve
\nabla(\Delta \d(s)-f(\d(s)))\rve^2\Big]ds\\=2 \int_0^{t\wedge \tau_k}\langle
B(\v(s),\v(s))+\nabla \d(s) f(\d(s)),\rA\v(s)\rangle
ds+2\int_0^{t\wedge \tau_k}\langle\nabla \d(s)
[\Delta \d(s) -f(\d(s))], \rA \v(s)\rangle ds\\
-2\int_0^{t\wedge \tau_k} \langle\nabla \d(s)  \Delta \v(s) , \Delta \d(s)-f(\d(s))\rangle
ds-2\int_0^{t\wedge \tau_k} \langle \v(s)\cdot \nabla [f(\d(s))], \Delta \d(s)
-f(\d(s))\rangle ds\\  -2\int_0^{t\wedge \tau_k} \langle\fp(\d(s))\Big(-\v(s)\cdot
\nabla \d(s)+\Delta \d(s)-f(\d(s))\Big), \Delta \d(s)-f(\d(s)) \rangle ds\\
-2\int_0^{t\wedge \tau_k} \langle \v(s)\cdot\nabla[\Delta \d(s)-f(\d(s))], \Delta \d(s)
-f(\d(s))\rangle ds + OT.
\end{split}
\end{equation}
where the other term $OT$ is defined by
\begin{equation*}
\begin{split}
OT=\int_0^{t\wedge \tau_k} \lve  S(\v(s))\rve^2_{\mathcal{T}_2(\rK_1,\ve)} ds +\frac 12\int_0^{t\wedge \tau_k}
\pp(\d(s))[G(\d(s)),G(\d(s))] ds\\
+2\int_0^{t\wedge \tau_k} \langle \nabla S(\v(s)), \nabla
\v(s)\rangle d\W(s)+\int_0^{t\wedge \tau_k} \ps(\d(s))[G(\d(s))]dW_2(s)\\
+\frac
12 \int_0^{t\wedge \tau_k} \ps(\d(s))[G^2(\d(s))] ds.
\end{split}
\end{equation*}
Thanks to \eqref{B3} the identity
\eqref{ST4} can be simplified as follows
\begin{equation}\label{MITOV-VE}
\begin{split}
\lve \nabla \v(t\wedge \tau_k)\rve^2+\Psi(\d(t\wedge \tau_k))-\lve \nabla \v_0\rve^2-\Psi(\d_0)+2\int_0^{t\wedge \tau_k}\Big[\lve \Delta \v(s)\rve^2+\lve
\nabla(\Delta \d(s)-f(\d(s)))\rve^2\Big]ds\\=2\int_0^{t\wedge \tau_k}\langle\nabla \d(s)
[\Delta \d(s) -f(\d(s))], \rA \v(s)\rangle ds
-2\int_0^{t\wedge \tau_k} \langle\nabla \d(s)  \Delta \v(s) , \Delta \d(s)-f(\d(s))\rangle
ds\\
-2\int_0^{t\wedge \tau_k}
\langle\fp(\d(s))\Big(-\v(s)\cdot \nabla \d(s)+\Delta \d(s)-f(\d(s))\Big), \Delta
\d(s)-f(\d(s)) \rangle ds\\
+2 \int_0^{t\wedge \tau_k}\langle
B(\v(s),\v(s))+\nabla \d(s) f(\d(s)),\rA \v(s)\rangle ds \\
-2\int_0^{t\wedge \tau_k} \langle \v(s)\cdot \nabla
[f(\d(s))], \Delta \d(s) -f(\d(s))\rangle ds+ OT.
\end{split}
\end{equation}
Now note that for $j=1,2$ we have
\begin{align*}
 [\v\cdot \nabla [f(\d)]]_j=& \v_i \cdot \frac{\partial [f(\d)]_j}{\partial x_i}\\
 =& \tilde{f}(\vert \d\vert^2)\v_i \frac{\partial \d_j}{\partial x_i}+2 \tilde{f}^\prime(\vert \d\rvert^2)\d_j \v_i \frac{\partial \d_k}{\partial x_i} \d_k\\
 =& [\tilde{f}(\vert \d\vert^2)\v\cdot \nabla \d]_j+ 2 \tilde{f}^\prime(\vert\d\vert^2)\left([\v\cdot\nabla \d]\cdot \d\right)\,\d_j ,
\end{align*}
where above the summation over repeated indexes are enforced.
On the other hand we have
\begin{align*}
 f^\prime(\d)[\v\cdot \nabla \d]=\tilde{f}(\vert\d\vert^2) \v\cdot \nabla \d+2\tilde{f}^\prime(\vert \d\vert^2)\d \left([\v\cdot\nabla \d]\cdot \d\right).
\end{align*}
Thus we just proved that
\begin{equation*}
 \v\cdot\nabla [f(\d)]=f^\prime(\d)[\v\cdot \nabla \d].
\end{equation*}
From this last identity and \eqref{MITOV-VE} \del{and the chain rule $\v\cdot \nabla
[f(\d)]=\fp(\d) \v\cdot \nabla \d $} we deduce that
\begin{equation*}
\begin{split}
\lve \nabla \v(t\wedge \tau_k)\rve^2+\Psi(\d(t\wedge \tau_k))-\lve \nabla \v_0\rve^2-\Psi(\d_0)+2\int_0^{t\wedge \tau_k}\Big[\lve \Delta \v(s)\rve^2+\lve
\nabla(\Delta \d(s)-f(\d(s)))\rve^2\Big]ds\\ =2\int_0^{t\wedge \tau_k}\langle\nabla \d(s)
[\Delta \d(s) -f(\d(s))], \rA \v(s)\rangle ds
-2\int_0^{t\wedge \tau_k} \langle\nabla \d(s)  \Delta \v(s) , \Delta \d(s)-f(\d(s))\rangle
ds\\
-2\int_0^{t\wedge \tau_k}
\langle\fp(\d(s))\Big(\Delta \d(s)-f(\d(s))\Big), \Delta \d(s)-f(\d(s)) \rangle
ds \\+2 \int_0^{t\wedge \tau_k}\langle
B(\v(s),\v(s))+\nabla \d(s) f(\d(s)),\rA \v(s)\rangle ds+ OT.
\end{split}
\end{equation*}
But since $\nabla \d f(\d)=\nabla F(\d)$ and $\rA \v$ is a divergence
free function the last identity becomes
\begin{equation}\label{DEAL}
\begin{split}
\lve \nabla \v(t\wedge \tau_k)\rve^2+\Psi(\d(t\wedge \tau_k))-\lve \nabla \v_0\rve^2-\Psi(\d_0)+2\int_0^{t\wedge \tau_k}\Big[\lve \Delta \v(s)\rve^2+\lve
\nabla(\Delta \d(s)-f(\d(s)))\rve^2\Big]ds\\=2\int_0^{t\wedge \tau_k}\langle\nabla \d(s)
[\Delta \d(s) -f(\d(s))], \rA \v(s)\rangle ds
-2\int_0^{t\wedge \tau_k} \langle\nabla \d(s)  Delta \v(s) , \Delta \d(s)-f(\d(s))\rangle
ds\\
-2\int_0^{t\wedge \tau_k}
\langle\fp(\d(s))\Big(\Delta \d(s)-f(\d(s))\Big), \Delta \d(s)-f(\d(s)) \rangle
ds \\
+2 \int_0^{t\wedge \tau_k}\langle
B(\v(s),\v(s)),\rA \v(s)\rangle ds + OT.
\end{split}
\end{equation}
To get rid of some bad terms in the right-hand side of the above identity we will apply the It\^o's formula to
the function
\begin{equation}\label{SCHM-trick}
\begin{split}
\Xi(t\wedge \tau_k,\v,\d):=&\Phi(t\wedge \tau_k)
\left(\lve \nabla \v(t\wedge \tau_k)\rve^2+\Psi(\d(t\wedge \tau_k))\right)\\
=&
 e^{-\int_0^{t\wedge \tau_k} \phi(s)
ds}\left(\lve \nabla \v(t\wedge \tau_k)\rve^2+\Psi(\d(t\wedge \tau_k))\right).
\end{split}
\end{equation}
Before implementing this idea, we shall make few remarks. Since, by assumption, $h \in \h^2$ it is not
difficult to check that
$$G^2(\d)=(\d\times h)\times h \in D(A_2), $$ and
$$ G(\d)=\d\times h\in D(A_2).$$ Next, observe that
\begin{align*}
\frac 12 \ps(\d)[G^2(\d)]=&\langle \Delta \d-f(\d), \Delta\d\times
h\rangle +\langle\Delta \d-f(\d),\d\times \Delta h+(\d\times
h)\times \Delta h\rangle\nonumber \\ &\quad -2\langle \Delta
\d-f(\d), \fp(\d) (\d\times h)\times
h\rangle,\\
\frac 1 2 \pp(\d)[G(\d), G(\d)]=&\lve \Delta
G(\d)-\fp(\d)G(\d)\rve^2-\langle \Delta \d-f(\d), \fpp(\d) [G(\d),
G(\d)]\rangle.
\end{align*}
Thus, an application of the It\^o formula
to the stochastic processes $\Phi(t)\Big(\lve
\nabla \v(t) \rve^2+\Psi(\d(t))\Big)$ and the estimates \eqref{ST10-b}, \eqref{ST12}, \eqref{ST6-B}, \eqref{ST120} and \eqref{MOST-DIFF} yields
\begin{equation}\label{ST15-a}
\begin{split}
& \E \Phi(t\wedge \tau_k)\left(\lve \nabla \v(t\wedge \tau_k)\rve^2+\Psi(\d(t\wedge \tau_k))\right)\\
 & \qquad \qquad + \E
\int_0^{t\wedge \tau_k} \Phi(s)\Big(2\lve \Delta \v(s)
\rve^2+ (1-\delta_7)\lve \nabla (\Delta \d(s) -f(\d(s)) )\rve^2 \Big)ds\\
& \qquad \qquad \qquad \le  E_1+E_2 + \lve\nabla \v_0\rve^2+\Psi(\d_0)+\E \int_0^{t\wedge \tau_k} \Phi(s) \lve
S(\v(s))\rve^2_{\mathcal{T}_2(\rK_1,\ve)} ds\\
& \qquad \qquad \qquad \qquad + C(h)\E\int_0^{t\wedge \tau_k}
\Phi(s)\lve \Delta \d(s) -f(\d(s))\rve^2 ds
\\
& \qquad \qquad \qquad \qquad +C\E\int_0^{t\wedge \tau_k} \Phi(s)(C+C(h)\lve \d(s)\rve_{\h^1}^{4N+2}) ds,
\end{split}
\end{equation}
where the terms $E_1$ and $E_2$ are defined in \eqref{DEF-E_1} and \eqref{ST14} below.
 With the help of Burkholder-Davis-Gundy's inequality we have
 \begin{equation}\label{DEF-E_1}
E_1:=2\E \sup_{0\le s\le t\wedge \tau_k}\biggl\lvert\int_0^{s} \Phi(r)
\ps(\d(r))[G(\d(r))]dW_2(r) \biggr\rvert \le C \E
\biggl(\int_0^{t\wedge \tau_k} [\Phi(s)]^2 \lvert \ps(\d(s))[G(\d(s))]\rvert^2 ds
\biggr)^\frac 12.
\end{equation}
Now invoking the definition of $\ps(\d)[ G(\d)]$, the Cauchy-Schwarz and Young
inequalities we have
\begin{equation*}
\begin{split}
& 2\E \sup_{0\le s\le t\wedge \tau_k}\biggl\lvert\int_0^{s} \Phi(r)
\ps(\d(r))[G(\d(r))]dW_2(r) \biggr\rvert \\ & \le \E\biggl[\biggl(\int_0^{s} \Phi(r) \lve \Delta\Big((\d(r)\times h) \times
h\Big)-\fp(\d(r))\Big((\d(r)\times h)\times h\Big) \rve^2 dr
\biggr)^\frac 12\\
&\qquad \qquad \qquad \times \sup_{0\le s \le
t\wedge \tau_k}\sqrt{\Phi(s)} \lve \Delta \d(s) -f(\d(s))\rve
 \biggr]\\
& \le C
\E \int_0^{t\wedge \tau_k} \Phi(s) J_1(s) ds
+\frac12 \E \sup_{0\le s\le t} \Phi(s) \lve \Delta
\d(s)-f(\d(s))\rve^2,
\end{split}
 \end{equation*}
 where the process $J_1(s)$, $s<\tilde{\tau}_\infty$ is defined as in \eqref{J_1}.
  Thanks to this last inequality and Eqs.
 \eqref{ST11-b} it is easy to deduce that
\begin{equation}\label{ST13}
\begin{split}
2\E \sup_{0\le s\le t}\biggl\lvert\int_0^{t\wedge \tau_k} \Phi(s)
\ps(\d(s))[G(\d(s))]dW_2(s) \biggr\rvert \le C(h) \E
\int_0^{t\wedge \tau_k} \Phi(s) \lve \Delta \d(s) -f(\d(s))\rve^2 ds \\+
C(h) t\wedge \tau_k \E\sup_{s[0,{t\wedge \tau_k}]} \Phi(s) \lve \d(s)
\rve^{4N+2}_{\h^1} \\ +\frac12 \E \sup_{0\le s\le t\wedge \tau_k} \Phi(s) \lve
\Delta \d(s)-f(\d(s))\rve^2.
\end{split}
\end{equation}
With similar idea we can prove that
\begin{equation}\label{ST14}
\begin{split}
E_2:=2 \E \sup_{0\le s\le t\wedge \tau_k}\biggl\lvert\int_0^{s} \Phi(r) \langle
\nabla S(\v(r)), \nabla \v(r)\rangle d\W(r)\biggr\rvert\le \frac12 \E \int_0^{t\wedge \tau_k}
\Phi(s) \lve  S(\v(s))\rve^2_{\mathcal{T}_2(\rK_1,\ve)} ds\\
+\frac12 \E
\sup_{0\le s\le t\wedge \tau_k} \Phi(s)\rve \nabla \v(s) \rve^2 .
\end{split}
\end{equation}
Plugging
\eqref{ST13} and  \eqref{ST14} in \eqref{ST15-a}, using the hypotheses on $S$ in
Assumption \ref{HYPO-ST} and invoking Proposition \ref{EST1} with $p=2N+1$ we infer that there exists a constant $\tilde{C}>0$ such that
\begin{equation*}
\begin{split}
& \E \biggl[e^{-\int_0^ {t\wedge \tau_k} \phi(r) dr}\left(\lve \nabla
\v(t\wedge \tau_k)\rve^2+\Psi(\d(t\wedge \tau_k))\right) \biggr] \\
& \qquad \qquad + \E \int_0^{t\wedge \tau_k}
\Phi(s)\Big(2\lve \Delta \v(s)
\rve^2+ \lve \nabla (\Delta \d(s) -f(\d(s)) )\rve^2 \Big)ds\\
& \qquad \qquad \qquad \le \lve\nabla \v_0\rve^2+\Psi(\d_0)+ \tilde{C} \E\int_0^{t\wedge \tau_k} \Phi(s)
\Big(\lve \nabla \v(s)\rve^2+ \lve \Delta \d(s) -f(\d(s))\rve^2
\Big)ds\\
& \qquad \qquad \qquad \quad +\tilde{C} (\mathbb{E}\mathfrak{G}_1(t,2N+1)+1),
\end{split}
\end{equation*}
which altogether with the  Gronwall inequality implies the desired estimate \eqref{ST15}.
\end{proof}

\subsection*{Acknowledgments}
Razafimandimby's research was partially supported by the FWF-Austrian Science through the Stand-Alone project P28010.
The research on this paper was initiated during the visit of Razafimandimby to the University of York in October 2012. Part of this work was also carried out when he visited the University of York
 during October 2013, February 2014.
He would like to thank the Mathematics Department at York for hospitality. Z. Brze{\' z}niak presented a
lecture on the subject of this paper at the RIMS Symposium on Mathematical Analysis of Incompressible Flow held at Kyoto in  February 2013.
He would like to thank Professor Toshiaki Hishida for the kind invitation. P. Razafimandimby is very grateful to the organizers of the conference `` Nonlinear PDEs in Micromagnetism: Analysis, Numerics and Applications'', which was held in ICMS Edinburgh, UK, for their invitation to present a talk at this meeting. He is also very grateful for the financial support he received from the International Centre for Mathematical Sciences (ICMS) Edinburgh.

\end{document}